\documentclass[a4paper, 12pt]{article}
\usepackage{amsfonts}
\usepackage{amsmath}
\usepackage{amsthm}
\usepackage{amssymb}
\usepackage{color}
\usepackage{setspace} 
 \usepackage{hyperref}
\newtheorem{theo}{Theorem}[section]
\newtheorem{prop}[theo]{Proposition}
\newtheorem{lemma}[theo]{Lemma}

\newtheorem{coro}[theo]{Corollary}
\theoremstyle{definition}\newtheorem{de}{Definition}[section]
\theoremstyle{remark}\newtheorem{rem}{Remark}[section]
\def\D{\mathbb{D}}
\def\R{\mathbb{R}}
\def\N{\mathbb{N}}
\def\div{\mathrm{div}\/}
\def\id{\mathrm{Id}\/}
\def\tr{\mathrm{Tr}\/}
\def\simple{\mathrm{Simple}\/}
\def\pO{\partial\Omega}
\def\oOm{{\tilde{\Omega}}}
\def\eps{\varepsilon}
\def\lb{\lambda}
\newcommand{\n}{{ n}}
\newcommand{\br}{{\bar{r}_0}}
\newcommand{\bet}{{\bar{\eta}_0}}
\newcommand{\x}{x}
\newcommand{\y}{y}
\newcommand{\z}{z}
\newcommand{\uu}{{{u}}}
\newcommand{\E}{{{E}}}

\title{{Generic properties of the spectrum of the Stokes system with Dirichlet boundary condition in $\R^3$}}
\author{
Y. \textsc{Chitour} \footnote{L2S, Universit\'e Paris-Sud XI, CNRS,
Sup\'elec, 3 Rue Joliot-Curie, 91192 Gif-sur-Yvette, France, email:
yacine.chitour@lss.supelec.fr}\and
D. \textsc{Kateb} \footnote{Centre de
Recherche de Royallieu, LMAC, 60020 Compi\`egne, France, email:
dkateb@dma.utc.fr} \and
R. \textsc{Long} \footnote{Department of Chemistry, Princeton University, NJ 08544, USA,
email: rlong@princeton.edu}}

\begin{document}

\maketitle
{\bf Abstract.}  Let $(SD_\Omega)$ be the Stokes operator defined in a bounded domain $\Omega$ of $\R^3$ with Dirichlet boundary conditions. We prove that, generically with respect to the domain $\Omega$ with $C^5$ boundary,
the spectrum of $(SD_\Omega)$ satisfies a non resonant property introduced by C. Foias and J. C. Saut in \cite{Foias:1987dq} to linearize the Navier-Stokes system in a bounded domain $\Omega$ of $\R^3$ with Dirichlet boundary conditions. For that purpose, we first prove that, generically with respect to the domain $\Omega$ with $C^5$ boundary, all the eigenvalues of $(SD_\Omega)$ are simple. That answers positively a question raised by J. H. Ortega and E. Zuazua in \cite[Section 6]{Zuazua1}. The proofs of these results follow a standard strategy based on a contradiction argument requiring shape differentiation. One needs to shape differentiate at least twice the initial problem in the direction of carefully chosen domain variations. The main step of the contradiction argument amounts to study the evaluation of Dirichlet-to-Neumann operators associated to these domain variations.

\tableofcontents

\section{Introduction and main results}
In this paper, we consider the eigenvalue problem for the Stokes system with Dirichlet boundary conditions defined in 
a bounded open subset $\Omega$ of $\R^3$ with $C^\ell$ boundary, $\ell\geq 4$,
\begin{equation*}\label{eq:stokes}
(SD_\Omega)\qquad\left\{\begin{array}{cccc}
\displaystyle -\Delta \phi +\nabla p&=&\lambda \phi&\textrm{ in }\Omega,\\
\displaystyle\div \phi&=&0&\textrm{ in }\Omega,\\
\displaystyle\phi&=&0&\textrm{ on }\pO,\\
\displaystyle\int_{\Omega}p(x)dx&=&0.
\end{array}\right.
\end{equation*}
Here we use $\phi\in\R^3$ and $p\in\R$ to denote respectively the velocity field and the pressure. 
%We denote by $\n$ the outward unit normal 
%to $\partial \Omega$. The corresponding co-normal derivative $\frac{\partial \phi}{\partial N}$ on $\partial \Omega$ is defined to be 
%\begin{equation}
%\frac{\partial \phi}{\partial N}=(\nabla \phi +\nabla \phi^T)\n -p\n ~~\text{on}~~\partial \Omega.
%\end{equation}
%Here, the subscript $^T$ denotes the transpose of a matrix. In the sequel, we will denote 
%\begin{equation}
%\frac{\partial \phi}{\partial \textbf{\n}}= \nabla \phi \n;
%\end{equation}
%and   
%\begin{equation}
%\frac{\partial \phi}{\partial \textbf{N}}= (\nabla \phi +\nabla \phi^T)\n.
%\end{equation}
It is well-known that $(SD_\Omega)$ admits an increasing sequence of positive eigenvalues $(\lambda_n)_{n\geq 1}$ tending to infinity as $n$ goes to infinity. 

The purpose of this paper is to prove \emph{genericity} results on the spectrum of $(SD_\Omega)$ with respect to domains of $\R^3$. We start by clarifying the notion of genericity considered below. Recall that the set of bounded domains of $\R^3$ with $C^\ell$ boundary denoted by $\D^3_\ell$ can be endowed with the following topology: the base of open neighborhoods is (essentially) given by the sets $V(\Omega,\varepsilon)$
defined, for any domain $\Omega\in\D^3_\ell$ and
$\varepsilon>0$ small enough, as the images of $\Omega$ by $\id+u$, with
$u\in W^{\ell+1,\infty}(\Omega,\R^3)$ and
$\|u\|_{W^{\ell+1,\infty}}<\varepsilon$ (cf. \cite{DH} and \cite{Simon}). Then $\varepsilon$ is chosen so that 
$\id+u:\Omega\rightarrow (\textrm{Id}+u)(\Omega)$ is a diffeomorphism. As shown by A. M. Michelleti in \cite{Micheletti:1972nx} (see also \cite[Appendix 2]{DH}), a neighborhood $V(\Omega,\varepsilon)$ of $\Omega\in\D^3_\ell$ as defined previously is metrizable using a Courant-type distance, denoted by $d_{\ell+1}$, and each $(V(\Omega,\varepsilon),d_{\ell+1})$ is complete and separable. For any domain $\Omega\in \D^3_\ell$, we use $\D^3_\ell(\Omega)$ to denote the Banach manifold obtained as the set of images $(\textrm{Id}+u)(\Omega)$  by $u\in W^{\ell+1,\infty}(\Omega,\R^3)$, which are diffeomorphic to $\Omega$. A property $(P)$ will thus be referred to as ``being generic with respect to $\Omega\in\D^3_\ell$'' if, for every  $\Omega\in\D^3_\ell$, the set of of domains of $\D^3_\ell(\Omega)$ where $(P)$ holds true contains a countable intersection of open and dense subsets of $\D^3_\ell(\Omega)$.\medskip

The main contribution of this paper consists in proving a conjecture formulated by C. Foias and J. C. Saut in \cite{Foias:1987dq} on the generic \emph{non-resonant} character of the spectrum of the Stokes operator with Dirichlet boundary conditions. Note that the non-resonance condition plays a crucial role in linearizing the Navier-Stokes systems to obtain a normal form, and then computing a useful asymptotic expansion for its solutions (cf. \cite{Foias:1987dq}, \cite{Foias-Ziane-06} and \cite{Foias-Ziane-09}). This situation is analogous to the classical Poincar\'e's normal form theorem for ordinary differential equations (cf. \cite[Chapter 5]{Arnold:1988fk}) although the proof is more involved. See also \cite{Foias:2011uq} for a recent development in the spirit of \cite{Foias:1987dq}. As noticed in \cite{Foias:1987dq}, non-resonance does not occur for periodic boundary conditions. However, the authors conjectured that non-resonance should be generic for Dirichlet boundary conditions. In this paper, we confirm that conjecture in Theorem \ref{main-theo-2}. Let us first recall the following definition (cf. \cite[Definition 1]{Foias:1987dq}). 

\begin{de}\label{def:reso}
We call {\it resonance} in the spectrum of $(SD_\Omega)$ a relation of the type
\begin{equation}\label{eq:reso-1}
\lambda_{k+1}=\sum_{j=1}^km_j\lambda_j, \hbox{ where }m_j\in\N, \qquad 1\leq j\leq k.
\end{equation}
If no resonance occurs in the spectrum of $(SD_\Omega)$, then $(SD_\Omega)$ will be called \emph{non-resonant}.
\end{de}

%The concept of resonance was introduced by C. Foias and J. C. Saut in \cite{Foias:1987dq} in order to linearize the Navier-Stokes system and obtain a normal form for it as well as a useful asymptotic expansion for its solutions in case when the corresponding Stokes operator  $(SD_\Omega)$ is nonresonant. As noticed in \cite{Foias:1987dq}, nonresonance does not occur for periodic boundary conditions. However the authors conjectured that nonresonance should be generic for Dirichlet boundary conditions. In this paper, we confirm that conjecture. 

\begin{theo}\label{main-theo-2}
Generically with respect to $\Omega\in\D^3_5$, the spectrum of the operator $(SD_\Omega)$ is non-resonant.
\end{theo}

In order to prove Theorem \ref{main-theo-2}, we need first to establish another genericity result on the spectrum of the Stokes operator known as \emph{generic simplicity}.
 
%Consider the following property, referred as $(\simple)$, regarding the spectrum of $(SD_\Omega)$
%$$(\simple)~ \textrm{  All the eigenvalues of }(SD_\Omega) \textrm{ are simple.}$$
%We note that $(\simple)$ is not always true, for instance if $\Omega$ is a disk (cf. \cite[Chapter 4, pages 49-50]{DH}). 
%
%
%
%One of the goals of the present paper consists in proving the following theorem.
\begin{theo}\label{main-theo}
Generically with respect to $\Omega\in\D^3_4$, all the eigenvalues of $(SD_\Omega)$ are simple. 
\end{theo}

\begin{rem}\label{rem:zuazua}
In \cite{Zuazua1},  several properties for the Stokes system with Dirichlet boundary conditions (in particular the simplicity of the spectrum) were proved to be generic for domains in $\R^2$. Moreover, in the same paper, the three dimensional case was considered in Section $6$, pointing out why techniques developed in \cite{Zuazua1} could only handle the two dimensional case. In this regard, Theorem \ref{main-theo} answers positively the open question of Section $6$ in \cite{Zuazua1}.
\end{rem}

Theorem \ref{main-theo} is of course a particular case of Theorem \ref{main-theo-2}, but will allow us to work only with simple eigenvalues in the main step of Theorem \ref{main-theo-2}'s proof. This reduction is essential in our arguments. We now describe the strategy of the proofs. As it is standard since \cite{Alb}, the reasoning goes by contradiction and is based on shape differentiation. \smallskip

We start with a description of the proof of Theorem \ref{main-theo}. Fix a domain $\Omega_0\in\D^3_\ell$. For every integer $k$, we define $A_k$ as the (open) subset of $\D^3_\ell(\Omega_0)$ whose 
elements $\Omega$ verify that the $k$ first eigenvalues of $(SD_{\Omega})$
are simple. Clearly, by Baire's lemma, proving Theorem \ref{main-theo} amounts to show that 
$A_{k+1}$ is dense in $A_k$ for every $k\geq 0$. We argue by contradiction and assume that there 
exists an integer $k$, a domain $\Omega$ with $C^\ell$ boundary in $A_k$ and $\varepsilon>0$, %the domain $(Id_3+u)\Omega$ or simply $\Omega+u$ 
%an open neighborhood
%$V(\Omega,\varepsilon)$
such that, for every $u\in W^{\ell+1,\infty}(\Omega,\R^3)$ with
$\|u\|_{W^{\ell+1,\infty}}<\varepsilon$, the domain $(\id+u)\Omega$, or simply $\Omega+u$,
belongs to  $A_k$ but not to  $A_{k+1}$. Let $m\geq 2$ be the multiplicity of $\lambda$, the value of the $(k+1)$-th eigenvalue of $(SD_{\Omega})$ and $\phi_i$, $i=1,\cdots,m$, orthonormal eigenfunctions associated to $\lambda$. Finally, let $n$ be the outward unit normal vector field on $\partial\Omega$. By computing the shape derivative of the $(n+1)$-th eigenvalue of $(SD_{\Omega})$, J. H. Ortega and E. Zuazua obtained in \cite{Zuazua1} that, at every $x\in\partial\Omega$, one has, for $i,j=1,\cdots,m$, and $i\neq j$,
\begin{equation}\label{intro1}
\frac{\partial \phi_i}{\partial n}\cdot n=0,\qquad 
\frac{\partial \phi_i}{\partial n}\cdot\frac{\partial \phi_j}{\partial n}=0,\qquad \Vert \frac{\partial \phi_i}{\partial n}\Vert^2=\Vert \frac{\partial \phi_j}{\partial n}\Vert^2.
\end{equation}
If $m>2$, then there necessarily exists $1\leq i\leq m$ so that $\displaystyle\frac{\partial \phi_i}{\partial n}\equiv 0$ on $\partial\Omega$ and one reaches a contradiction using 
a unique continuation result due to Osses (cf. \cite{Osses}). {However, in order to obtain generic simplicity ($m=1$), it was not clear how to pursue the reasoning by contradiction, i.e., showing that relations in (\ref{intro1}) do not hold true generically with respect to the domains of $\R^3$ if $m=2$. Note that, for questions involving scalar PDEs, if one wants to prove generic simplicity of the spectrum of a self-adjoint operator with Dirichlet boundary conditions, then it is standard to follow the lines of the above mentioned contradiction argument and to reach Eq.  (\ref{intro1}). The second equation there is now a product of real numbers and a contradiction follows readily by unique continuation, cf. \cite{Alb} and \cite{DH}.  Therefore, the difficulty for showing the generic simplicity of the spectrum of $(SD_{\Omega})$ stems, at this stage of the argument, from the vectorial character of $\phi_i$, i.e., the fact that we are dealing with a system of PDEs.}

In this paper, we push further the
contradiction argument by computing the shape derivative of the $(k+1)$-th 
eigenvalue of $(S)_{\Omega+u}$
at every $u\in W^{\ell+1,\infty}(\Omega,\R^3)$ with
$\|u\|_{W^{\ell+1,\infty}}<\varepsilon$ small enough. The relations obtained in Eq. \eqref{intro1} for $\Omega$ are now valid for every domain $\Omega+u$ with $u$ small enough. At this stage, we are not able to derive a contradiction. So
we again take the shape derivative of the above relations  on
$\partial\Omega$ and end up with expressions of the type 
\begin{equation}\label{key}
M'(u)(x)=-(u\cdot n)(x)\frac{\partial M(0)}{\partial n}(x), \ \
x\in\partial\Omega,
\end{equation}
for $\|u\|_{W^{\ell+1,\infty}}<\varepsilon$ and where 
$$M(\cdot):=\frac{\partial \phi_i}{\partial n}\cdot n,\quad
\frac{\partial \phi_i}{\partial n}\cdot\frac{\partial \phi_j}{\partial n},\quad\hbox{or}\quad\Vert \frac{\partial \phi_i}{\partial n}\Vert^2-\Vert \frac{\partial \phi_j}{\partial n}\Vert^2.
$$

Taking into account the expression of $M$, its shape derivative $M'(u)$ can also be expressed in terms of Neumann data of the shape derivatives of the eigenfunctions whose values on $\partial\Omega$ have the regularity of $u\cdot n$. By standard elliptic theory, if $u\cdot n$ belongs to the Sobolev space $H^s(\partial\Omega)$, $M'(u)$ \emph{a priori} belongs to $H^{s-1}(\partial\Omega)$. Then, the key observation is that a \emph{gap of regularity} exists between the two sides of Eq. (\ref{key}) since the right-hand side trivially belongs to $H^s(\partial\Omega)$, for $s\leq 1$ and $\ell\geq 4$, the latter assumption needed to assert that $\frac{\partial M(0)}{\partial n}$ is continuous on $\partial\Omega$. The whole point now comes down to use that gap of regularity in order to reach a contradiction. %Two angles of attack are possible, one developped here and the other one considered by D. Henry \cite[Chap. 8]{DH}. 
In this paper, we reformulate the issue at hand as follows: how to extract pointwise information (i.e., for $x\in\partial\Omega$) reflecting the aforementioned gap of regularity and thus allow us to pursue the reasoning by contradiction. This rather elementary line of attack, first considered in \cite{Coron-Chitour-Mauro} and also applied in \cite{BCKL}, consists in choosing appropriate variations $u$ ``localized'' at an arbitrary point $x\in\partial\Omega$. 
%Ideally, one would like to use for instance a Dirac distribution supported at $x\in\partial\Omega$. However, this distribution does not belong to the space of variations  $W^{4,\infty}(\Omega,\R^3)$. {\Rd add other reasons: a Dirac distribution will not provide enough degree of freedom...} 
We note that problems treated in \cite{Coron-Chitour-Mauro} and \cite{BCKL} concerned planar domains and, therefore, equations of the type \eqref{key} were valid on closed $C^3$ curves of $\R^2$. In that case, the localization procedure is easier to handle. Indeed, the strategy adopted in \cite{Coron-Chitour-Mauro} and \cite{BCKL} consisted in extending $M'(u)$ for variations $u$ defined on $\partial\Omega$ as continuous functions except at some point $x\in\partial\Omega$. More particularly, $u=u_x$ can be taken as a Heaviside like function admitting a single jump of discontinuity at $x$. In order to exploit the  gap of regularity,  the singular part of $M'(u_x)(\cdot)$ at
$x$ (in the distributional sense) had to be computed, to eventually obtain the following expression,
$$
M'(u_{x})(\sigma)=M_0~\textrm{p.v.} \left( \frac{1}{\sigma} \right) +
R(\sigma),
$$
where $\sigma$ denotes the arclength (with $\sigma=0$ corresponding to $x$) and $R(\cdot)$ belongs to $H^{1/2-\eps}(\pO)$ for every $\eps>0$. Plugging back the above expression into Eq. (\ref{key}), one deduces that $M_0(\cdot)\equiv 0$ on $\partial\Omega$. In \cite{Coron-Chitour-Mauro}, the previous relation provided additional information and allowed to conclude the contradiction
argument. However, in \cite{BCKL}, it turns out that
$M_0(\cdot)$ is proportional to $M(0)(\cdot)$ and hence is trivially
equal to zero. To determine the first non trivial term
in the ``singular'' expansion of
$M'(u_x)+(u_x\cdot n)\frac{\partial M(0)}{\partial n}$ at $x$, in the distributional sense, a detailed study of Dirichlet-to-Neumann operators
associated to several Helmholtz equations was required. %Once such non trivial term was characterized, it was easy to conclude the contradiction argument. 

In the present paper, the ``localization'' procedure, i.e., the choice of appropriate variations $u$ for any arbitrary point $x\in\partial\Omega$, must be performed for functions
defined on a surface $\partial\Omega$ and not anymore on a curve, as in \cite{Coron-Chitour-Mauro} and \cite{BCKL}. For that purpose, after fixing an arbitrary point $x\in\partial\Omega$, we will choose sequences of smooth functions $u_{\varepsilon,x_\varepsilon}$ approximating the Dirac distribution at $x_\varepsilon$ as $\varepsilon$ tends to zero, the point $x_\varepsilon\in\partial\Omega$ being any point at distance $\varepsilon$ of $x$. The gap of regularity between the two sides of Eq.~(\ref{key}) will be now  quantified in terms of powers of 
$\displaystyle\frac{1}{\varepsilon}$ and  not anymore measured in the distributional sense.
We are therefore led to an asymptotic analysis as $\displaystyle\frac{1}{\varepsilon}$ tends to infinity and, more precisely, the right-hand side of Eq.~(\ref{key}) is a $\displaystyle O(\frac1{\varepsilon^2})$ meanwhile we will establish that the left-hand side of of Eq. (\ref{key}) 
is equal to $\displaystyle\frac{w_\varepsilon}{\varepsilon^3}+O(\frac1{\varepsilon^2})$, where $w_\varepsilon$ is bounded independently of $\varepsilon$. %only depends on the angular part of $x_\varepsilon$ (once a chart centered at $x$ is fixed). 
Letting $\varepsilon$ tend to zero, one deduces that $\lim_{\varepsilon\rightarrow 0}w_\varepsilon=0$ and finally one concludes the contradiction argument and Theorem \ref{main-theo} is established. Note that the exact characterization of $w_\varepsilon$ requires, as in \cite{BCKL}, 
a detailed study of certain Dirichlet-to-Neumann operators, but here, associated to the Stokes system. That study heavily uses many technical results borrowed from \cite[Chap. 7]{DH}, not only for handling certain weakly singular operators but also for the material which is necessary to evaluate integrals defined on the surface $\partial\Omega$. It is noteworthy that, to perform the evaluation of the surface integrals, we choose charts based at $x_\varepsilon\in\partial\Omega$ near the fixed point $x\in\partial\Omega$, but not exactly at $x$. This trick turns out to be crucial for handling the singularities in computations involving boundary layer potentials. Of an equal importance, it also provides two degrees of freedom, namely the distance and the angle (in local coordinates) between $x_\eps$ and $x$, and functions of these two variables being equal to zero give additional information to yield a contradiction.

Let us now briefly mention how the argument for Theorem \ref{main-theo-2} goes. Since the resonance relations of the type (\ref{eq:reso-1}) are clearly of countable number, we can start a contradiction argument similar to the abovementioned one. 
Therefore, there exists a resonance relation of the type (\ref{eq:reso-1}) and denoted here by $(RR)$, a domain $\Omega$ with $C^{\ell+1}$ boundary, $\ell\geq 3$ and $\varepsilon>0$, such that, for every $u\in W^{\ell+2,\infty}(\Omega,\R^3)$ with
$\|u\|_{W^{\ell+2,\infty}}<\varepsilon$, the domain $\Omega+u$ verifies $(RR)$. Moreover, since Theorem \ref{main-theo} holds true, one can assume that the eigenvalues involved in $(RR)$ are all simple for  $\Omega+u$ with $\|u\|_{W^{\ell+2,\infty}}<\varepsilon$. We then take the shape derivative of $(RR)$ but we are unable to derive any contradiction. Assuming thus that this shape derivative is equal to zero for $\Omega+u$ with $\|u\|_{W^{\ell+2,\infty}}$ small enough, we again differentiate the shape derivative of $(RR)$ at $u=0$. We then consider the  variations $u_{\varepsilon,x_\varepsilon}$ introduced previously and embark into the characterization of the main term of the second shape derivative of $(RR)$. After lengthy computations (where an extra shape derivative is performed and this justifies the extra degree of regularity of the boundary as compared with the argument for generic simplicity), we get a contradiction and conclude. It is interesting to notice the following difference between the proofs of Theorem \ref{main-theo} and Theorem \ref{main-theo-2} respectively. 
Indeed, for the first result, one uses, in the contradiction argument, the parameter defined by the angular part between $x$ and $x_\varepsilon$ whereas for the second result, it is the radial part between $x$ and $x_\varepsilon$ which plays a crucial role. Both parameters actually result from the vectorial character of our variations and that enables one to adequately address the fact that $(SD_{\Omega})$ is a system of PDEs. Therefore, one should  emphasize the flexibility of the approach proposed in this paper, which can be applied to genericity questions for other systems of PDEs. 

Before passing to the plan of the paper, we would like to make two remarks. The first one regards the reference \cite{DH}, which provides the best update for genericity questions related to PDEs, where genericity is meant with respect to the domain $\Omega$. Moreover, many new genericity results are proven there and in several situations, the author (essentially) arrives to the same critical issue as the one explained previously, i.e., equations of the type \eqref{key}. D. Henry's approach is based on transversality theory (see for instance Theorem $5.4$ in \cite[p. 63]{DH} instead of shape derivation. In his arguments, the main issue to contradict, ``$A_{k+1}$ is not dense in $A_k$'' translates into the fact that a certain operator acting on the space defined by $u\cdot n$ (i.e., the domain variation restricted to $\partial\Omega$) is actually of finite rank.  In the spirit of pseudo-differential calculus, he pursues the argument by evaluating that operator for functions $u\cdot n$ with rapidly oscillations of the type $\gamma(x)\cos(\omega\theta(x))$ where $\omega$ tends to infinity (see  \cite[Chap. 8]{DH}. The asymptotic analysis is therefore performed in terms
of the phase $\omega$. Even though the contradiction arguments follow two different view points, our approach and that of Dan Henry both consider parameterized appropriate domain variations and an asymptotic analysis with respect to the parameters. 
One must however notice that the technical details to be handled in D. Henry's approach are much more elaborated and complicated compared to ours. 
%

% proposes the second angle of attack to derive a contradiction from the gap remark, which is based on transversality theory and micro-local analysis. He chooses to reformulate such an equation as the fact that a certain pseudo-differential  operator has finite rank. Then, to contradict that finiteness assumption, D. Henry applies the operator to rapidly oscillating functions, which is a strategy much more general than ours but which seems more complicated to implement when one deals with systems of PDEs, such as in here with the Stokes system (see \cite{Pereira1} for a nice application of the strategy advocated in \cite{DH} and also \cite{Pereira2} for an extension of \cite[Chapter 8]{DH}). 

We close this introduction with a brief scope on the present work and, more generally, on genericity problems regarding differential operators (not necessarily self-adjoint) admitting only a discrete spectrum. The fact that a property is generic with respect to the domain translates geometrically as a transversality property but the technics presented in this paper
only rely on real and functional analysis tools. We do believe though that they are flexible enough to tackle similar problems for other systems of PDEs. For simplicity, we assume that the differential operators are completed with Dirichlet boundary conditions. Assume that the following two facts hold true. Firstly, one has a result of unique continuation for the 
associated overdetermined eigenvalue problem, i.e., the zero function is the unique solution of the eigenvalue problem where both Dirichlet and Neumann boundary conditions are satisfied over the whole boundary. Secondly, there exists a ``good'' representation formula for the Dirichlet-to-Neumann operator associated to the original operator, i.e., 
one is able to compute with ``good'' precision the kernels of the required layer potentials.
Since we only perform computations localized in an arbitrarily small neighborhood of any point of the boundary, one should (in principle) only need Taylor expansions of the appropriate kernels rather than their explicit global expression. Once these two sets of information are available (unique continuation result and ``good'' representation formula), our technics can come into play for genericity questions. 

The paper is organized as follows. In Section 2, we present the necessary material on the Stokes system, shape differentiation and the result displayed in Eq. \eqref{intro1} and first established in \cite{Zuazua1}. The third section is devoted to the proof of Theorem \ref{main-theo} assuming that the expansion of a Dirichlet-to-Neuman operator in terms of  inverse powers of $\varepsilon$ is available. 
Then, in Section 4, the argument to achieve such an expansion is provided
using technical results on representation formulas for Dirichlet-to-Neuman operators gathered in Section $6$. The proof of Theorem \ref{main-theo-2} is given in Section 5. Background materials on layer potentials and integral representation formulas for the Stokes system as well as the proofs of computational lemmas are gathered in Appendices A and B.

\bigskip

{\bf Acknowledgements} The authors would like to thank E. Zuazua and J. C. Saut for having suggested the problems.

\section{Definitions and preliminary results}
We start by defining precisely in Section \ref{sec:domain-topo} the topology for the set of domains in $\R^d$ with $C^\ell$ boundary, where $d,\ell\geq 2$. The material is standard and borrowed from \cite{DH} and \cite{Simon}. We then recall in Section \ref{sec:spectrum} the definition of the Stokes operator and its spectrum. The presentation adopted in this section is inspired by \cite[Chapter II]{Foias:2001cr}, \cite[Chapter 5]{simon-cours} and \cite{Zuazua1}. Results on the regularity of the eigenvalues and eigenfunctions of the Stokes operator with respect to domain variations are derived in Section \ref{sec:regularity-spectrum} and essentially based on \cite[Chapter 7]{Ka} and \cite{Conca}. Section \ref{shape0} is devoted to the shape differentiation for the Stokes system following the strategy of \cite{Simon}. We finally recall in Section \ref{sec:zuazua} J. H. Ortega and E. Zuazua's result obtained in  \cite{Zuazua1} and provide an alternative proof. This result will be the starting point of our proof for Theorems \ref{main-theo-2} and \ref{main-theo}.  

\subsection{Topology on the domains}\label{sec:domain-topo}
In this section, we provide the basic definitions needed in the paper. We work in this section in $\R^d$, $d\geq 2$, even though we will only be interested by the case $d=3$.
%A direct consequence of Lemma \ref{baire1} is the following result.
%\begin{coro}\label{baire2}
%Let $X$ be a complete metric space and $\{A_n\}_{n\in\N}$ be a sequence of open subsets of $X$ such that 
%\begin{itemize}
%\item $A_0=X$;
%\item $A_{n+1}$ is a dense subset of $A_n$, for $n\geq 0$. 
%\end{itemize}Then, the set $\bigcap_{n\in\N}A_n$ is dense in $X$.
%\end{coro}
A domain $\Omega$ of $\R^d$, $d\geq 2$, is an open bounded subset of $\R^d$. We provide now the standard topology for domains with a regular boundary. 
%\subsection{Topology on $\D^3_\ell$}
For $\ell\geq 2$, the set of domains $\Omega$ of $\R^d$ with
$C^\ell$ boundary will be denoted by $\D^d_\ell$. Following \cite{Simon}, we can define a topology on
$\D^d_\ell$. Consider the Banach space
$W^{\ell+1,\infty}(\Omega,\mathbb{R}^d)$ equipped with its standard
norm defined by 
$$
\Vert u\Vert_{l+1,\infty}:=\textrm{supess}\{\Vert D^{\alpha}
u(x)\Vert; \ 0\leq \alpha\leq l+1,\ x\in\Omega\}.$$

For $\Omega\in \D^d_\ell$, $u\in
W^{\ell+1,\infty}(\Omega,\mathbb{R}^d)$, let
$\Omega+u:=(\textrm{Id}+u)(\Omega)$ be the subset of points
$y\in\mathbb{R}^d$ such that $y=x+u(x)$ for some $x\in\Omega$ and
$\partial\Omega+u:=(\textrm{Id}+u)(\partial\Omega)$ its boundary.
For $\varepsilon>0$, let $V(\Omega,\epsilon)$ be the set of all
$\Omega+u$ with $u\in W^{\ell+1,\infty}(\Omega,\mathbb{R}^d)$ and
$\Vert u\Vert_{W^{\ell+1,\infty}}\leq\varepsilon$, small enough so that 
$\id+u:\Omega\rightarrow (\textrm{Id}+u)(\Omega)$ is a diffeomorphism. The topology of
$\D^d_\ell$ is defined by taking the sets $V(\Omega,\varepsilon)$ with
$\varepsilon$ small enough as a base of open neighborhoods of $\Omega$.

A. M. Michelleti in \cite{Micheletti:1972nx} (and also reported in \cite[Appendix 2]{DH}) considered a Courant-type metric, denoted $d_{\ell+1}$ in this paper, so that $V(\Omega,\varepsilon)$ is metrizable and each $(V(\Omega,\varepsilon),d_{\ell+1})$ is complete and separable. For any domain $\Omega\in \D^d_\ell$, we use $\D^d_\ell(\Omega)$ to denote the  the set of images $(\textrm{Id}+u)(\Omega)$  by $u\in W^{\ell+1,\infty}(\Omega,\R^d)$, which are diffeomorphic to $\Omega$. Then $\D^d_\ell(\Omega)$ is a Banach manifold modeled on $u\in W^{\ell+1,\infty}(\partial\Omega,\R^d)$ as proved in \cite[Theorem A.10]{DH}. In the sequel, we will sometimes identify, without further notice, the neighborhoods $V(\Omega,\varepsilon)$ with the corresponding open balls of $W^{\ell+1,\infty}(\Omega,\R^d)$ centered at $0$.

\begin{de}
We say that a property $(P)$ is generic in $\D^d_\ell$ if, for every $\Omega\in\D^d_\ell$, the set of
domains of $\D^d_\ell(\Omega)$ on which Property $(P)$ holds true is residual i.e., contains a countable intersection of open and dense subsets of $\D^d_\ell(\Omega)$.
\end{de}

%================================================================================================

\subsection{Spectrum of the Stokes operator with Dirichlet boundary conditions}\label{sec:spectrum}
The presentation here is inspired by \cite[Chapter II]{Foias:2001cr}, \cite[Chapter 5]{simon-cours} and \cite{Zuazua1}. 
Let $\Omega$ be a domain of $\R^d$, $d\geq 1$ with $C^1$ boundary. We use $\mathcal{D}(\Omega)$ and $\mathcal{D}'(\Omega)$ to denote respectively the space of $C^\infty$ functions with compact support in $\Omega$ and the space of distributions on $\Omega$. The duality bracket will be denoted by $\langle\cdot,\cdot\rangle_{\mathcal{D}'\times \mathcal{D}}$. 

Consider the following fundamental functional spaces for the Stokes system:
\begin{eqnarray*}
V(\Omega)&:=&\{v\in (H_0^1(\Omega))^d~\vert~\div~ v=0\},\\
H(\Omega)&:=&\{v\in (L^2(\Omega))^d~\vert~ \div~v=0~\textrm{ in }\Omega, ~v\cdot n=0~\textrm{ on }\pO\}.
\end{eqnarray*}The space $V(\Omega)$ is equipped with the scalar product of $(H^1_0(\Omega))^d$ defined by
\begin{equation}\label{hadamard}
\langle u,v\rangle_V:=\int_{\Omega}\nabla u\cdot\nabla v:=\sum_{i,j=1}^d\int_{\Omega}\frac{\partial u^i}{\partial x_j}\frac{\partial v^i}{\partial x_j}dx,
\end{equation}for $u:=(u^1,\cdots,u^d)$ and $v:=(v^1,\cdots,v^d)$ in $V(\Omega)$. The space $H(\Omega)$ is equipped with the scalar product of $(L^2(\Omega))^d$ which will be denoted by $\langle\cdot,\cdot\rangle_H$. Note that $V(\Omega)$ and $H(\Omega)$ are separable Hilbert spaces as they are closed sub-spaces of respectively $(H^1_0(\Omega))^d$ and $(L^2(\Omega))^d$. We use $L_0^2(\Omega)$ to denote the subspace of 
$L^2(\Omega)$ made of the functions $f$ with zero mean, i.e. $\displaystyle\int_{\Omega}f(x)dx=0$. %We use $\mathcal{D}(\Omega)$ the space of smooth functions with compact support in $\Omega$.
 \begin{rem}
If we define $\mathcal{V}(\Omega):=\{v\in (\mathcal{D}(\Omega))^d~\vert~ \div~v=0\}$, one can show that $V(\Omega)$ is the closure of $\mathcal{V}$ in $(H^1(\Omega))^d$ (cf. \cite[Theorem 1.6, page 18]{temam1}), and $H(\Omega)$ is the closure of $\mathcal{V}(\Omega)$ in $(L^2(\Omega))^d$ (cf. \cite[Theorem 1.4, page 15]{temam1} and \cite[Theorem 2.8, page 30]{girault-raviart}).
\end{rem}
%{\Rd{}
%We recall the strong formulation of the eigenvalue problem. For more details, the reader can consult \cite[Chap. II]{Foias:2001cr}.  Given an arbitrary function $u \in V \cup W^2(\Omega)$, the Stokes operator $T_S$ is defined to be the operator such that $T_S u \in H$ is the unique element satisfying
%$$
%\Delta_u +T_s u=\nabla p
%$$
%for some harmonique pressure field $p$.  From the same authors \cite[Chap. II]{Foias:2001cr}, we learn that $T_S=-\mathcal{P}\Delta$ where $\mathcal{P}$ is the Leray projector.  Furthermore, $T_S$ maps $V \cap W^2(\Omega)$ onto $H$ and is  coercive, self adjoint. As a map from $V$ to its dual $V^*$, we learn also that the composition of $T_S$ with the inclusion map of 
%$V$ into $V^*$ is compact. Hence,  by classical spectral theory (cf. \cite[Theorem VI.11, page 97]{brezis}), this  operator admits a non-increasing sequence of positive eigenvalues $(\mu_i)_{i\in\N}$ tending to $0$,  the corresponding eigenfunctions $(\phi_i)_{i\in \N}$ being such   that they constitute an orthonormal basis of both  $H$ and $V$.  In particular, one has
%\begin{equation}\label{eq:stokes3}
%\int_\Omega\nabla \phi_i\cdot\nabla v=\lambda_i\int_\Omega\phi_i\cdot v,\qquad \forall ~v\in V,
%\end{equation}
%where $\displaystyle \lambda_i:=\frac{1}{\mu_i}$. Note that $(\lambda_i)_{i\in\N}$ is a non-decreasing sequence tending to infinity.  We use $m(\lambda)$
%to denote the multiplicity of the eigenvalue $\lambda$.

%}

Let $f\in H$. Since the linear form on $V(\Omega)$ defined by $\displaystyle \ell(v):=\int_{\Omega} f\cdot v$, for $v\in V(\Omega)$, is continuous, by Lax-Milgram's Theorem, there exists a unique $w\in V(\Omega)$ such that, for every $v\in V(\Omega)$, 
$\langle w,v\rangle_V=\ell(v)$ and $\Vert w\Vert_V\leq C(\Omega)\Vert f\Vert_H$, where the constant $C(\Omega)$ only depends on $\Omega$. Therefore, the linear operator $L$ from $H(\Omega)$ to $H(\Omega)$ defined by $Lf=w$ is continuous. As $L$ is also self-adjoint and compact (cf. \cite[Theorem IX.16, page 169]{brezis}), then, by classical spectral theory (cf. \cite[Theorem VI.11, page 97]{brezis}), the operator $L$ admits a non-increasing sequence of positive eigenvalues $(\mu_i)_{i\in\N}$ tending to $0$, and the corresponding eigenfunctions $(\phi_i)_{i\in \N}$ can be taken so that they constitute an orthonormal basis of $H$. In particular, one has
\begin{equation}\label{eq:stokes3}
\int_\Omega\nabla \phi_i\cdot\nabla v=\lambda_i\int_\Omega\phi_i\cdot v,\qquad \forall ~v\in V,
\end{equation}
where $\displaystyle \lambda_i:=\frac{1}{\mu_i}$. Note that $(\lambda_i)_{i\in\N}$ is a non-decreasing sequence tending to infinity.  We use $m(\lambda)$
to denote the multiplicity of the eigenvalue $\lambda$.

For $v\in\mathcal{V}$, Eq. (\ref{eq:stokes3}) is equivalent to 
\begin{equation}\label{eq:stokes4}
\langle \Delta\phi_i+\lambda_i\phi_i,v\rangle_{\mathcal{D}'\times \mathcal{D}}=0.
\end{equation}

\begin{theo}[de Rham-Lions]\label{th:deRham}
Let $q\in(\mathcal{D}'(\Omega))^d$ such that
\begin{equation}\label{eq:deRham}
\langle q,v\rangle_{\mathcal{D}'\times \mathcal{D}}=0,\qquad \forall~ v\in\mathcal{V}.
\end{equation}Then, there exists $p\in \mathcal{D}'(\Omega)$ such that $q=\nabla p$.
As a consequence of Theorem \ref{th:deRham}, one deduces from Eq. (\ref{eq:stokes3}) that there exists $p_i\in \mathcal{D}'(\Omega)$ such that
\begin{equation}\label{eq:stokes5}
\Delta\phi_i+\lambda_i\phi_i=\nabla p_i.
\end{equation}
\end{theo}

\begin{rem}
Note that $p$ in Theorem \ref{th:deRham} is unique up to an additive constant. 
\end{rem}
\begin{rem}
Theorem \ref{th:deRham} is a consequence of a more general result due to de Rham (cf. \cite[Theorem 17', page 95 ]{deRham}). The version adopted in Theorem \ref{th:deRham} is due to Lions, also stated in \cite[Proposition 1.1, page 14]{temam1}. A constructive proof can be found in \cite{simon-derham}.
\end{rem}

\begin{rem}
There exists an equivalent presentation of the eigenvalue problem for the Stokes system based on the Stokes operator $T_S$, which is defined as the operator defined on $V \cap W^2(\Omega)$ by $T_S u \in H$ being the unique element satisfying
$$
\Delta u +T_S u=\nabla p,
$$
for some harmonic pressure field $p$, cf. \cite[Chapter II]{Foias:2001cr}. Then, one has $T_S=-\mathcal{P}\Delta$ where $\mathcal{P}$ is the Leray projector. One then proceeds by standard functional analysis arguments.
\end{rem}\smallskip

The following regularity result holds for $\phi_i$ and $p_i$ (cf. \cite[Section 2.6, page 38]{temam1}). 
\begin{theo}[Regularity]\label{th:regulariy}
If the domain $\Omega$ is of class $C^\ell$, for an integer $\ell\geq 2$, then, for $i\in \N$, $\phi_i\in H^\ell(\overline{\Omega})$ and $p_i\in H^{\ell-1}(\overline{\Omega})$. If $\Omega$ is of class $C^\infty$, then, for $i\in\N$, $\phi_i\in C^\infty(\overline{\Omega})$ and $p_i\in C^{\infty}(\overline{\Omega})$.
\end{theo}

We now summarize some computational results related to the Stokes system. We start by providing several notions of ``normal'' derivatives used in this context. If $\phi=(\phi^i)_{1\leq i\leq d}$, the Jacobian matrix of $\phi$ defined as $\displaystyle (\frac{\partial\phi^i}{\partial x_j})_{1\leq i,j\leq d}$ will be denoted by $\nabla\phi$. We use $\n$ to denote the outward unit normal 
to $\partial \Omega$ and the superscript $^T$ used below denotes the transpose of a matrix. The corresponding normal derivative is given by \begin{equation}
\frac{\partial \phi}{\partial n}:= \nabla \phi\cdot n,
\end{equation}
and we also have 
\begin{equation}\label{der-nu}
\frac{\partial \phi}{\partial N}:= (\nabla \phi +\nabla \phi^T)\cdot \n.
\end{equation}
Finally, the conormal derivative $\frac{\partial \phi}{\partial \nu}$ on $\partial \Omega$ is defined as follows
\begin{equation}\label{conor0}
\frac{\partial \phi}{\partial \nu}:=\frac{\partial \phi}{\partial N}-p\n.
\end{equation}
Moreover, we will use $n_x$ or $n(x)$, with $x\in\partial\Omega$, to denote the value of the outward normal vector at point $x$.

\begin{de}\label{hada}
For $a$  and  $b$ are $C^1$ functions defined on an open neighborhood of $\Omega$, we use $\nabla a : \nabla b$ to denote the following function
$$
\nabla a : \nabla b=\frac{1}{2}(\nabla a+\nabla^Ta)\cdot(\nabla b+\nabla^Tb),
$$
where $\cdot$ is defined in Eq. \eqref{hadamard} as the Hadamard product of two matrices.
\end{de}

We recall the following Green's formulas (cf. \cite[page 53]{lady}).
\begin{lemma}\label{green-coro}
Assume that $d=3$.  The  formulas
\begin{equation}
\int_{\partial\Omega} a\cdot\frac{\partial b}{\partial\nu}=\int_{\Omega}\nabla a:\nabla b+\int_{\Omega}a\cdot(\Delta b-\nabla q),
\end{equation} and  
\begin{equation}
\int_{\pO}a\cdot\frac{\partial b}{\partial \nu}-\int_{\pO}b\cdot\frac{\partial a}{\partial \nu}=\int_\Omega a\cdot ((\Delta+\eta) b-\nabla q)-\int_\Omega b\cdot ((\Delta+\eta) a-\nabla p),
\end{equation}
hold for every $\eta\in\R$ and for every pairs $(a,p)$  and  $(b,q)$ of $C^1$ functions defined on an open neighborhood of $\Omega$, taking values in $\R^3\times\R$ and satisfying $\div~a=\div~b=0$.
\end{lemma}
\begin{rem} One has noticed that the dot ``$\cdot$" has been used for scalar product a well as for the Hadamard product in Eq. (\ref{hadamard}). We will make that abuse of notation throughout the paper.
\end{rem}
%{{\Rd {Attention : vous utilisez le point pour un produit scalaire. Alors que juste avant, vous d\'efinissez le point comme une multiplication de matrices de grad. Donc enlever le point et le remplcaer par le crochet. SInon, r?fl\'echir \`a une phrase du type : we denote abusively by }}}
We also need the following obvious result. 
\begin{lemma}\label{F12}
 Let $d\geq 2$ be an integer, $a\in (C^1(\overline{\Omega}))^d\cap (H_0^1(\Omega))^d$ and $\Omega\subset\R^d$ be an open domain of class $C^1$. Then, 
 \begin{equation}\label{F1} 
 \nabla a=\frac{\partial a}{\partial n}n^T,\qquad\textrm{on }\pO.
 \end{equation}
 %,\label{F1} \\
 %\frac{\partial a}{\partial n}\cdot n&=&\div~a,\qquad\textrm{on }\pO.\label{F2}
% \end{eqnarray}
\end{lemma}

\subsection{Regularity of the eigenvalues and eigenfunctions with respect to the shape perturbation parameter}\label{sec:regularity-spectrum}
In this section, $\ell-1\geq d\geq 2$. Let $\Omega$ be a domain in $\D^d_\ell$. We consider perturbations $u$ in the space
$W^{\ell+1,\infty}(\mathbb{R}^d,\mathbb{R}^d)$ with its standard norm $\Vert\cdot\Vert_{\ell+1,\infty}
$. To study perturbations of eigenvalues, we adopt the strategy described in \cite[Chapter 7, Section 6.5, pages 423-425]{Ka} and also follow the developments of \cite[Section 4, pages 1541-1548]{Conca}.

Recall that the eigenvalue problem associated to the Stokes system on $\Omega$ with Dirichlet boundary condition is given by 
\begin{equation*}\label{eq:stokes}
(SD_\Omega)\qquad\left\{\begin{array}{cccc}
\displaystyle -\Delta \phi+\nabla p&=&\lambda \phi&\textrm{ in }\Omega,\\
\displaystyle\div \phi&=&0&\textrm{ in }\Omega,\\
\displaystyle \phi &=&0&\textrm{ on }\pO,\\
\displaystyle\int_{\Omega}p(x)dx&=&0.
\end{array}\right.
\end{equation*}
%Here we use $\phi\in\R^3$ and $p\in\R$ to denote respectively the velocity field and the pressure.  

%If $u\in W^{l,\infty}(\Omega,\mathbb{R}^d)$, 

Consider any smooth map $t\rightarrow T_t$ defined for $t$ small enough so that $T_0=\id$ and $T_t$ is a diffeomorphism from  $\Omega$ onto its image %. The transformed domain is denoted by 
$\Omega_t:=T_t(\Omega)$.  Let  $(\phi_t,p_t,\lambda_t)$  be the solution of 
\begin{equation*}\label{eq:stokest}
(SD_{\Omega_t})\qquad\left\{\begin{array}{cccc}
\displaystyle -\Delta \phi_t +\nabla p_t&=&\lambda_t \phi_t&\textrm{ in }\Omega_t,\\
\displaystyle\div ~\phi_t&=&0&\textrm{ in }\Omega_t,\\
\displaystyle \phi_t&=&0&\textrm{ on }\pO_t,\\
\displaystyle\int_{\Omega_t}p_t(y^t)dy^t&=&0.
\end{array}\right.
\end{equation*}
By Theorem \ref{th:regulariy},  $\phi_t\in (H^\ell(\overline{\Omega_t}))^d\cap (H_0^1(\Omega_t))^d$ and $p_t\in H^{\ell-1}(\overline{\Omega_t})\cap L_0^2(\Omega_t)$.

We next turn to the variational formulation of the above eigenvalue problem. 

For every $(w,q) \in (H_0^1(\Omega_t))^d \times L_0^2(\Omega_t)$, it comes
$$
\int_{\Omega_t} \nabla \phi_t : \nabla w ~dy^t-\int_{\Omega_t}p_t~\div(w)~dy^t+\int_{\Omega_t} \tr(\nabla \phi_t)q ~dy^t=\int_{\Omega_t}\lambda_t ~\phi_t w ~dy^t. 
$$

We set $\phi^t:=\phi_t\circ T_t \in (H^\ell(\overline{\Omega}))^d\cap (H_0^1(\Omega))^d$ and $p^t:=p_t\circ T_t\in H^{\ell-1}(\overline{\Omega})\cap L_0^2(\Omega)$. Define  the change of variables $y^t:=T_t(y)$ and set $z(y):=w(y^t)$ and $r(y):=q(y^t)$. Then, one  shows that 
$(\phi^t,p^t)$ satisfies the following identity 
\begin{equation}\label{eq:weak}
\int_\Omega A(t) \nabla \phi^t:\nabla z -\int_\Omega p^t \tr(B(t) \nabla z) \gamma(t)
+\int_\Omega \tr(B(t) \nabla \phi^t)r \gamma(t)=\int_{\Omega }\lambda_t ~\phi^t z \gamma(t),
\end{equation}
where $\gamma(t)=\det(DT_t) $, $A(t)=\gamma(t) (DT_t^{-1})^* (DT_t^{-1})$ and $B(t)=(DT_t^{-1})^{*}$. For fixed $t$ (small enough), all these functions defined on $\overline{\Omega}$ are of class $C^\ell$.

It follows that $( \phi^t,p^t)$ satisfies 
\begin{equation}\label{domaine_fixe0}
\left\{
\begin{array}{rlll}
-\textrm{div}(A(t)\nabla  \phi^t)+\textrm{div}(p^t \gamma(t) B(t)^*)&=&\lambda_t ~ \phi^t \gamma(t)& \text{in ~} \Omega,\\
\textrm{Tr}(B(t)\nabla  \phi^t)&=&0& \text{in ~} \Omega,\\
 \phi^t&=&0 & \text{on ~}\partial  \Omega,\\
\displaystyle\int_{\Omega}p^t(y)\gamma(t) dy&=&0.
\end{array}
\right.
\end{equation}
Let $L^2_{T_t}$ be the Hilbert space equipped with the scalar product
\begin{equation}
\langle \phi,\psi\rangle_{T_t}=\int_{\Omega} \phi(x)\psi(x) \gamma(t)~dx, 
\end{equation}
and define $\displaystyle L^2_{0,T_t}:=\left\{ v \in L^2:\int_{\Omega} v(x) \gamma(t)dx=0\right\}$.
We consider $\mathcal{C}(t)$ and $\mathcal{D}(t)$ the two operators on  $(H_0^1(\Omega))^d$ given by 
\begin{equation}\label{op-C}
\mathcal{C}(t) v =-\displaystyle\frac{1}{\gamma(t)} \div(A(t) \nabla v),
\end{equation}
and 
\begin{equation}\label{op-L}
\mathcal{D}(t)v=-\textrm{Tr}\left((DT_t^{-1})^*\nabla v\right).
\end{equation}
For fixed $t$ (small enough), the operators $\mathcal{C}(t)$ and $\mathcal{D}(t)$ have respectively coefficients of class $C^{\ell-1}$ and of class $C^{\ell}$.
The following result holds true.
\begin{theo}(cf. \cite[Lemma 4.2]{Conca})
\noindent
\begin{enumerate}
\item The operator $\mathcal{C}(t)$ is self-adjoint with respect to $\langle\cdot,\cdot\rangle_{T_t}$ and $\mathcal{C}(t)^{-1}$ is coercive, i.e., there exists $C>0$  such that, for every
$g\in H^{-1}(\Omega)$, one has 
$\langle g,\mathcal{C}(t)^{-1} g \rangle \ge C \parallel g \parallel_{H^{-1}}$.
\item The range of $\mathcal{D}(t)$ is closed and the adjoint $\mathcal{D}(t)^*$ of $\mathcal{D}(t)$ with respect to $\langle\cdot,\cdot\rangle_{T_t}$ is given by
\begin{equation}
\mathcal{D}^* q(t)= \displaystyle\frac{1}{\gamma(t)}\div(q \gamma(t)).
\end{equation}
Moreover, the null space of $\mathcal{D}(t)$ is made of constant functions on $\Omega$ and its range is equal to $L^2_{0,T_t}(\Omega)$.
\end{enumerate}
\end{theo}
Using the operators $\mathcal{C}(t)$ and $\mathcal{D}(t)$, we rewrite System \eqref{domaine_fixe0} as
\begin{equation}\label{domaine_fixe}
\left\{
\begin{array}{rlll}
\mathcal{C} (t)\phi^t+\mathcal{D}(t)^* p^t&=&\lambda_t \phi^t,&\textrm{~in ~} \Omega,\\
\mathcal{D} (t)\phi^t&=&0&\textrm{~in ~} \Omega,\\
\phi^t&=&0&\textrm{~on ~} \partial\Omega.\\
\end{array}
\right.
\end{equation}
Since the operator $\mathcal{C}(t) : (H_0^1(\Omega))^d \rightarrow (H^{-1}(\Omega))^d$ is an isomorphism, we can write 
\begin{equation}\label{fofo0}
\phi^t+\mathcal{C}(t) ^{-1}\mathcal{D}^*(t)p^t=\lambda_t \mathcal{C} (t)^{-1}\phi^t,
\end{equation}
and since $\mathcal{D}(t)\phi^t=0$, one has
$$
\mathcal{D}(t)\mathcal{C}(t) ^{-1}\mathcal{D}(t)^*p^t=\lambda_t \mathcal{D}(t)\mathcal{C} (t)^{-1}\phi^t.
$$
Thanks to the coercivity of $\mathcal{C} (t)^{-1}$, one concludes that $ \mathcal{D}(t)\mathcal{C} (t)^{-1}\mathcal{D}(t)^*$ is continuous and one-to-one  in the space orthogonal to the null space of $\mathcal{D}(t)^*$. It follows that 
$$
p^t=\lambda_t\left(\mathcal{D}(t)\mathcal{C} (t)^{-1}\mathcal{D}(t)^*\right)^{-1}\mathcal{D}(t)\mathcal{C} (t)^{-1}\phi^t.$$
Finally, reporting this expression of $p^t$ into \eqref{fofo0}, we derive that
\begin{equation}
\mathcal{C} (t) \phi^t+\lambda_t \mathcal{D}(t)^* \left(\mathcal{D}(t)\mathcal{C} (t)^{-1}\mathcal{D}(t)^*\right)^{-1}\mathcal{D}(t)\mathcal{C} (t)^{-1}\phi^t=\lambda_t \phi^t,
\end{equation}
or equivalently
\begin{equation}\label{good0}
\mathcal{C}(t)  \phi^t=\lambda_t \mathcal{A}(t) \phi^t,
\end{equation}
where we have set $$\mathcal{A}(t):= \Big[\textrm{Id}- \mathcal{D}(t)^* \left(\mathcal{D}(t)\mathcal{C}(t) ^{-1}\mathcal{D}(t)^*\right)^{-1}\mathcal{D}(t)\mathcal{C}(t) ^{-1}\Big].$$
%Setting $y^t=\lambda_t u^t$, it comes 
%$$
%\displaystyle\frac{1}{\lambda_t} \mathcal{C}  y^t+\mathcal{L}^* \left(\mathcal{L}\mathcal{C} ^{-1}\mathcal{L}^*\right)^{-1}\mathcal{L}\mathcal{C} ^{-1}y^t=y^t
%$$
%and then 
%$$
%\Big[\mathcal{C} ^{-1} \left(I-  \mathcal{L}^* \left(\mathcal{L}\mathcal{C} ^{-1}\mathcal{L}^*\right)^{-1}\mathcal{L}\right)\mathcal{C} ^{-1} \Big]y^t=-\displaystyle\frac{1}{\lambda_t} y^t.
%$$
Assume now that $t\mapsto T_t$ is analytic in a neighborhood of $t=0$. We are therefore dealing with the analytic perturbation problem described in \cite[Equations (6.42) page 424 and (6.47) page 426]{Ka}. Indeed,   (\ref{domaine_fixe}) shows that the $t$-dependent operators $\mathcal{A}(\cdot)$ and $\mathcal{C}(\cdot)$ are defined on a fixed (i.e., $t$-independent) Hilbert space. We also have that 
\begin{itemize}
\item
the operator $\mathcal{A}(t)$ is a closed operator with coefficients of class $C^{\ell-1}$.
\item
The operators  $t \mapsto  \mathcal{C}(t)$ and $ t \mapsto \mathcal{C}(t)^{-1}$ are analytic in a neighborhood of $t=0$, 
%from $W^{1,\infty}_0$ into $\mathcal{L}((H^1_0(\Omega)^3),(H^{-1}(\Omega)^3))$ and  $\mathcal{L}((L^2(\Omega)^3),(H^{-1}(\Omega)^3))$  respectively, and so is the inversion of continuous operators. 
This shows that the mapping $t\mapsto\mathcal{A}(t)$  is analytic in a neighborhood of $t=0$. Furthermore, $\mathcal{A}(t)$ is bounded when $t$ is sufficiently small. 
\end{itemize}

We next prove the following theorem.

\begin{theo}\label{th:regularity}
Let $\Omega \subset \mathbb{R}^{d}$ be an open bounded domain of
class $C^{\ell}$. Assume that $\lambda$ is an eigenvalue of multiplicity $m(\lambda)=h$ of the Stokes system with Dirichlet boundary condition on the domain $\Omega$.
%, and $(\phi_1,\dots,\phi_h)$ are corresponding orthonormal eigenfunctions with associated %pressures $(p_1,\dots, p_h)$. 
Then, there exist $h$ real-valued continuous
functions, $u \mapsto \lambda_{i}(u)$
%, and $h$ continuous
%functions with values in $H_0^1(\Omega)\times L^2_0(\Omega)$, $u \mapsto (\phi_{i}(u),%\rho_i(u))$, for $i=1,...,h$, 
defined in a neighborhood $V$ of $0$ in $W^{\ell+1,\infty}(\Omega,\mathbb{R}^{d})$ such that the
following properties hold,
\begin{itemize}
\item $\lambda_{i}(0)=\lambda$, for $i=1,...,h$;
\item for every open interval $I \subset \mathbb{R}$, such that the
  intersection of $I$ with the set of eigenvalues of $(SD_\Omega)$ contains only $\lambda$, there exists a
  neighborhood $V_{I} \subset V$ such that, for every  $u \in U_{I}$,
  there exist exactly $h$ eigenvalues counting with multiplicity,
  $\lambda_{i}(u)$, $1 \leq i \leq h$, of
  $(SD_{(\textrm{\emph{Id}}+u)\Omega})$ contained in $I$;
\item for every $u \in W^{\ell+1,\infty}(\Omega,\R^d)$ and $1
  \leq i \leq h$, consider the map
$$\Psi_i:\quad\begin{array}{ccccccc}
  J & \rightarrow & \R & \times &(H^{\ell}(\Omega)\cap H_0^1(\Omega))^d&\times &(H^{\ell-1}\cap L_0^2(\Omega))
\\
t & \mapsto & (~Ê\lambda_{i}^t(u), & & \phi_{i}^t(u),&&p_i^t(u)~Ê)
\end{array}$$
with $J\subset \R$ an open interval containing $0$, for $1
  \leq i \leq h$, $\phi_i^t(u):=\phi_{t,i}(u)\circ (\id+tu)$ and 
$p_i^t(u):=p_{t,i}(u)\circ(\id+tu)$, where $\phi_{t,i}(u)$ and $p_{t,i}(u)$ are respectively  eigenfunction and  eigenpressure of $(SD_{\Omega+tu})$. Then, for $1\leq i \leq h$, $\Psi_i$ is analytic in a neighborhood of $t=0$. Moreover,  the family $(\phi_{t,1}(u),\dots,\phi_{t,h}(u))$ is orthonormal in $H^1_0(\Omega+tu)$. 
\end{itemize}
\end{theo}

\begin{rem}\label{danger} This result is actually the Stokes system's version of \cite[Theorem 3]{Zuazua2}. It is important to insist on the fact that at $t=0$
the orthonormal family
$$(\phi_{0,1}(u),\dots,\phi_{0,h}(u)),$$ of eigenfunctions associated to $\lb$ does in general depend on $u$ and continuity of the eigenfunctions with respect to the shape parameter $u$ does not hold true. Therefore, only \emph{directional} continuity and derivability with respect to $u$ can be achieved and this is the object of the next paragraph. 
%That fact was overlooked in the statement of Theorem 3.4 in \cite{Zuazua1}. It results in several incomplete arguments for the subsequent theorems proved in the abovementioned paper. 
\end{rem}

\begin{proof}[Proof of Theorem \ref{th:regularity}]

From \cite[Chapter 7, Sections 6.2 and 6.5]{Ka}, we deduce that $(\lambda_t, \phi^t,p^t)$ defined in \eqref{domaine_fixe} is analytic in a neighborhood $J_0$ of $t=0$.
Moreover, if $\lambda=\lambda(0)$ is an eigenvalue of multiplicity $h$, by applying a standard Lyapunov-Schmidt argument (cf. for instance \cite[Chapter 7]{Ka}, \cite{PH} or \cite{DH}), one gets Theorem \eqref{th:regularity} when $T_t=\id+tu$, with $u\in W^{\ell+1,\infty}(\Omega,\mathbb{R}^{d})$, except that the maps $\Psi_i$'s take values in 
$ \R \times (H_0^1(\Omega))^d\times L_0^2(\Omega)$. 

Since the boundary $\partial \Omega$ is of class $C^\ell$, we now want to get that the maps $t\mapsto (\phi_{i}^t(u),p_i^t(u))$'s take values in   $(H^{\ell}(\Omega)\cap H_0^1(\Omega))^d\times (H^{\ell-1}\cap L_0^2(\Omega))$ and analytic (in a neighborhood of $t=0$) in this space. It is sufficient to show the result for the map $t\mapsto \phi_{i}^t(u)$. To see that, consider the power series $\phi_{i}^t(u)=\sum_{k\geq 0}A_i^{(k)}(u)t^k$ which is convergent in $J_0$ as element of $H_0^1(\Omega))^d$. Develop all terms in Eq. \eqref{good0} as power series. A trivial induction shows that $A_i^{(k)}(u)\in H_0^1(\Omega))^d$ verifies the equation $\mathcal{C}(0)A_i^{(k)}(u) =\lambda \mathcal{A}(0)A_i^{(k)}(u)+R_k$ where $R_0=0$ and $R_k$ involves the terms $A_i^{(j)}(u)$, $0\leq j\leq k-1$. 
By a standard argument (cf. \cite[Theorem 9.19, page 243]{GilTru} and \cite[Theorem 5, page 323]{Evans}), one shows by induction that firstly $A_i^{(k)}(u)\in (H^{\ell}(\Omega))^d$ (since the operators $\mathcal{A}(t)$ and $\mathcal{C}(t)$ of class $C^{\ell-1}$) and secondly their $(H^{\ell}(\Omega))^d$-norms verify the appropriate upper bounds insuring that the map $t\mapsto \sum_{k\geq 0}A_i^{(k)}(u)t^k$ takes values in $(H^{\ell}(\Omega))^d$ and is a convergent power series in some neighborhood $J$ of $t=0$ contained in $J_0$. 

\end{proof}

\subsection{Shape differentiation}\label{shape0}
The subsequent developments follow a standard strategy (cf. \cite[Theorem 2.13]{Simon} for instance) but seem to be new for the Stokes system with Dirichlet boundary conditions. 
Fix $u\in W^{\ell+1,\infty}(\Omega,\mathbb{R}^{d})$ and set $T_t=\id+tu$ for $t$ small enough. In this section, we define and calculate the differential systems verified by the derivatives at $t=0$ of the eigenfunctions $(\phi_{i,t}(u),p_{i,t}(u))$ defined in Theorem \ref{th:regularity}. For that purpose, we must first consider the  derivatives of the maps $\phi_i^t(u)$ and $p_i^t(u)$. Since we perform such a computation along a fixed analytic branch $(\lambda_{i}^t(u),\phi_{t,i}(u),p_{t,i}(u))$, the index $i$ is omitted for the rest of the paragraph.

%Recall that $(\phi_t(u),p_t(u))$ verifies the differential system $(SD_{\Omega_t})^{\lambda_t}$ and thus can also be denoted simply $(\phi_t,p_t)$. Similarly, 
%$(\phi^t,\rho^t)$ verifies respectively the differential system \eqref{domaine_fixe0} and
%can be denoted simply by $(\phi^t,p^t)$. 

According to Theorem \ref{th:regularity}, $(\phi^t(u),p^t(u))$ is analytic in a neighborhood of $t=0$ and we set
\begin{equation}\label{def:totale}
\dot \phi(u):=\frac{d\phi^t(u)}{dt}\Big\vert_{t=0},\qquad
\dot p(u):=\frac{dp^t(u)}{dt}\Big\vert_{t=0}.
\end{equation}

We next proceed in a similar way as in \cite[Theorem 2.13]{Simon}. For every open set $\omega$ whose closure is included in $\Omega$, we consider $(\phi_t(u))\vert_{\omega}$ and $(p_t(u))\vert_{\omega}$, the restrictions
of $\phi_t(u)$ and $p_t(u)$ respectively to $\omega$. As compositions of two analytic maps in a neighborhood of $t=0$,  $(\phi_t(u))\vert_{\omega}$ and $(p_t(u))\vert_{\omega}$ are also analytic in a neighborhood of $t=0$ and their derivatives at $t=0$ are equal to 
$(\dot\phi(u)-\nabla\phi\cdot u)\vert_{\omega}$ and $(\dot p(u)-\nabla p\cdot u)\vert_{\omega}$  respectively. It is then easy to see that these formulas are actually valid over the whole $\Omega$ and thus, if we use $\phi'(u)$ and $p'(u)$ to denote the derivatives at $t=0$ of $\phi_t$ and $p_t$ respectively, one finally gets that
\begin{equation}\label{eq:shape}
\phi'(u)=\dot\phi(u)-\nabla\phi\cdot u, \quad
p'(u)=\dot p(u)-\nabla p\cdot u, \hbox{ in }\Omega.
\end{equation}
We refer to $\phi'(u)$ and $p'(u)$ as the shape derivatives  in the direction $u$ of the eigenfunction and eigenpressure $(\phi,p)$ associated to $\lambda$. 

According to Theorem \ref{th:regularity}, $\dot\phi(u)$ and $\dot p(u)$ belong to 
$(H^{\ell}(\Omega)\cap H_0^1(\Omega))^d$ and $H^{\ell-1}\cap L_0^2(\Omega)$ respectively. Therefore, they admit traces on $\partial\Omega$ in $H^{\ell-1/2}(\partial\Omega)$ and $H^{\ell-3/2}(\partial\Omega)$ respectively. Since $\ell-3/2>d/2$, these traces are continous functions on $\partial\Omega$, which are of course equal to zero. From Eqs. \eqref{eq:shape} and \eqref{domaine_fixe0}, we deduce at once, by using 
Eq. \eqref{F1} that $p'(u)+\div(u p)\in L^2_0(\Omega)$ and 
$$
\phi'(u)+(u\cdot n)\frac{\partial \phi}{\partial n}=0 \hbox{ on }\partial\Omega.
$$
It remains to determine the relations satisfied by the derivatives $\phi'(u)$ and 
$p'(u)$ inside the domain $\Omega$. For that end (see \cite[Proposition 4.6]{Conca} for more details), we take the derivative with respect to time evaluated at $t=0$ of Eq. \eqref{eq:weak}. For arbitrary test functions $(z,r)\in (\mathcal{D}(\Omega))^d\times \mathcal{D}(\Omega)$, we obtain
$$
\int_\Omega \Big(A'(0) \nabla \phi+ \nabla \dot \phi(u)\Big):\nabla z -
 \int_\Omega\Big( \dot{p}(u) \div(z)+p\tr(B'(0) \nabla z)+p\div(z)\gamma'(0)\Big)
 $$
 \begin{equation}\label{eq:timeder}
 +\int_\Omega \Big(\tr(B'(0)\nabla\phi)+\div(\dot\phi(u))+\div(\phi) \gamma'(0)\Big)r
=\int_{\Omega } \Big(\lambda'(u) \phi+\lambda \dot{\phi}(u)+\lambda\phi  \gamma'(0)\Big)z.
\end{equation}
To simplify the previous equation, we use the following relations between time derivatives and shape derivatives,
$$
\gamma'(0)=\div(u),\quad A'(0)=\div(u)\id-(\nabla u+ \nabla^Tu)\hbox{ and }B'(0)=-\nabla^Tu.
$$
We first use the boundary conditions for $\phi$ and notice that the term multiplied by $\gamma'(0)$ in the integrand of Eq. \eqref{eq:timeder} is the PDE satisfied by $\phi$. Eq. \eqref{eq:timeder} then reduces to
$$
\int_\Omega \Big(\nabla \phi'(u)+\nabla(\nabla\phi\cdot u)-(\nabla u+ \nabla^Tu)\nabla \phi\Big):\nabla z-\int_\Omega \Big((p'(u)+\nabla p\cdot u)\div(z)-p\tr(\nabla^Tu\nabla z)\Big)
$$
$$
=\int_{\Omega } \Big(\lambda'(u) \phi+\lambda\phi'(u)+\lambda \nabla\phi\cdot u\Big)z,
$$
and
$$
\int_{\Omega } \Big(-\tr(\nabla^Tu\nabla \phi)+\div(\phi'(u))+\div(\nabla\phi\cdot u)\Big)r=0.
$$
After some integrations by parts and using the boundary conditions, one deduces the two identities
$$
\int_\Omega \nabla \phi'(u):\nabla z+\int_\Omega \nabla p'(u)\cdot z
=\int_{\Omega } \Big(\lambda'(u) \phi+\lambda\phi'(u)\Big)z\quad
\hbox{ and }
\int_{\Omega } \div(\phi'(u))r=0.
$$
These identities hold for every 
$(z,r)\in (\mathcal{D}(\Omega))^d\times \mathcal{D}(\Omega)$, and they yield to the equations which are valid in $\Omega$
$$
-(\Delta+\lambda)\phi'(u)+\nabla p'(u)=-\lambda'(u)\phi, \qquad
\div(\phi'(u))=0.
$$
In summary, the shape derivatives $\phi'(u)$ and $p'(u)$ satisfy the following inhomogeneous Stokes system of PDEs
\begin{equation}\label{eq:shape-stokes-tv}
\left\{\begin{array}{cccc}
\displaystyle -(\Delta+\lambda)  \phi'(u) +\nabla p'(u)&=&-\lambda'(u) 
\phi&\textrm{ in }\Omega,\\
\displaystyle \div~  \phi'(u)&=&0&\textrm{ in }\Omega,\\
\displaystyle  \phi'(u)+(u\cdot n)\frac{\partial \phi}{\partial n}
&=&0&\textrm{ on }\partial\Omega,\\
\displaystyle p'(u)+\div(u p)&\in& L_0^2(\Omega).
%\displaystyle\int_{\Omega+u}p(u)&=&0.
\end{array}\right.
\end{equation}

\subsection{Ortega-Zuazua's result}\label{sec:zuazua}
Our argument for establishing Theorem \ref{main-theo} requires the shape differentiation of the eigenvalue problem $(SD_{\Omega})$. The first step of the contradiction argument (i.e., assuming that the simplicity of the spectrum is not generic) was already conducted by J. H. Ortega and E. Zuazua in \cite{Zuazua1}. We next recall precisely the main result they obtained and, for that purpose, we introduce the following definition. 
\begin{de}
Let $\ell,d$ be two integers such that $\ell-1\geq d\geq 2$. A domain $\Omega\in\D^d_\ell$ verifies Property $(P_{OZ})_d$ if, for every $\lambda$ eigenvalue of the Stokes operator with Dirichlet boundary conditions $(SD_\Omega)$, one has $m(\lambda)\leq d-1$ and if $m(\lambda)=d-1$, for $1\leq i, j\leq d-1$ and $i\neq j$, the following three conditions must hold on $\pO$,
\begin{eqnarray}
\frac{\partial \phi_i}{\partial n}\cdot n&=&0,\label{C1}\\
\frac{\partial \phi_i}{\partial n}\cdot\frac{\partial \phi_j}{\partial n}&=&0,\label{C2}\\
\left\Vert \frac{\partial \phi_i}{\partial n}\right\Vert&=&\left\Vert \frac{\partial \phi_j}{\partial n}\right\Vert,\label{C3}
\end{eqnarray}
where the $\phi_i$'s, $1\leq i\leq d-1$ are orthonormal eigenfunctions associated with $\lambda$.
\end{de}
Then, the main result in \cite{Zuazua1} is the following.
\begin{theo}\label{th:zuazua}
Let $\ell,d$ be two integers such that $\ell-1\geq d$ and $d$ is equal to $2$ or $3$. Then Property $(P_{OZ})_d$ defined above holds true, generically with respect to $\Omega\in\D^d_\ell$.
\end{theo}
As an immediate corollary, it is proved in \cite{Zuazua1} that Property $(\simple)$ holds true generically for domains in $\R^2$. Since we adopt a  viewpoint different from \cite{Zuazua1}, we provide below a complete argument. We need to provide the following definition, similar to that of ``minimal multiplicity'' in \cite[page 56]{DH}.
\begin{de}\label{ext-eig}
Let $\Omega\in\D^d_\ell$ and $\lambda$ an eigenvalue of $(SD_\Omega)$. We use $m_\Omega(\lambda)$ to denote the liminf over the multiplicities
$m(\lambda_n)$, where  $\lambda_n$ is an eigenvalue of $(SD_{\Omega_n})$ such that $\Omega_n\rightarrow \Omega$ and $\lambda_n\rightarrow \lambda$ as $n$ tends to infinity.
\end{de}
Several remarks are now in order with regard to Definition \ref{ext-eig}.  
\begin{rem}\label{rem:1}
There exists a sequence of domains $(\Omega_n)$ in $\D^d_\ell$ and a sequence $(\lambda_n)$, where $\lambda_n$ is an eigenvalue of $(S_{\Omega_n})$, such that $\Omega_n\rightarrow \Omega$, $\lambda_n\rightarrow \lambda$ as $n$ tends to infinity and $m(\lambda_n)=m_\Omega(\lambda)$ (and it is also equal to $m_{\Omega_n}(\lambda_n)$).
\end{rem}
\begin{rem}\label{rem:2}
Moreover, Property $(P_{OZ})_d$ for a domain $\Omega\in\D^d_\ell$ is clearly equivalent to the fact that, for every $\lambda$ eigenvalue of $(SD_\Omega)$, $m_\Omega(\lambda)\leq d-1$, with the equality case described by Eqs. \eqref{C1},\eqref{C2},\eqref{C3}. 
\end{rem}
\begin{proof}[Proof of Theorem \ref{th:zuazua}]
Fix a domain $\Omega_0\in\D^d_\ell$. We define, for $l\in\N$, the sets $$A_0:=\D^d_\ell(\Omega_0),$$
and, for $l\geq 1$, consider 
$$
A_l:=\{\Omega_0+u\in A_0,\ u\in W^{\ell+1,\infty}(\Omega_0,\R^d),~m_{\Omega_0}(\lambda)\leq d-1
\textrm{ for the first $l$ first eigenvalues of }(SD_{\Omega_0+u})\}.$$ 
Set $A:=\bigcap_{l\in\N}A_l$. Note that $$A=\{\Omega_0+u\in A_0,\ u\in W^{\ell+1,\infty}(\Omega_0,\R^d), m_{\Omega_0+u}(\lambda) \leq d-1~\textrm{ if $\lambda$ is an eigenvalue of }(SD_{\Omega_0+u})\}.$$

The proof is based on the application of Baire's lemma to the sequence $\{A_l\}_{l\in\N}$. As $A_l$ is open in $A_0$ for every $l\in\N$, we only need to prove that, for $l\in\N$, $A_{l+1}$ is dense in $A_l$. 

We proceed by contradiction. Assume that $A_{l+1}$ is not dense in $A_l$. Then, there exists $u\in A_{l}\setminus A_{l+1}$ and a neighborhood $U$ of $u$ such that $U\subset A_l\setminus A_{l+1}$. 
Set $\tilde{\Omega}:= \Omega_0+u$ and let $\lambda$ be the $(l+1)$-th eigenvalue of $(SD_\oOm)$. For $s\geq 1$, let $\lambda_s(\cdot)$ be the function which associates to $\Omega\in\D^d_\ell$ the $s$-th eigenvalue of $(SD_\Omega)$. Note that $\lambda_s(\cdot)$ is continuous and 
$\lambda=\lambda_{l+1}(\oOm)$. According to the contradiction assumption, one has $m:=m_\oOm(\lambda)\geq d$ and then
$\lambda_l(\oOm)<\lambda$. As a consequence, %$m(\lambda_{s+1})$ admits a local maximum at $\Omega=\oOm$ and, 
if $(\Omega_n)$ is the sequence  in $\D^d_\ell$ considered in Remark \ref{rem:1} and associated to $\oOm$, then it has the following additional property: for $n$ large enough, there exists $\eps_n>0$ such that, for every $\Omega'$ with $d(\Omega',\Omega_n)<\eps_n$, one has that
$$
m(\lb_{l+1}(\Omega'))=m\geq d.
$$
In particular, $m(\lb_{l+1}(\cdot))$ is locally constant, equal to $m\geq d$ in an open neighborhood of $\Omega_n$, for $n$ large enough. We will contradict that latter fact, i.e. the existence of a domain $\Omega_*$ where $m(\lb_{l+1}(\cdot))$ is constant and equal to $m\geq d$ in an open neighborhood $U_*$ of $\Omega_*$. For simplicity, $\lb$ is used to denote $\lb_{l+1}(\Omega_*)$ in the remaining part of the argument. Once for all, fix an orthonormal family $v=(v_1,\dots,v_m)$ of eigenfunctions of $(SD_{\Omega_*})$ associated to $\lb$ and define the $m\times m$ matrix
$$
M(v)=\Big(\int_{\partial\Omega_*} (u \cdot  n)\frac{\partial v_i}{\partial n}\cdot\frac{\partial v_j}{\partial n}\Big)_{1\leq i,j\leq m}.
$$
Note that $M(v)$ is real symmetric. 
We next perform shape differentiation with respect to the parameter $u\in U_*$. Using the notations of Theorem \ref{th:regularity}, we consider, for every $u\in U_* $,  the $m$ analytic branches $t\mapsto(\lambda_i^t(u), \phi_{t,i}(u),p_{t,i}(u))$, for $i=1,\dots,m$, given by Theorem \ref{th:regularity}. We use $\phi(u):=(\phi_1(u),\dots,\phi_m(u))$ and $(p_1(u),\dots, p_m(u))$ respectively to denote 
$$(\phi_{0,1}(u),\dots,\phi_{0,m}(u)), \qquad (p_{0,1}(u),\dots, q_{0,m}(u)),
$$
the eigenfunctions and eigenpressures associated to $\lb$ (i.e., which correspond to the values of the
$\phi_{t,i}(u)$'s and $p_{t,i}(u)$'s at $t=0$). 
Since $v$ and $\phi(u)$ are orthonormal families of eigenfunctions associated to the same eigenvalue $\lb$, then, for every $1\leq i\leq m$, there exists $S(u)\in SO(m)$ such that
%$m$ real numbers $s_{ij}$ such that $\phi_i(u)=\sum_{j=1}^m s_{ij}v_j$ and, if  $S(u):=(s_{ij})_{1\leq i,j\leq m}$, then $S(u)\in SO(m)$ and 
$\phi(u)=vS(u)$ (with the convention that the $\phi_i(u)$'s and the $v_i$'s are viewed as column vectors of $\R^m$).
One clearly obtains that 
\begin{equation}\label{rel:1}
M(\phi(u))=S(u)M(v)S(u)^T.
\end{equation} 

We now need the following standard result whose proof is given in Section \ref{pf-deriv1} of Appendix. 
 \begin{lemma}\label{deriv1}
Using the notations defined above, then
\begin{equation}\label{id_m}
\emph{\textrm{diag}}({\lambda}_i'(u))_{1\leq i\leq m}=-M(\phi(u))
\end{equation}holds for every $u\in W^{\ell+1,\infty}(\Omega,\R^d)$.
\end{lemma}\smallskip
We next proceed with the proof of Theorem \ref{th:zuazua}, \smallskip

The fact that $m(\lb_{l+1}(\cdot))$ is constant and equal to $m$ in a neighborhood of $u=0$ is equivalent to the fact that $\lb_{i}^t(u)\equiv\lb_{j}^t(u)$, $1\leq i,j\leq m$, for $t$ small enough, implying that $\lb_i'(u)$ takes only one single value 
$\mu$ as $i$ runs from $1$ to $m$. In other words, $M(\phi(u))=-\mu \id_m$ and then one gets
$$M(v)=-\mu \id_m,
$$
thanks to Eq. (\ref{rel:1}). That yields the equations
\begin{eqnarray}
\int_{\partial\Omega_*} (u\cdot n)\big(\left\Vert \frac{\partial v_i}{\partial n}\right\Vert^2-
\left\Vert \frac{\partial v_j}{\partial n}\right\Vert^2\big)&=&0,\hbox{ for }1\leq i,j\leq m,\label{CC1}\\
\int_{\partial\Omega_*} (u\cdot n)\frac{\partial v_i}{\partial n}\cdot\frac{\partial v_j}{\partial n}&=&0,
\hbox{ for }1\leq i,j\leq d-1, ~i\neq j\label{CC2}.
\end{eqnarray}
The integrals in the above equations define linear maps in $(u.n)$ and are equal to zero in an open neighborhood of $u=0$. It thus implies that, for distinct $1\leq i,j\leq m$,
\begin{eqnarray}
\left\Vert \frac{\partial v_i}{\partial n}\right\Vert-
\left\Vert \frac{\partial v_j}{\partial n}\right\Vert&\equiv 0&\hbox{ on }
\partial\Omega_*,\label{CD1}\\
\frac{\partial v_i}{\partial n}\cdot\frac{\partial v_j}{\partial n}&\equiv 0&\hbox{ on }
\partial\Omega_*\label{CD2}.
\end{eqnarray}
Moreover, using Lemma (\ref{F12}),
%as proved in \cite{Zuazua} as a consequence of the Dirichlet boundary conditions, 
one has, for $1\leq i\leq m$,
\begin{equation}\label{CD3}
\frac{\partial v_i}{\partial n}\cdot n\equiv 0, \hbox{ on }\partial\Omega_*.
\end{equation}
Assume now that there exists $x_0\in \partial\Omega_*$ and an index $i\in \{1,\cdots,m\}$ such that $\displaystyle\frac{\partial v_i}{\partial n}(x_0)$ is not zero. According to Eqs. \eqref{CD1}, \eqref{CD2} and \eqref{CD3} , the $(m+1)$ vectors given by $\displaystyle\frac{\partial v_j}{\partial n}(x_0)$, $1\leq j\leq m$ and $n(x_0)$ are all non zero and two by two perpendicular. This is a contradiction because these vectors
 belong to a $d$-dimensional vector space. Therefore, $\displaystyle\frac{\partial v_i}{\partial n}$ must be identically  equal to zero,  for $1\leq i\leq m$.
 
Thanks to a unique continuation type of argument due to Osses (cf. \cite{Osses}) and which is valid only for $d$ equal to $2$ or $3$, one concludes that the $v_i$'s must also be identically  equal to zero, which is in contradiction with the facts that the $v_i$'s have $L^2$-norm equal to one.

\end{proof}
 
\begin{rem} 
This argument is an adaptation of the original proof by J. H. Albert in \cite{Alb} to the Stokes system with Dirichlet boundary conditions, and the perturbation parameters being the domains of $\R^d$. See also \cite[Example 4.4]{DH} for a more general situation.
\end{rem}

%================================
% Proof of the main theorem %====================================================================

\section{Proof of Theorem \ref{main-theo}}
For the rest of the paper, domains $\Omega$ are bounded subsets of $\R^3$ with $C^\ell$ boundary, i.e., $d=3$ and $\ell\geq 4$.
We follow the classical strategy initiated by J. H. Albert in \cite{Alb} for the Laplace operator with Dirichlet boundary conditions. This strategy was in particular applied successfully in \cite{Zuazua1} for the generic simplicity of the Stokes operator in two space dimensions, and in \cite{Coron-Chitour-Mauro} for other Laplacian-like operators.
Fix a domain $\Omega_0\in\D^3_\ell$. We define, for $l\in\N$, the sets 
$$
A_0:=\D^3_\ell(\Omega_0),
$$
and, for $l\geq 1$,
$$
A_l:=\{\Omega_0+u,\ u\in W^{\ell+1,\infty}(\Omega_0,\R^3),\ \Omega_0+u\in A_0
~\textrm{and the } l \textrm{ first eigenvalues of }(SD_{\Omega_0+u})\textrm{ are simple}\}.
$$ 
Set $A:=\bigcap_{l\in\N}A_l$. Note that 
$$
A=\{u\in W^{\ell+1,\infty}(\Omega_0,\R^3),\ \Omega_0+u\in A_0~\textrm{ and the eigenvalues of }(SD_{\Omega_0+u})\textrm{ are simple}\}.
$$

Again, the proof of the generic simplicity of $(SD_\Omega)$ is based on the application of Baire's lemma to the sequence $\{A_l\}_{l\in\N}$. As $A_l$ is open in $\D^3_\ell(\Omega_0)$
for every $l\in\N$, we only need to prove that, for $l\in\N$, $A_{l+1}$ is dense in $A_l$. 
We proceed by contradiction. Assume that $A_{l+1}$ is not dense in $A_l$. Then, there exists $u\in A_{l}\setminus A_{l+1}$ and a neighborhood $U$ of $u$ such that $U\subset A_l\setminus A_{l+1}$. By Theorem \ref{th:zuazua}, we can assume, without loss of generality, that there exists $\Omega:=\Omega_0+u_0$ for some $u_0\in U$ verifying the following: there exists an open neighborhood $V\subset U$ of $0$ such that, for every $u\in V$, then $\Omega+u$ verifies:
\begin{itemize}
\item[(i)] the first $l$ eigenvalues $\lambda_1(u),\dots, \lambda_l(u)$ of $(SD_{\Omega+u})$ are simple;
\item[(ii)] the multiplicity of the $(l+1)$-th eigenvalue $\lambda_{l+1}(u)$ of  $(SD_{\Omega+u})$ is equal to $2$ and, on $\pO+u$, one has
\begin{eqnarray}
\frac{\partial \phi_i}{\partial n_u}\cdot n_u&=&0,\qquad i=1,2,\label{C1-u}\\
\frac{\partial \phi_1}{\partial n_u}\cdot\frac{\partial \phi_2}{\partial n_u}&=&0\label{C2-u}\\
\left\Vert \frac{\partial \phi_1}{\partial n_u}\right\Vert&=&\left\Vert \frac{\partial \phi_2}{\partial n_u}\right\Vert,\label{C3-u}
\end{eqnarray}where $\n_u$ is used to denote the outer unit normal at $\partial\Omega+u$ and $(\phi_1,\phi_2)$ is any pair of orthonormal eigenfunctions associated with 
$\lambda_{l+1}(u)$.
\end{itemize} 
\begin{rem}
These conditions simply state that, for an eigenvalue $\lambda$ of $(SD_{\Omega})$ (say the 
$(l+1)$-th), its multiplicity is larger than or equal to $2$ and, for every variation $v$ in $W^{\ell+1,\infty}(\Omega+u,\R^3)$, there are two equal directionnal derivatives (in the direction of $v$) of $\lambda_{l+1}$ at $u$. This fact actually does not depend on the dimension $d\geq 2$ of the domain $\Omega$. In dimension two, the above conditions immediately yield that $$\frac{\partial \phi_1}{\partial n_u}\equiv\frac{\partial \phi_2}{\partial n_u}\equiv 0,$$ for any pair of orthonormal eigenfunctions associated with $\lambda_{l+1}(u)$,
and one derives at once a contradiction by the unique continuation result of \cite{Osses}, see also \cite{Zuazua1}. However, in dimension $d=3$, conditions (\ref{C1-u}), (\ref{C2-u}), and (\ref{C3-u}) do not immediately yield a contradiction since three non-zero two-by-two orthogonal vectors may exist in dimension $d=3$. 
\end{rem}

\subsection{Shape derivation of Equations  (\ref{C1-u}) (\ref{C2-u}) and (\ref{C3-u})}

%Fix a domain  $\Omega\in W^{4,\infty}(\Omega,\R^3)$ and $\lambda$, the $(l+1)$-th eigenvalue of $(S_{\Omega})$. For every variation $u$ dans $W^{4,\infty}(\Omega,\R^3)$, we use $t\mapsto(\lambda_i(tu), \phi_i(t,u),p_i(t,u))$, for $i=1,2$, to denote the two analytic branches given by Theorem \ref{th:regularity} and satisfying, in an open neighborhood of $t=0$, equation \ref{eq:stokes-tv}
%and with derivatives $\lambda_i'$, $\phi_i'(u)$ and $p_i'(u)$ verifying Eqs. \eqref{eq:stokes-tv1}-\eqref{eq:stokes-tv4}.

%We next perform the shape derivation of Eqs.  (\ref{C1-u}), (\ref{C2-u}), and (\ref{C3-u}).   

We begin with the following preliminary result.
\begin{lemma}\label{shap_der}
The  shape derivative $\phi_i'$  of $\phi_i$ in the direction $V$ satisfies 
\begin{equation}
 \frac{\partial \phi_i'}{\partial n}= \frac{\partial \phi_i'}{\partial\nu}+\langle   \frac{\partial \phi_i}{\partial n},\n'\rangle\n +V_n \frac{\partial}{\partial n}\Big( (\nabla \phi_i)^T \n \Big)+p'_i \n,
 \end{equation}
 where we use $V_n$ to denote the normal component $V\cdot n$ of the direction $V$ evaluated on $\partial\Omega$.
\end{lemma}
\begin{proof}[Proof of Lemma \ref{shap_der}]
From the fact that $\phi_i$ vanishes on $\partial\Omega$ and satisfies $\div(\phi_i)=0$, one knows  that
\begin{equation}\label{shape1}
(\nabla \phi_i)^T \n=0.
\end{equation}
Taking the shape derivative of the two sides of Eq. (\ref{shape1}), one gets 
$$
(\nabla \phi'_i)^T~ \n + (\nabla \phi_i)^T ~\n' +V_n \frac{\partial }{\partial n} \Big(  (\nabla \phi_i )^T\n  \Big)=0.
$$
Since $(\nabla \phi_i)^T=\Big(\displaystyle\frac{\partial\phi_i}{\partial n}~ \n^T\Big)^T=\n  ~ \Big(\frac{\partial\phi_i}{\partial n}    \Big)^T$, it comes that 
$$
(\nabla \phi'_i)^T~ \n +\langle  \frac{\partial\phi_i}{\partial n} , \n' \rangle \n+V_n \frac{\partial}{\partial n} \Big(  (\nabla \phi_i )^T\n  \Big)=0,
$$
hence  
$$
(\nabla \phi'_i)^T~ \n =-\langle  \frac{\partial\phi_i}{\partial n} , \n' \rangle \n-V_n \frac{\partial}{\partial n} \Big(  (\nabla \phi_i )^T\n  \Big).
$$
The proof is finished once we report this expression in the definition of the co-normal derivative   of $\phi_i$.

\end{proof}
%The shape derivation of the Eqs.  (\ref{C1-u}), (\ref{C2-u}), and (\ref{C3-u}) gives the following identities 
\begin{prop}\label{shape-derivative}
If  $\phi_i$ satisfies  (\ref{C1-u}) and (\ref{C2-u}), then we have, for $j=1,2$, 
\begin{eqnarray}
&&\Big\langle \frac{\partial \phi'_i}{\partial \nu}, \frac{\partial \phi_j}{\partial n}\Big\rangle + 
\Big\langle \frac{\partial \phi'_j}{\partial \nu}, \frac{\partial \phi_i}{\partial n}\Big\rangle\nonumber\\
&=&-V_n \Big( \frac{\partial}{\partial n}\Big(\Big\langle    \frac{\partial \phi_j}{\partial n}   , \frac{\partial \phi_i}{\partial n}\Big\rangle + 
\Big\langle \frac{\partial}{\partial n}(\nabla \phi_i ^T \n ),\frac{\partial \phi_j}{\partial\n}\Big\rangle+\Big\langle \frac{\partial}{\partial n}(\nabla \phi_j^T \n ),\frac{\partial \phi_i}{\partial\n}\Big\rangle\Big).\label{C111}
\end{eqnarray}
\end{prop}
\begin{proof}[Proof of Proposition \ref{shape-derivative}]
The shape derivative of  Eq. (\ref{C1-u}) gives 
\begin{equation}\label{ps_shape}
\Big\langle \frac{\partial \phi'_i}{\partial n}, \frac{\partial \phi_j}{\partial n}\Big\rangle+\Big\langle \frac{\partial \phi_i}{\partial n}, \frac{\partial \phi'_j}{\partial n}\Big\rangle
=-\displaystyle V_n \frac{\partial}{\partial n}\Big(\Big\langle    \frac{\partial \phi_j}{\partial n}   , \frac{\partial \phi_i}{\partial n}\Big\rangle\Big).
\end{equation}
Since $\displaystyle\frac{\partial  \phi_i}{\partial n} \cdot\n=0$, it comes from Lemma (\ref{shap_der})   that 
$$
\Big\langle \frac{\partial \phi'_i}{\partial n}, \frac{\partial \phi_j}{\partial n}\Big\rangle=\Big\langle \frac{\partial \phi'_i}{\partial \nu}, \frac{\partial \phi_j}{\partial n}\Big\rangle + V_n \Big\langle \frac{\partial}{\partial n}(\nabla \phi ^T \n ),\frac{\partial \phi_j}{\partial\n}\Big\rangle,
$$
hence we deduce that  
\begin{eqnarray*}
&&\Big\langle \frac{\partial \phi'_i}{\partial n}, \frac{\partial \phi_j}{\partial n}\Big\rangle+
\Big\langle \frac{\partial \phi'_j}{\partial n}, \frac{\partial \phi_i}{\partial n}\Big\rangle\\
& =&\Big\langle \frac{\partial \phi'_i}{\partial \nu}, \frac{\partial \phi_j}{\partial n}\Big\rangle + 
\Big\langle \frac{\partial \phi'_j}{\partial \nu}, \frac{\partial \phi_i}{\partial n}\Big\rangle
+
V_n\Big( \Big\langle \frac{\partial}{\partial n}(\nabla \phi_i ^T \n ),\frac{\partial \phi_j}{\partial\n}\Big\rangle+\Big\langle \frac{\partial}{\partial n}(\nabla \phi_j^T \n ),\frac{\partial \phi_i}{\partial\n}\Big\rangle\Big).
\end{eqnarray*}
From Eq. (\ref{ps_shape}), we get after identification that 
$$
\Big\langle \frac{\partial \phi'_i}{\partial \nu}, \frac{\partial \phi_j}{\partial n}\Big\rangle + 
\Big\langle \frac{\partial \phi'_j}{\partial \nu}, \frac{\partial \phi_i}{\partial n}\Big\rangle=-V_n \Big( \frac{\partial}{\partial n}\Big(\Big\langle    \frac{\partial \phi_j}{\partial n}   , \frac{\partial \phi_i}{\partial n}\Big\rangle\Big) + 
\Big\langle \frac{\partial}{\partial n}(\nabla \phi_i ^T \n ),\frac{\partial \phi_j}{\partial\n}\Big\rangle+\Big\langle \frac{\partial}{\partial n}(\nabla \phi_j^T \n ),\frac{\partial \phi_i}{\partial\n}\Big\rangle\Big),
$$
and this ends the proof of Proposition \ref{shape-derivative}.

\end{proof}

%For ease of notation, we set $\displaystyle\psi=\frac{\partial \phi}{\partial n}$ with $\phi$ satisfying the conditions (\ref{C1-u}), (\ref{C2-u}), and (\ref{C3-u}) and $\alpha=V_n$.   Let us show how to choose the adequate parameter $V_n$.

\subsection{Special choice of $V_n$}\label{choice-Vn}

Let  $\x \in \partial \Omega$ such that the vectors $\displaystyle\frac{\partial \phi_i}{\partial n}(\x)$ and $\displaystyle\frac{\partial \phi_j}{\partial n}(\x)$ span the tangent space  $T_\x(\partial \Omega)$. Let $\mathcal{U}_\x $ be a neighborhood of $\x$ in $\partial \Omega$ such that, for all $\y$ belonging to $\mathcal{U}_\x$, the vectors $ \frac{\partial \phi_i}{\partial n}(\y)$ and $\frac{\partial \phi_j}{\partial n}(\y)$ span $T_\y(\partial \Omega)$.   For $\y \in \partial \Omega$ near $\x$, we write the parametrized form of $\partial \Omega$ near $\x$ as a graph over the tangent plane at $\x$ :  if $\eta=P_\x(\y-\x)$ is the orthogonal projection of $y-x$ onto the tangent plane 
$T_\x(\partial\Omega)$ with $\eta$ sufficiently small, there exists an open neighborhood 
$T_\x \mathcal{U}_\x$ of $0$ in $T_\x(\partial \Omega)$ such that the 
 map $h_\x$ given by
\begin{equation}\label{HX}
\begin{array}{lrl}
h_\x :  &T_\x\mathcal{U}_\x&\mapsto \mathcal{U}_\x\\
      &\eta &\mapsto y=\x+\eta-\nu_\x(\eta)\n_\x,
\end{array}
\end{equation}
is well-defined and is a diffeomorphism onto its image. For $\y$ near $\x$, we have 
$$
\nu_\x(\eta)=\displaystyle\frac{1}{2}\eta^T K_x\eta+O(\vert\eta\vert^3),~~\text{~as~}\eta\rightarrow 0,
$$
where $K_\x$ is the symmetric matrix representing the curvature operator at $\x$. We fix once for all $\delta>0$ small enough so that
$\vert\eta\vert\leq 2\delta$ implies that $\y=\x+\eta-\nu_\x(\eta)\n_\x$ belongs to $\mathcal{U}_\x $.

One has the following standard relations where all the vectors are embedded in $\R^3$ and $\langle\cdot\rangle$ denotes here the standard scalar product in $\R^3$.
\begin{lemma}\label{diff_geometry}
Let  $n'_\x(\eta)=\frac{\partial}{\partial \eta}n_x(\eta)\in T_\x(\partial\Omega)$.  We have 
\begin{equation}\label{geo-diff}
\begin{array}{lllll}
&i)&\n_y&=&\displaystyle \frac{\nu'_\x(\eta)+\n_\x}{\sqrt{1+\vert\nu'_\x(\eta)\vert^2}},\\
&ii)& \langle \n_\x,{y}-x\rangle&=&-\displaystyle\frac{1}{2}\eta^T K_\x\eta+O(\vert\eta\vert^3) ~~\text{~as~}\eta\rightarrow 0,\\
&iii)& \langle \n_\x,\n_\y\rangle&=&\displaystyle\frac{1}{\sqrt{1+\vert\nu'_\x(\eta)\vert^2}}\\
& & &=&1-\displaystyle \frac{1}{2} \vert K_\x\eta\vert^2+O(\vert\eta\vert^3)~~\text{~as~}\eta\rightarrow 0.
\end{array}
\end{equation}
\end{lemma}
\begin{proof}[Proof of Lemma \ref{diff_geometry}]

These equations are easily obtained by standard facts from the theory of surfaces (cf. \cite[Chapter 10]{berger}) and are explicitly given in \cite[page 146]{DH}.
\end{proof}\smallskip

\begin{rem}\label{jacob}
We note that the inverse of the Jacobian of the change of variables $h_\x^{-1}:\y\rightarrow \eta=h_\x^{-1}(\y)$ from a neighborhood of $\x$ on $\partial\Omega$ to a neighborhood of $0$ in $\R^2$ is equal to $\langle \n_\x,\n_\y\rangle$.
\end{rem}\medskip

We are now ready to define $V_n$.  Let $\varepsilon <<\delta$ be a positive real number. 
For $\eta_0\in B_\varepsilon(0)\subset T_x(\partial \Omega)$, consider the points  $\x_0 \in \mathcal{U}_\x$ which can be written as 
$$
\x_0=\x+\eta_0-\nu_\x(\eta_0)\n_\x,
$$
and set 
$$
\eta_0=r_0 (\cos{\theta_0},\sin{\theta_0})^T,
$$
where $\theta_0 \in S^1$ and $0<r_0\le \varepsilon$.

%Note that $x_0$ is an arbitrary point in $\mathcal{U}_x$. %Typically, $\theta_0$ will be an arbitrary angle but $r_0$ will be chosen later, of the magnitude of $\varepsilon$.

%We  define $h_\x$ as  
%\begin{equation}
%\begin{array}{lrl}
%h_\x :  &T_\x(\partial \Omega)&\mapsto \partial \Omega\\
%      &\eta &\mapsto y=\x+\eta-n_\x(\eta)\n_\x,~~({\eta}\text{~~sufficiently small})
%\end{array}
%\end{equation}

Our choice for $V_n$ will be 
\begin{equation}\label{eq:Vn}
V_n(\y):=(\alpha_{\varepsilon,\eta_0}\beta_\delta)\circ h_\x^{-1}(\y), %\displaystyle\frac{1}{\varepsilon^2} \displaystyle \exp[-\frac{\mid h_\x^{-1}(y)-h_\x^{-1}(\x_0)\mid^2}{\varepsilon^2}]
\end{equation}
where, for $\eta\in\R^2$ (identified with $T_\x(\partial \Omega)$), 
$$
\alpha_{\varepsilon,\eta_0}(\eta):=\displaystyle\frac{1}{\varepsilon^2} \exp[-\frac{\mid \eta-\eta_0\mid^2}{\varepsilon^2}],
$$and $\beta_\delta(\cdot)$ is a smooth cut-off function equal to $1$ on $B(0,3\delta/2)$ and $0$ on $\R^2\setminus B(0,2\delta)$. 

\begin{lemma}\label{grad-V}
If $\y=h_\x(\eta)$ with $\eta\in B(0,\delta)$, then we have
\begin{equation}\label{eq:vn1}
\nabla V_n(\y)=W(h_x(y))-\langle W(h_x(y)), \nu'_\x(h_x(y))\rangle n_\x,
\end{equation}
where $\nabla V_n(\y)$ is used to denote the tangential gradient of $V_n$ along $\partial \Omega$ and $W(h_x(y))\in T_\x(\partial \Omega)$ is given by
\begin{equation}\label{eq:vn2}
W(h_x(y))=\nabla\alpha_{\varepsilon,\eta_0}(h_x(y))-\frac{\langle \nabla\alpha_{\varepsilon,\eta_0}(h_x(y)), \nu'_\x(h_x(y))\rangle }{1+\vert\nu'_\x(h_x(y))\vert^2}\nu'_\x(h_x(y)).
\end{equation}
\end{lemma}

\begin{proof}[Proof of Lemma \ref{grad-V}] Recall that $\eta=h_x(y)$. By definition of $V_n$, one has, for every tangent vector $w\in T_y(\partial\Omega)$ with $\xi=dh_x^{-1}(y)w\in T_x(\partial\Omega)$
\begin{equation}\label{eq:vn3}
\begin{array}{lrl}
\langle\nabla V_n(y), w\rangle&=&dV_n(y)w
=d(\alpha_{\varepsilon,\eta_0}\circ h_\x^{-1})(\y)w~=~d\alpha_{\varepsilon,\eta_0}(\eta)dh_x^{-1}(y)w\\
&=&d\alpha_{\varepsilon,\eta_0}(\eta)\xi~=~\langle\nabla\alpha_{\varepsilon,\eta_0}(\eta),\xi\rangle,
\end{array}
\end{equation}
where $d\alpha_{\varepsilon,\eta_0}$ denotes the differential of the scalar function 
$\alpha_{\varepsilon,\eta_0}$. 
%Since $n'_x(\eta)=K_\x\eta$ is a vector belonging  to $T_x(\partial \Omega)\perp \n_x$, we deduce after  a straightforward chain rule computation that 
%\begin{eqnarray*}
%\nabla V_n(y)&=&[\nabla h_\x^{-1}(\y)]^T\nabla\alpha_{\varepsilon,\eta_0}(\eta)=[I-n_\x'(\eta))\n_\x^T]^{-1}\nabla\alpha_{\varepsilon,\eta_0}(\eta)\\
%&=&-2\frac{\alpha_{\varepsilon,\eta_0}}{\varepsilon^2}[I+n_\x'(\eta))\n_\x^T](\eta-\eta_0) \\
%&=&-2\frac{\alpha_{\varepsilon,\eta_0}}{\varepsilon^2}(\eta-\eta_0).
%\end{eqnarray*}
From Eq. \eqref{HX}, one gets that $w=\xi-\langle \nu'_x(\eta),\xi\rangle n_x$. By a simple computation, one deduces that
$$
\xi=w+\langle \nu'_x(\eta),w\rangle n_x,
$$
where the latter inner product is taken in $\mathbb{R}^3$. According to the orthogonal sum $\mathbb{R}^3=T_x(\partial\Omega)\stackrel{\perp}{\oplus} \R n_x$, one has $\nabla V_n(y)=W(\eta)+\langle V_n(y),n_x\rangle n_x$ for some vector $W(\eta)\in T_\x(\partial \Omega)$. Since $\nabla V_n(y)\in T_y(\partial\Omega)$, on gets after taking the inner product of the decomposition of $V_n(y)$ with $n_y$ that
$$
\langle V_n(y),n_x\rangle=-\langle W(\eta),\nu'_x(\eta)\rangle.
$$
Plugging the two previous displayed equations into Eq. \eqref{eq:vn3}, one deduces that
$$
W(\eta)+\langle W(\eta),\nu'_x(\eta)\rangle \nu'_x(\eta)=\nabla\alpha_{\varepsilon,\eta_0}(\eta),
$$
which in turns yields Eq. \eqref{eq:vn2} and hence Eq. \eqref{eq:vn1}.
\end{proof}

\paragraph{Convention} For the ease of notation, the gradient of a scalar function will be considered in the following as a \emph{line} vector instead of a column vector. This convention will allow us to use the notation $\nabla$ for both scalar and vector-valued functions in a consistent way.

\subsection{End of the proof of Theorem \ref{main-theo}}
The main technical result of the paper is summarized in the following proposition. The proof is provided in Section \ref{sec:pf-expansion}.
\begin{prop}\label{expansion}
Let $x\in \partial\Omega$ such that the vectors $\displaystyle\frac{\partial \phi_i}{\partial n}(\x)$ and $\displaystyle\frac{\partial \phi_j}{\partial n}(\x)$ span the tangent space  $T_\x(\partial \Omega)$. Use $P_x$ to denote the orthogonal projection onto $T_x(\partial\Omega)$. Then, for $\varepsilon$ small enough, one has, for every $\eta_0\in B_\varepsilon\subset T_x(\partial\Omega)$ and the corresponding variation $V_n$ defined in Eq. \eqref{eq:Vn}, that, for $j=1,2$, 
\begin{eqnarray}
P_x(\frac{\partial\phi_j'}{\partial \nu}(\x))&=&2\frac{e^{-\br^2}}{\varepsilon^3}\Big(M_2^{A_1}(\br)+M_5^{A_1}(\bar{r}_0)-\bar{r}_0^2M_3^{A_1}(\bar{r}_0)\Big)\frac{\partial \phi_j}{\partial n}(x)\nonumber\\
&&+2\frac{e^{-\br^2}}{\varepsilon^3}M_4^{A_1}(\br)\langle\bar{\eta}_0,\frac{\partial \phi_j}{\partial n}(x)\rangle \bar{\eta}_0+O(\frac{1}{\varepsilon^2}),
\end{eqnarray}
where $\bar{r}_0=\frac{\Vert \eta_0\Vert}\varepsilon$, $\bar{\eta}_0=\frac{\eta_0}{\Vert \eta_0\Vert}=:(\cos(\theta_0),\sin(\theta_0))^T$ and 
$M_k^{A_1}(\cdot)$, $2\leq k\leq 5$, are nonzero entire function defined in Eqs. (\ref{eq:m2}), (\ref{eq:m3}), (\ref{def-M4}) and (\ref{eq:m5}) respectively.
\end{prop}

We can now conclude the proof of Theorem \ref{main-theo}. By conditions \eqref{C1} and \eqref{C2}, and Proposition \ref{expansion}, we have
\begin{eqnarray*}
&&\langle\frac{\partial\phi_1'}{\partial{\nu}}(\x),\frac{\partial\phi_2}{\partial\n}(\x)\rangle+\langle\frac{\partial\phi_2'}{\partial{\nu}}(\x),\frac{\partial\phi_1}{\partial\n}(\x)\rangle\\
&=&-\frac{e^{-\br^2}}{\pi\eps^3}M_4^{A_1}(\br)\langle\bar{\eta}_0,\frac{\partial\phi_1}{\partial\n}(\x)\rangle\langle\bar{\eta}_0,\frac{\partial\phi_2}{\partial\n}(\x)\rangle+O(\frac{1}{\eps^2})\\
&=&-\frac{e^{-\br^2}}{\pi\eps^3}M_4^{A_1}(\br)r_\phi^2\cos(\theta_1-\theta_0)\cos(\theta_2-\theta_0)+O(\frac{1}{\eps^2}),
\end{eqnarray*}with $\displaystyle \frac{\partial\phi_j}{\partial\n}(\x)=r_\phi(\cos\theta_j,\sin\theta_j)^T$, for $j=1,2$.\smallskip

However, Proposition \ref{shape-derivative} implies that 
\begin{equation*}
\langle\frac{\partial\phi_1'}{\partial \nu}(\x),\frac{\partial\phi_2}{\partial\n}(\x)\rangle+\langle\frac{\partial\phi_2'}{\partial \nu}(\x),\frac{\partial\phi_1}{\partial\n}(\x)\rangle=O(\frac{1}{\varepsilon^2}).
\end{equation*}
Therefore, if we now fix $\br\leq 1$ such that $M_4^{A_1}(\br)\neq 0$ and recall that $r_\phi>0$,  we have, for every $\theta_0\in S^1$, 
\begin{equation}\label{cond1}
\cos(\theta_1-\theta_0)\cos(\theta_2-\theta_0)=O(\eps).
\end{equation}
%By Lemma \ref{de:M-A1}, Eq. (\ref{cond1}) is equivalent to
%\begin{equation}\label{cond2}
%\cos(\theta_1-\theta_0)\cos(\theta_2-\theta_0)=O(\eps).
%\end{equation}
By letting $\eps$ tend to zero, we deduce that 
\begin{equation}\label{cond3}
\cos(\theta_1-\theta_0)\cos(\theta_2-\theta_0)=0,
\end{equation}
since $\theta_0$ does not depend on $\eps$. 
Again, by conditions \eqref{C1} and \eqref{C2}, one has $\vert \theta_{1}-\theta_{2}\vert=\pi/2$. Then, by 
replacing the arbitrary angle $\theta_0$ by $\theta_0-\theta_1$ in Eq. (\ref{cond1}), one derives that 
\begin{equation}\label{cond3}
\sin2\theta_0=0,
\end{equation} holding for an arbitrary angle $\theta_0\in S^1$. This yields the final contradiction and Theorem \ref{main-theo} is established.

%========================================== 
%Proof of the technical proposition on expansion %===========================================

\section{Proof of Proposition \ref{expansion}}\label{sec:pf-expansion}
This section is devoted to the proof of Proposition  \ref{expansion}. The argument starts by applying \eqref{conormal_derivative0} to  $\phi^{\lambda}=\phi'_j$, $j=1,2$, solution of \eqref{eq:stokes-tv1}-\eqref{eq:stokes-tv4}. The four terms of the right-hand side of \eqref{conormal_derivative0} correspond to four terms $W_i^j$, $1\leq i\leq 4$ respectively. 
Since $\phi'_j=-\displaystyle V_n\frac{\partial \phi_j}{\partial n}$ on $\partial \Omega$, it comes that 
\begin{equation}\label{solution}
\displaystyle\frac{\partial \phi'_j}{\partial  \nu}(\x) = W_1^j(\x)+W_2^j(\x)+W_3^j(\x)+W_4^j(\x),
\end{equation}
where we have in coordinates, for $1\leq s\leq 3$, and $\phi_j=(\phi_j^m)_{1\leq m\leq 3}$,
\begin{equation}\label{w1-0}
\Big[W_1^j(\x)\Big]_s=- 2~\textrm{p.v.}  \displaystyle\int_{\partial \Omega}\frac{\partial^2\Gamma^0_{s m}(\x-\y)}
{{\partial N(\x)\partial  N(\y)}}~V_n(\y)\frac{\partial \phi_j^m}{\partial n}(\y)~d\sigma(\y), \\
\end{equation}

\begin{eqnarray}
\Big[W_2^j(\x)\Big]_s&=&-\Big(\displaystyle\Big( \sum_{k=1}^N\Big[ (-2) ({K}_\Omega^\lambda)^*\Big]^k \Big) \Big[-2 \textrm{ p.v.}  \displaystyle\int_{\partial \Omega}\frac{\partial^2\Gamma^0_{s m}(x-y)}
{{\partial N(x)\partial  N(\y)}} ~ V_n(\y)\frac{\partial \phi_j^m}{\partial n}(\y)~d\sigma(\y) \Big)(\x)\Big]\nonumber\\
&=&\Big(\displaystyle\Big[ \sum_{k=1}^N\Big[ (-2) ({K}_\Omega^\lambda)^*\Big]^k  \Big]W_1^j\Big)(x),\label{w2-0}
\end{eqnarray}

\begin{equation}\label{w3-0}
\Big[W_3^j(\x)\Big]_s=-\displaystyle\Big( \sum_{k=0}^N\Big[ (-2) ({K}_\Omega^\lambda)^*\Big]^k \Big) \displaystyle\int_{\partial \Omega}\frac{\partial^2 \Delta_{s m}^{\lambda}(x-\y)}
{{\partial  N(x)\partial  N(\y)}} ~V_n(\y)\frac{\partial \phi_j^m}{\partial n}(\y)~d\sigma(\y),
\end{equation}
and
\begin{eqnarray}
\Big[W_4^j(\x)\Big]_s&=&-\Big[ R-\displaystyle\Big( \sum_{k=1}^N\Big[ (-2) ({K}_\Omega^\lambda)^*\Big]^k \Big)   \Big] \Big[\textrm{ p.v.}  \displaystyle\int_{\partial \Omega}\frac{\partial^2\Gamma^0_{s m}(x-\y)}
{{\partial  N(x)\partial  N(\y)}}  ~V_n(\y)\frac{\partial \phi_j^m}{\partial n}(\y)~d\sigma(\y)\Big]\nonumber \\
&-&\Big[ R-\displaystyle\Big( \sum_{k=1}^N\Big[ (-2) ({K}_\Omega^\lambda)^*\Big]^k \Big)   \Big] \displaystyle\int_{\partial \Omega}\frac{\partial^2 \Delta_{s m}^{\lambda}(x-\y)}
{{\partial  N(x)\partial  N(\y)}} V_n(\y)\frac{\partial \phi_j^m}{\partial n}(\y)~d\sigma(\y).\label{w4-0}
\end{eqnarray}
We take $V_n(\y)=\displaystyle\frac{1}{\varepsilon^2} \displaystyle \exp[-\frac{\mid h_\x^{-1}(\y)-h_\x^{-1}(\x_0)\mid^2}{\varepsilon^2}]$ and  tackle the asymptotic expansion of each term appearing in the right hand side of the equation quoted above. Our strategy is simple:  we show that the main term of the expansion is contained in $W_1^j$, where appears the effect of the hyper-singular operator.  Next, we prove that all other terms $W_i^j(x), ~i=2,3, 4$ are actually remainder terms. These are the contents of Proposition \ref{term-w1}
and Proposition \ref{term-w2-4} respectively given in the next subsections.

\subsection{Expansion of $W_1^j$}
The goal of this subsection is to provide the main term in the expansion of $W_1^j(x)$ defined in Eq. (\ref{w1-0}).
More precisely, we prove the following.
\begin{prop}\label{term-w1}With the notations of Proposition \ref{expansion}, we have, for $\eps>0$ small enough
and $j=1,2$, 
\begin{eqnarray}
P_x(W_1^j(x))&=&2\frac{e^{-\br^2}}{\varepsilon^3}\Big(M_2^{A_1}(\br)+M_5^{A_1}(\bar{r}_0)-\bar{r}_0^2M_3^{A_1}(\bar{r}_0)\Big)\frac{\partial \phi_j}{\partial n}(x)\nonumber\\
&+&2\frac{e^{-\br^2}}{\varepsilon^3}M_4^{A_1}(\br)\langle\bar{\eta}_0,\frac{\partial \phi_j}{\partial n}(x)\rangle \bar{\eta}_0+O(\frac{1}{\varepsilon^2})\label{est-W1}.
%=-\frac{e^{-\br^2}}{\pi\varepsilon^3}M_4^{A_1^1}(\br)
%\left[\langle\bar{\eta}_0,\frac{\partial \phi_j}{\partial n}(x)\rangle \bar{\eta}_0+\frac{\partial \phi_j}{\partial n}(x)
%\right]+O(\frac{1}{\varepsilon^2}),
\end{eqnarray}
%where $\bar{r}_0=\frac{\Vert \eta_0\Vert}\varepsilon$, $\bar{\eta}_0=\frac{\eta_0}{\Vert \eta_0\Vert}=:(\cos(\theta_0),\sin(\theta_0))^T$ and $M_k^{A_1}(\cdot)$, $2\leq k\leq 5$, are nonzero entire function defined in Eqs. (\ref{eq:m2}), (\ref{eq:m3}), (\ref{def-M4}) and (\ref{eq:m5}) respectively.
\end{prop}
\subsubsection{Computational lemmas}
We begin by studying the term $W_1^j(x)$ defined in Eq. (\ref{w1-0}). We start with the following lemma whose proof is deferred in Appendix. For $u=(u^m)_{1\leq m\leq 3}:\partial\Omega\mapsto\R^3$, we will use ${E}(u)(x)$ to denote the value at $x\in \partial\Omega$ of the hypersingular operator
\begin{equation}\label{E}
\Big[{E}(u)(x)\Big]_s= ~\textrm{p.v.}  \displaystyle\int_{\partial \Omega}\frac{\partial^2\Gamma^0_{s m}(x-y)}
{{\partial  N(x)\partial  N(y)}}~u^m(y)~d\sigma_y,\qquad 1\leq s\leq 3.
\end{equation}

%$$
%E(u)= ~p.v  \displaystyle\int_{\partial \Omega}\frac{\partial^2\Gamma_{ij}(x-y)}
%{{\partial N(x)\partial N(y)}}~\textbf{u}(y)~d\sigma_y 
%$$
%when $\textbf{u}= \alpha \psi$ where $\alpha : \partial \Omega\mapsto \R$ and $\psi : \partial \Omega \mapsto \R^3$. We have : 

\begin{lemma}\label{singularity_couche}
Let $\alpha : \partial \Omega\mapsto \R$ and $\psi : \partial \Omega \mapsto \R^3$ be $C^1$ functions. One has
\begin{equation}\label{decompE}
4\pi {E}(\alpha \psi)(x)=\sum_{i=1}^{5}A_i(\alpha,\psi)(x),
\end{equation}where
\begin{align}
A_1(\alpha,\psi)(x)&=\emph{\textrm{p.v.}} \int_{\partial\Omega} \frac{\langle n_x,n_y\rangle}{\vert x-y\vert^3}\Big( \langle\psi(y),x-y\rangle\nabla^T\alpha(y)+(\nabla\alpha(y)(x-y))\psi(y)\Big)d\sigma_y,\label{de:A1}\smallskip\\
A_2(\alpha,\psi)(x)&=\emph{\textrm{p.v.}} \int_{\partial\Omega} \frac{\alpha(y)\langle n_x,n_y\rangle}{\vert x-y\vert^3}\Big(\nabla\psi(y)+\nabla^T\psi(y)\Big)(x-y)d\sigma_y,\label{de:A2}\smallskip\\
A_3(\alpha,\psi)(x)&=\emph{\textrm{p.v.}} \int_{\partial\Omega}\frac{\langle n_x,\psi(y)\rangle\nabla\alpha(y)(x-y)-\langle\psi(y),x-y\rangle\nabla\alpha(y) n_x}{\vert x-y\vert^3}n_yd\sigma_y,\label{de:A3}\smallskip\\
A_4(\alpha,\psi)(x)&=\emph{\textrm{p.v.}} \int_{\partial\Omega}\frac{\alpha(y)\langle n_x,(\nabla\psi(y)-\nabla^T\psi(y))(x-y)\rangle}{\vert x-y\vert^3}n_yd\sigma_y,\label{de:A4}\smallskip\\
A_5(\alpha,\psi)(x)&=\int_{\partial\Omega}l(x,y)[\nabla(\alpha\psi)(y)]d\sigma_y,\label{de:A5}\smallskip
%A_5(\alpha,\psi)(x)&=\int_{\partial\Omega}\frac{\langle x-y, n_y\rangle}{\vert x-y\vert^3}\Big((\nabla\alpha(y)n_x)\psi(y)+\langle\psi(y),n_x\rangle\nabla^T\alpha(y)\Big)d\sigma_y,\label{de:A5}\smallskip\\
%A_6(\alpha,\psi)(x)&=-\int_{\partial\Omega}\frac{\alpha(y)\langle x-y, n_y\rangle}{\vert x-y\vert^3}\Big(\nabla\psi(y)+\nabla^T\psi(y)\Big)n_xd\sigma_y,\label{de:A6}\smallskip\\
%A_7(\alpha,\psi)(x)&=-\int_{\partial\Omega}\frac{\langle x-y, n_y\rangle}{\vert x-y\vert^3}\Big(I-3\frac{(x-y)(x-y)^T}{\vert x-y\vert^2}\Big)(\nabla\alpha(y)\psi(y))n_yd\sigma_y,\label{de:A7}\smallskip\\
%A_8(\alpha,\psi)(x)&=\int_{\partial\Omega}\frac{\alpha(y)\langle x-y, n_y\rangle}{\vert x-y\vert^3}\Big(I-3\frac{(x-y)(x-y)^T}{\vert x-y\vert^2}\Big)\mathcal{M}(\partial_y,n_y)\psi(y)d\sigma_y.\label{de:A8}
\end{align}
where $l(\cdot,\cdot)$ is a weakly singular operator of class $C^3_*(1)$ (see Appendix \ref{WSO} for a definition).
\end{lemma}\medskip

Lemma \ref{singularity_couche} will be applied with $\alpha=V_n$ and $\displaystyle\psi=\frac{\partial \phi_j}{\partial n}$, $j=1,2$. We will consider the change of variables introduced in Section \ref{choice-Vn} and, using these notations, we set 
$$\eta:=r\left( \begin{array}{l}\cos{\theta}\\ \sin{\theta}\end{array}\right),\ \eta_0:=r_0\left( \begin{array}{l}\cos{\theta_0}\\ \sin{\theta_0}\end{array}\right), \ 
\psi(\x):=r_\psi  \left( \begin{array}{l}\cos{\theta_{\psi}}\\ \sin{\theta_{\psi}}\end{array}\right),
\ \displaystyle\bar{\eta}_0:=\frac{\eta_0}{\varepsilon}, \ \
\displaystyle\br:=\frac{r_0}{\varepsilon}.
$$ 
Recall that, with the conventions of  Subsection \ref{choice-Vn}, one has $\br\leq 1$.
In the sequel, we will provide an asymptotic expansion for each of the $A_i$, $1\leq i\leq 5$,
using powers in the variable $\displaystyle\frac1\varepsilon$. We will have two types of terms, one
of the type $\displaystyle\frac{e^{-\br^2}}{\varepsilon^{m_i}}X_i$ (or $\displaystyle\frac{1}{\varepsilon^{m_i}}X_i$) 
and the other one of the type $\displaystyle\frac{e^{-\frac{\delta^2}{4\varepsilon^2}}}{\varepsilon^{m_i}}Y_i$, where $m_i$ is an integer and $X_i$, $Y_i$ are vectors with bounded norms.  For each $A_i$, $1\leq i\leq 5$, we will identify the term of the first type (i.e., $\displaystyle\frac{e^{-\br^2}}{\varepsilon^{m_i}}X_i$ or $\displaystyle\frac{1}{\varepsilon^{m_i}}X_i$) with the largest value of $m_i$, then gather them and consider all the others terms as a rest. For that purpose, we will use repeatedly the following two lemmas whose proofs are deferred in Sections \ref{pf-miracle1} and \ref{pf-miracle2} in Appendix B.
\begin{lemma}\label{miracle1}
With the notations above and for any non negative integer $m$, one has 
\begin{equation}\label{est1}
\int_{B(0,\delta)}\frac{\alpha_{\varepsilon,\eta_0}(\eta)}{\mid \eta\mid^{1-m}}d\eta\leq \frac{C(m)}{\varepsilon^{1-m}},
\end{equation}
with $C(m)$ a positive constant only depending on $m$.
\end{lemma}

\begin{lemma}\label{miracle2}
With the notations above,
\begin{equation}\label{est2}
~\emph{\textrm{p.v.}}\int_{\R^2}\frac{\alpha_{\varepsilon,\eta_0}(\eta)\eta}{\mid \eta\mid^3}d\eta=\frac{e^{-\br^2}}{\varepsilon^2}M_3^{A_1}(\br)\bar\eta_0,
\end{equation}
where $M_3^{A_1}(\cdot)$ is a nonzero entire function defined in (\ref{def-M3}) or (\ref{cv-M3}) below.
\end{lemma}\smallskip

We will provide detailed computations for $A_1(\alpha,\psi)(x)$ in the expansion of $W_1^j(x)$ and will only sketch the main steps for the other terms.
In these computations, we will systematically refer to the following procedures. 
\begin{description}
\item[(P1)] The first one consists of decomposing a $C^1$ vector-valued function $F(y)$ in two parts as $F(y)=F(x)+ G(x)(y-x)$, where $G$ is a continuous matrix-valued function. 
\item[(P2)] The second procedure consists of cutting an integral $\displaystyle\int_{\partial\Omega}\cdots d\sigma_y$ as 
$$
\int_{\partial\Omega}\cdots d\sigma_y=\int_{B(0,2\delta)}\cdots d\eta=
\int_{B(0,\delta)}\cdots d\eta+\int_{B(0,2\delta)\setminus B(0,\delta)}\cdots d\eta,
$$
and majorizing the second one by $\displaystyle C_i\frac{e^{-\frac{\delta^2}{4\varepsilon^2}}}{\varepsilon^{m_i}}$ for appropriate
constant $C_i$ and integer $m_i$. Finally, note that $\displaystyle\langle \psi(x),\n_x\rangle=0$.
\item[(P3)] In certain  integrals of the type $\int_{\partial\Omega}\cdots d\sigma_y=
\int_{B(0,2\delta)}\cdots d\eta$, the term  
$\nabla V_n(y)$ will be expressed after changing variables as
\begin{equation}\label{eq:VN}
\nabla \alpha_{\varepsilon,\eta_0}(\eta)+G(\eta)\nabla \alpha_{\varepsilon,\eta_0}(\eta)+H(\eta)n_x,
\end{equation}
where $G(\eta)$ denotes the non-positive symmetric matrix 
$-\frac{\nu'_\x(\eta)\nu'_\x(\eta)^T}{1+\vert\nu'_\x(\eta)\vert^2}$ and $H(\eta)$ is a real-valued function. Note first that, for $\eta$ small enough, one has that $\vert G(\eta)\vert \leq C\vert \eta\vert^2$ for some universal positive constant $C$. Thus the contribution arising from  $G(\eta)\nabla \alpha_{\varepsilon,\eta_0}(\eta)$ in  the expression of 
$\displaystyle\frac{\partial \phi'_j}{\partial  \nu}(\x)$ will be shown below to be trivially a $\displaystyle O(\frac1{\varepsilon})$. Moreover, in the integrals abovementionned, only their contributions tangent to $T_x(\Omega)$ are relevant, thanks to Eqs. \eqref{C1} and \eqref{C111}. In conclusion, it is enough to only estimate the contribution of $\nabla \alpha_{\varepsilon,\eta_0}(\eta)$ in Eq. \eqref{eq:VN}.
\end{description}

\subsubsection{Asymptotic expansion of $A_1$}
We give in this paragraph the asymptotic expansion of $A_1(\alpha,\psi)(x)$ with respect to $\varepsilon$. Recall that
$$A_1(\alpha,\psi)(x)=\textrm{p.v.} \int_{\partial\Omega} \frac{\langle n_x,n_y\rangle}{\vert x-y\vert^3}\Big( \langle\psi(y),x-y\rangle\nabla^T\alpha(y)+(\nabla\alpha(y)(x-y))\psi(y)\Big)d\sigma_y.$$
\begin{prop}\label{est-A1}
For $\varepsilon>0$ small enough, one has
\begin{eqnarray}
P_x(A_1(\alpha,\psi)(x))&=&2\frac{e^{-\br^2}}{\varepsilon^3}\Big(M_2^{A_1}(\br)+M_5^{A_1}(\bar{r}_0)-\bar{r}_0^2M_3^{A_1}(\bar{r}_0)\Big)\psi(\x)\nonumber\\
&+&2\frac{e^{-\br^2}}{\varepsilon^3}M_4^{A_1}(\br)\langle\bar{\eta}_0,\psi(\x)\rangle \bar{\eta}_0+O(\frac{1}{\varepsilon^2}).\label{EST-A1}
\end{eqnarray}
\end{prop}\medskip
For the sake of clarity, we set $\displaystyle A_1(\alpha,\psi)(x):=A_{1,1}(\alpha,\psi)(x)+A_{1,2}(\alpha,\psi)(x)$ with
\begin{eqnarray}
A_{1,1}(\alpha,\psi)(x)&:=&\textrm{p.v.} \int_{\partial\Omega} \frac{\langle n_x,n_y\rangle}{\vert x-y\vert^3}\langle\psi(y),x-y\rangle\nabla^T\alpha(y)d\sigma_y,\\
A_{1,2}(\alpha,\psi)(x)&:=&\textrm{p.v.} \int_{\partial\Omega} \frac{\langle n_x,n_y\rangle}{\vert x-y\vert^3}(\nabla\alpha(y)(x-y))\psi(y)d\sigma_y.
\end{eqnarray}We will establish separately estimates of these two terms in Lemma \ref{A11} and Lemma \ref{A12}.

\begin{lemma}\label{A11}
For $\varepsilon>0$ small enough, one has
\begin{equation}\label{est-A11}
P_x(A_{1,1}(\alpha,\psi)(x))=2\frac{e^{-\br^2}}{\varepsilon^3}\Big(M_4^{A_1}(\br)\langle\bar{\eta}_0,\psi(\x)\rangle \bar{\eta}_0+ÊM_2^{A_1}(\br)\psi(\x)\Big)+O(\frac{1}{\varepsilon^2}),
\end{equation}where $M_2^{A_1}(\cdot)$ and $M_4^{A_1}(\cdot)$ are non-zero entire functions defined by (\ref{eq:m2}) and (\ref{de:M4}) respectively.
\end{lemma}
\begin{proof}[Proof of Lemma \ref{A11}]
Using the change of variables introduced in Subsection \ref{choice-Vn} and taking into account Lemma \ref{grad-V} and Remark \ref{jacob}, we have
\begin{equation*}
P_x(A_{1,1}(\alpha,\psi)(x))=\frac{2}{\varepsilon^2}~\textrm{p.v.}\int_{B(0,2\delta)}
\frac{\alpha_{\varepsilon,\eta_0}(\eta)\langle \eta-\nu_\x(\eta)\n_\x
,\psi(y)\rangle}{\Big(\mid \eta\mid^2+\mid \nu_x(\eta)\mid^2\Big)^{\frac{3}{2}}}(\textrm{Id}_2+G(\eta))(\eta-\eta_0)d\eta.
%= I^{A_1^1}(\alpha,\psi)+R^{A_1^1}(\alpha,\psi),
\end{equation*}
Then, by taking into account Procedures $(P1)$ and $(P3)$, 
$$
P_x(A_{1,1}(\alpha,\psi)(x))= I^{A_{1,1}}(\alpha,\psi)(x)+J^{A_{1,1}}(\alpha,\psi)(x)+R^{A_{1,1}}(\alpha,\psi),
$$
with 
\begin{eqnarray}
I^{A_{1,1}}(\alpha,\psi)&:=&\frac{2}{\varepsilon^2}~\textrm{p.v.}\int_{B(0,\delta)}\frac{\alpha_{\varepsilon,\eta_0}(\eta)\langle \eta,\psi(x)\rangle}{\vert\eta\vert^3}(\eta-\eta_0)d\eta,\label{I1-1}\\
J^{A_{1,1}}(\alpha,\psi)(x)&:=&\frac{2}{\varepsilon^2}\int_{B(0,\delta)}\frac{\alpha_{\varepsilon,\eta_0}(\eta)O(\vert\eta\vert^2)}{\vert\eta\vert^3}(\eta-\eta_0)d\eta,\label{J1-1}\\
R^{A_{1,1}}(\alpha,\psi)(x)&:=&\int_{B(0,2\delta)\setminus B(0,\delta)}\cdots\label{R1-1},
\end{eqnarray}\smallskip
where, in $R^{A_{1,1}}(\alpha,\psi)(x)$, one has the same integrand (in local coordinates) as in $A_1(\alpha,\psi)(x)$.
Clearly, there  exists a positve constant $C_\delta$ only depending on $\delta$ such that,
for $\varepsilon$ small enough with respect to $\delta$, one has
\begin{equation}\label{est-Ra1}
\Vert R^{A_{1,1}}(\alpha,\psi)(x)\Vert\leq C_\delta\frac{e^{-\frac{\delta^2}{\varepsilon^2}}}{\varepsilon^4}.
\end{equation}
Moreover, one can apply Lemma~\ref{miracle1} to $J^{A_{1,1}}(\alpha,\psi)(x)$, one gets that 
$$
\Vert J^{A_{1,1}}(\alpha,\psi)(x)\Vert\leq \frac2{\varepsilon^2}(C(1)+\frac{C(0)r_0}{\varepsilon}),
$$
and since $\displaystyle\frac{r_0}{\varepsilon}=O(1)$, one finally deduces that there exists a positive constant $C_*$ such that
\begin{equation}\label{est-Ja1}
\Vert J^{A_{1,1}}(\alpha,\psi)(x)\Vert\leq \frac{C_*}{\varepsilon^2}.
\end{equation}
Note that, for $\varepsilon$ small enough the upper bound of \eqref{est-Ja1} is larger than that of \eqref{est-Ra1}.

It remains to estimate $I^{A_{1,1}}(\alpha,\psi)(x)$. First of all, notice that the norm of 
$$
\frac{2}{\varepsilon^2}\int_{\R^2\setminus B(0,\delta)}\frac{\alpha_{\varepsilon,\eta_0}(\eta)\langle \eta,\psi(x)\rangle}{\vert\eta\vert^3}(\eta-\eta_0)d\eta,
$$
is clearly less than or equal to $\displaystyle\frac{C_\delta e^{-\frac{\delta^2}{4\varepsilon^2}}}{\varepsilon^4}$ for some positive constant $C_\delta$ only dependent on $\delta$ and $\varepsilon$ small enough with respect to $\delta$. 

We can therefore estimate, instead of $I^{A_{1,1}}(\alpha,\psi)(x)$, the quantity $\tilde{I}^{A_{1,1}}(\alpha,\psi)(x)$ defined by
\begin{equation}\label{tI1-1}
\tilde{I}^{A_{1,1}}(\alpha,\psi)(x):=\frac{2}{\varepsilon^2}~\textrm{p.v.}\int_{\R^2}\frac{\alpha_{\varepsilon,\eta_0}(\eta)\langle \eta,\psi(x)\rangle}{\vert\eta\vert^3}(\eta-\eta_0)d\eta.
\end{equation}
By using polar coordinates, one gets
\begin{eqnarray*}
&&\tilde{I}^{A_{1,1}}(\alpha,\psi)(x)\\
&=&2\frac{e^{-\br^2}}{\varepsilon^4}r_{\psi}\int_{0}^{\infty}\exp(-\frac{r^2}{\varepsilon^2})dr\int_{0}^{2\pi}\cos(\theta-\theta_\psi)\exp(2\frac{r}{\varepsilon}\br\cos(\theta-\theta_0))\begin{pmatrix}\cos\theta\\ \sin\theta\end{pmatrix}d\theta\smallskip\nonumber\\
&&-2\frac{e^{-\br^2}}{\varepsilon^3}r_{\psi}\begin{pmatrix}\cos\theta_0\\\sin\theta_0\end{pmatrix}\br~\textrm{p.v.}\int_{0}^\infty \frac{\exp{(-\frac{r^2}{\varepsilon^2})}}{r}dr\int_{0}^{2\pi}\cos(\theta-\theta_\psi)\exp(2\frac{r}{\varepsilon}\br\cos(\theta-\theta_0))d\theta\smallskip\nonumber\\
%&=&2\frac{e^{-\br^2}}{\varepsilon^3}r_{\psi}\int_{0}^{\infty}\exp(-{r^2})dr\int_{0}^{2\pi}\cos(\theta+\theta_0-\theta_\psi)\exp(2{r}\br\cos\theta)\begin{pmatrix}\cos(\theta+\theta_0)\\ \sin(\theta+\theta_0)\end{pmatrix}d\theta\smallskip\nonumber\\
%&&-2\frac{e^{-\br^2}}{\varepsilon^3}r_{\psi}\begin{pmatrix}\cos\theta_0\\\sin\theta_0\end{pmatrix}\br~\textrm{p.v.}\int_{0}^{\infty} \frac{\exp{(-{r^2})}}{r}dr\int_{0}^{2\pi}\cos(\theta+\theta_0-\theta_\psi)\exp(2{r}\br\cos\theta)d\theta\smallskip\nonumber\\
&=&2\frac{e^{-\br^2}}{\varepsilon^3}r_{\psi}\begin{pmatrix}
M_1^{A_1}(\br)\cos(\theta_0-\theta_{\psi})\cos\theta_0+M_2^{A_1}(\br)\sin(\theta_0-\theta_{\psi})\sin\theta_0\smallskip\\
M_1^{A_1}(\br)\cos(\theta_0-\theta_{\psi})\sin\theta_0-M_2^{A_1}(\br)\sin(\theta_0-\theta_{\psi})\cos\theta_0
\end{pmatrix}\smallskip\nonumber\\
&&-2\frac{e^{-\br^2}}{\varepsilon^3}r_{\psi}\cos(\theta_0-\theta_\psi)\begin{pmatrix}\cos\theta_0\\\sin\theta_0\end{pmatrix}\br^2 M_3^{A_1^1}(\br)\smallskip\nonumber\\
&=&2\frac{e^{-\br^2}}{\varepsilon^3}r_{\psi}\begin{pmatrix}
[M_1^{A_1}(\br)-\br M_3^{A_1^1}(\br)]\cos(\theta_0-\theta_{\psi})\cos\theta_0+M_2^{A_1}(\br)\sin(\theta_0-\theta_{\psi})\sin\theta_0\smallskip\\
[M_1^{A_1}(\br)-\br M_3^{A_1^1}(\br)]\cos(\theta_0-\theta_{\psi})\sin\theta_0-M_2^{A_1}(\br)\sin(\theta_0-\theta_{\psi})\cos\theta_0
\end{pmatrix},\nonumber
%\int_{0}^{2\pi}\cos(\theta+\theta_0-\theta_\psi)\exp(2{r}\cos\theta)\begin{pmatrix}\cos(\theta+\theta_0)\\ \sin(\theta+\theta_0)\end{pmatrix}d\theta
\end{eqnarray*}where
\begin{eqnarray}
M_1^{A_1}(\br)&:=&\int_{0}^{\infty}\exp(-r^2)dr\int_0^{2\pi}\cos^2\theta\exp(2r\br\cos\theta)d\theta,\label{eq:m1}\\
M_2^{A_1}(\br)&:=&\int_{0}^{\infty}\exp(-r^2)dr\int_0^{2\pi}\sin^2\theta\exp(2r\br\cos\theta)d\theta,\label{eq:m2}\\
M_3^{A_1}(\br)&:=&\frac1{\br}~\textrm{p.v.}\int_{0}^{\infty}\frac{\exp(-r^2)}{r}dr\int_{0}^{2\pi}\cos\theta\exp(2r\br\cos\theta)d\theta..\label{eq:m3}
%M_3^{A_1^1}(\br)&:=&~\textrm{p.v.}\int_{0}^{\infty} \frac{\exp{(-{r^2})}}{r}dr\int_{0}^{2\pi}\cos\theta%\exp(2{r}\br\cos\theta)d\theta.\label{eq:m3}
\end{eqnarray}
%and $M_3^{A_1}(\cdot)$ is defined in (\ref{def-M3}).

The needed information about the functions $\displaystyle M_i^{A_1}(\cdot)$, $i=1,2,3$, is gathered in the following lemma, whose proof is given in Section \ref{pf-de:M-A1} in appendix.
\begin{lemma}\label{de:M-A1}
For $i=1,2$, $\displaystyle M_i^{A_1}(\cdot)$ are entire functions. Moreover, the  function $M_4^{A_1}(\cdot)$ defined by the relation
\begin{equation}\label{de:M4}
M_4^{A_1}(z):=\frac{1}{z^2}(M_1^{A_1}(z)-z^2M_3^{A_1}(z)-M_2^{A_1}(z))
\end{equation}
is a nonzero entire function.
\end{lemma}
Using Lemma \ref{de:M-A1}, we further simplify $\tilde{I}_1^{A_{1,1}}$ as follows.
\begin{eqnarray}
&&\tilde{I}_1^{A_{1,1}}(\alpha,\psi)\nonumber\smallskip\\
&=&2\frac{e^{-\br^2}}{\varepsilon^3}r_{\psi}\begin{pmatrix}
[M_1^{A_1}(\br)-\br^2 M_3^{A_1}(\br)]\cos(\theta_0-\theta_{\psi})\cos\theta_0+M_2^{A_1}(\br)\sin(\theta_0-\theta_{\psi})\sin\theta_0\nonumber\smallskip\\
[M_1^{A_1}(\br)-\br^2 M_3^{A_1}(\br)]\cos(\theta_0-\theta_{\psi})\sin\theta_0-M_2^{A_1}(\br)\sin(\theta_0-\theta_{\psi})\cos\theta_0
\end{pmatrix}\nonumber\smallskip\\
&=&2\frac{e^{-\br^2}}{\varepsilon^3}r_{\psi}\br^2M_4^{A_1}(\br)\cos(\theta_0-\theta_\psi)\begin{pmatrix}\cos\theta_0\\\sin\theta_0\end{pmatrix}+~Ê2\frac{e^{-\br^2}}{\varepsilon^3}r_{\psi}M_2^{A_1}(\br^2)\begin{pmatrix}\cos\theta_\psi\\\sin\theta_\psi\end{pmatrix}\smallskip\nonumber\\
&=&2\frac{e^{-\br^2}}{\varepsilon^3}M_4^{A_1}(\br)\langle\bar{\eta}_0,\psi(\x)\rangle \bar{\eta}_0
%r_{\psi}M_4^{A_1^1}(\br)\cos(\theta_0-\theta_\psi)\begin{pmatrix}\cos\theta_0\\\sin%\theta_0\end{pmatrix}
+~Ê2\frac{e^{-\br^2}}{\varepsilon^3}M_2^{A_1}(\br)\psi(\x).
\end{eqnarray}

This ends the proof of Proposition \ref{est-A1}.

\end{proof}

\begin{lemma}\label{A12}
With the above notations, for $\varepsilon>0$ small enough, one has
\begin{equation}\label{est-A12}
P_x(A_{1,2}(\alpha,\psi)(x))=\frac{2e^{-\bar{r}_0^2}}{\varepsilon^3}(M_5^{A_1}(\bar{r}_0)-\bar{r}_0^2M_3^{A_1}(\bar{r}_0))\psi(\x)+O(\frac{1}{\eps^2}),
\end{equation}
where $M_5^{A_1}(\cdot)$ is the non zero entire function defined as $M_1^{A_1}(\cdot)+M_2^{A_1}(\cdot)$.
\end{lemma}
\begin{proof}[Proof of Lemma \ref{A12}]
We proceed similarly as in the proof of Lemma \ref{A11}. Besides remainder terms, one must the principal term given by
\begin{equation*}
I^{A_{1,2}}(\alpha,\psi)(x)=\textrm{p.v.}\frac{2}{\varepsilon^2}\int_{\R^2}\frac{\alpha_{\varepsilon,\eta_0}(\eta)\langle\eta,\eta-\eta_0\rangle}{\vert\eta\vert^3}d\eta~Ê\psi(\x).
\end{equation*}

Using polar coordinates, one gets
\begin{eqnarray*}
&&I^{A_{1,2}}(\alpha,\psi)(\x)\\
&=&\Big(\frac{2}{\varepsilon^2}\int_{\R^2}\frac{\alpha_{\varepsilon,\eta_0}(\eta)}{\vert\eta\vert}d\eta-\frac{2}{\varepsilon^2}~\textrm{p.v.}\int_{\R^2}\frac{\alpha_{\varepsilon,\eta_0}(\eta)\langle\eta,\eta_0\rangle}{\vert\eta\vert^3}d\eta\Big)\psi(\x)\\
&=&\Big(\frac{2e^{-\bar{r}_0^2}}{\varepsilon^3}\int_{ r =0}^{\infty}\int_{0}^{2\pi}\displaystyle 
e^{-r^2}e^{2 r\bar{r}_0\cos{\theta}}~d\theta~\displaystyle{dr}-\frac{2e^{-\bar{r}_0^2}}{\varepsilon^3}\bar{r}_0~\textrm{p.v.}\int_{ r =0}^{\infty}\int_{0}^{2\pi}\displaystyle 
e^{-r^2}e^{2 r\bar{r}_0\ \cos{\theta}}\cos{\theta}~d\theta~\displaystyle\frac{dr}{r}\Big)\psi(x)\\
&=&\frac{2e^{-\bar{r}_0^2}}{\varepsilon^3}(M_5^{A_1}(\bar{r}_0)-\bar{r}_0^2M_3^{A_1}(\bar{r}_0))\psi(\x),
\end{eqnarray*}where $M_3^{A_1}(\bar{r}_0)$ and $M_5^{A_1}(\bar{r}_0)$ are given respectively by (\ref{def-M3}) and (\ref{eq:m5}).

\end{proof}

\subsubsection{Asymptotic expansion of $A_i$ for $2\leq i\leq 5$}
We establish the following proposition for the asymptotic expansion of $A_i$ with $i=2,\dots,5$.
\begin{prop}\label{est-others}
For $i=2,\dots, 5$ and $\varepsilon>0$ small enough, one has
\begin{equation}
P_x(A_i(\alpha,\psi)(x))=O(\frac{1}{\varepsilon^2}).
\end{equation}
\end{prop}

\begin{proof}[Proof of Proposition \ref{est-others}]
We proceed similarly as in the proof of Lemma \ref{A11}.\medskip

For $A_2(\alpha,\psi)(x)$, we only need to estimate the following term:
 \begin{equation*}
R^{A_2}(\alpha,\psi)(x):=(\nabla\psi(x)+\nabla^T\psi(x))~\textrm{p.v.}\int_{\R^2}\frac{\alpha_{\varepsilon,\eta_0}(\eta)}{\vert\eta\vert^3}\eta d\eta.
\end{equation*}
By Lemma \ref{miracle2}, one gets
\begin{equation}
R^{A_2}(\alpha,\psi)(x)=\frac{e^{-\br^2}}{\varepsilon^2}M_3^{A_1^1}(\br)(\nabla\psi(x)+\nabla^T\psi(x))\bar{\eta}_0=O(\frac{1}{\varepsilon^2}).
\end{equation}\medskip

For $A_3(\alpha,\psi)(x)$, we first note that $\displaystyle\nabla\alpha(y)n_x=0$, and 
\begin{eqnarray*}
\langle \n_x,\psi(y)\rangle&=&\langle\n_x,\psi(x+\eta-n_\x(\eta)\n_x)\rangle=\langle\n_x,\psi(x)+\nabla\psi(x)\eta+O(\vert\eta\vert^2)\rangle\\
&=&\langle\nabla\psi(x)^T\n_x,\eta\rangle+O(\vert\eta\vert^2).
\end{eqnarray*}Thus, we need to estimate the following integral,
\begin{eqnarray*}
R^{A_3}(\alpha,\psi)(x)&:=&\frac{2}{\varepsilon^2}\int_{\R^2}\alpha_{\varepsilon,\eta_0}(\eta)\frac{\langle\nabla\psi(x)^T\n_x,\eta\rangle}{\vert\eta\vert^3}\langle\eta-\eta_0,\eta\rangle d\eta~Ê\n_x.
\end{eqnarray*}
One can clearly apply Lemma~\ref{miracle1} to $R_1^{A_3}(\alpha,\psi)(x)$ with $m=0,1$ and one gets, 
$$
\Vert R^{A_3}(\alpha,\psi)(x)\Vert\leq \frac2{\varepsilon^2}(C(1)+\frac{C(0)r_0}{\varepsilon}),
$$
and since $\displaystyle\frac{r_0}{\varepsilon}=O(1)$, one finally deduces that 
\begin{equation}%\label{est-R1a3}
R^{A_3}(\alpha,\psi)(x)=O(\frac{1}{\varepsilon^2}).
\end{equation}\medskip

 For $A_4(\alpha,\psi)(x)$, we only need to estimate the following term:
 \begin{equation*}
R^{A_4}(\alpha,\psi)(x):=\Big\langle~\textrm{p.v.}\int_{\R^2}\frac{\alpha_{\varepsilon,\eta_0}(\eta)}{\vert\eta\vert^3}\eta d\eta,(\nabla\psi-\nabla^T\psi(x))\n_x\Big\rangle\n_x.
\end{equation*}
Using Lemma \ref{miracle2}, one gets
\begin{equation}
R^{A_4}(\alpha,\psi)(x)=\frac{e^{-\br^2}}{\varepsilon^2}M_3^{A_1}(\br)\langle(\nabla\psi-\nabla^T\psi(x))\bar{\eta}_0,\n_x\rangle\n_x=O(\frac{1}{\varepsilon^2}).
\end{equation}\medskip

 For $A_5(\alpha,\psi)(x)$, one gets the estimate
\begin{equation}\label{est-R2a3}
R^{A_5}(\alpha,\psi)(x)=O(\frac{1}{\varepsilon^2}),
\end{equation}\medskip
as a consequence of Lemma \ref{super0}.

In summary, for $i=2,\dots, 5$, $P_x(\displaystyle A_i(\alpha,\psi)(x))=O(\frac{1}{\varepsilon^2})$, which ends the proof of Proposition \ref{est-others}.

\end{proof}

\begin{lemma}\label{super0}
With the notations above, consider the function defined for $x\in\partial\Omega$
$$
R(x)=\int_{\partial\Omega}r(x,y)\cdot \nabla(\alpha\psi)(y)d\sigma(\y),
$$
where $r(\cdot,\cdot)$ is a $C^3_*(1)$ weakly singular kernel and $\cdot$ stands for a linear action of $r$ on the coefficients of $\nabla(\alpha\psi)(\cdot)$.
Then, there exists a positive constant $C_R$ such that, for $\varepsilon>0$ small enough
and $x\in\partial\Omega$, one gets
\begin{equation}\label{estiM}
\Vert P_x(R(x))\Vert\leq \frac{C_R}{\varepsilon^2}.
\end{equation}
\end{lemma}
\begin{proof}[Proof of Lemma \ref{super0}] As done for estimating $A_1$, we use the change of variables introduced in Subsection \ref{choice-Vn} and taking into account Lemma \ref{grad-V} and Remark \ref{jacob}, it is easy to see that the most `singular'' part corresponds to majorizing 
$$
\int_{B(0,\delta)}\frac{\nabla\alpha_\eps(\eta)}{\vert \eta\vert}d\eta.
$$
Thus Eq. \eqref{estiM} follows readily from Lemma \ref{miracle1}.

\end{proof}

\subsection{Estimates of the remainder terms  $W_i^j$, $i=2,3,4$ and $j=1,2$.}\label{remainder0}
In this subsection, we upper bound the remainder terms $P_x(W_i^j(x))$,  $i=2,3,4$ and $j=1,2$, defined respectively in Eqs. (\ref{w2-0}), (\ref{w3-0}) and (\ref{w2-0}).
More precisely, we prove that
\begin{prop}\label{term-w2-4}With the notations above, we have, for $\eps>0$ small enough,
$j=1,2$ and  $i=2,3,4$,
\begin{equation}\label{est-W2-4}
P_x(W_i^j(x))=O(\frac{1}{\varepsilon^2}).
\end{equation}
\end{prop}

\begin{rem}\label{DH0}
One must stress the similarity of our computations with those performed by D. Henry in \cite{DH}. More precisely, the terms $A_1$ and $A_2$ in $W_1^j(\cdot)$, which are (essentially) the most `singular'' part in the hypersingular operator $E$ defined in Eq. \eqref{decompE}, correspond to the operator $J(\cdot)$ defined in Theorem 7.4.1 of \cite{DH}, page 135, with the specific choice of $\displaystyle Q(x,y,\frac{y-x}{\vert y-x\vert})=V_n(y)\frac{y-x}{\vert y-x\vert}$ and $n=3$. Also notice that our Lemma \ref{miracle2} corresponds to an explicit computation of the polynomial $q(\cdot)$ (cf. Theorem 7.4.1 of  \cite{DH}) and follows the same lines as the strategy proposed in page 137 in \cite{DH}. In particular, one gets from Theorem 7.4.1 of \cite{DH}  that $W_1^j(\cdot)$ extends uniquely to a continuous operator on $\partial\Omega$. 
\end{rem}

\begin{proof}[Proof of Proposition \ref{term-w2-4}] All the estimates to be established are consequences of \eqref{EE1}-\eqref{EE4} obtained in Corollary \ref{coro:hsiao}. We rewrite it
as follows. For  $\uu$ of class $C^2$ and $x\in \partial\Omega$, one writes $4\pi \E\uu(x)$ as the sum of two operators,
\begin{equation}\label{remainB}
4\pi \E\uu(x)=F\uu(x)+L\uu(x)=\textrm{p.v.} \int_{\partial\Omega}f(x,y)\cdot\nabla\uu(y)d\sigma(\y)+\int_{\partial\Omega}l(x,y)\cdot\nabla\uu(y)d\sigma(\y),
\end{equation}
where `` $\cdot$ " stands for an action of the respective kernels which is linear with respect to $\nabla\uu(\cdot)$, $l(\cdot,\cdot)$ is a $C^3_*(1)$ kernel (of appropriate matrix size) 
defined in Appendix \ref{WSO} and the kernel $f(\cdot,\dot)$ defining the singular operator $F$ together with its action is given by 
\begin{equation}\label{kernel-f}
f(x,y)\cdot M(y):=[M(y)+M^T(y)]\frac{\x-\y}{\mid \x-\y\mid^3}+
n_xn_x^T[M(y)-M^T(y)]\frac{\x-\y}{\mid \x-\y\mid^3},
\end{equation}
for $x\neq y$, points on $\partial\Omega$ and $M$ is  $C^1$ matrix-valued function defined on $\partial\Omega$. We are then only interested in the first term of the above sum.

In order to handle the remainder terms $P_x(W_i^j)$'s, $i=2,3,4$, one must handle the evaluation at $V_n\frac{\partial\phi_j}{\partial\n}$ of the operators obtained as the composition of ${{K}}^{\lambda}_\Omega$ defined in \eqref{k-lamb} and its iterations with $W_1^j$. In fact, we will show next that all remainder terms $P_x(W_i^j)$'s, $i=2,3,4$ are $O(\frac1{\varepsilon^2})$
and to proceed, we will be only interested in the contribution of the ``most'' singular part in each term $W_i^j$'s, $i=2,3,4$. For that purpose, we will perform several (and standard) reductions. The first one consists in considering the operator ${{K}}_\Omega^0$ instead of ${{K}}_\Omega^{\lambda}$ since 
${{K}}_\Omega^{\lambda}-{{K}}_\Omega^0$ admits a $C^1$ kernel. Lemma \ref{super0} already handles the term $W_3^j$. Next, recall ${{K}}_\Omega^0$ is a weakly singular operator of class $C^3_*(1)$ (see Appendix \ref{WSO} for a definition). 
To handle the terms  $W_2^j$ and $W_4^j$, we first need the following result.
\begin{lemma}\label{compo0}
The operator defined on $C^1(\partial\Omega)$ as the composition of ${{K}}_\Omega^0$ and $F$ is a weakly singular operator of class $C^3_*(1)$. 
\end{lemma}

Thanks to the above lemma, the first term in the summation \eqref{w2-0} is controlled as $O(\frac1{\varepsilon^2})$. For the other terms, it is now enough to see that they correspond to compositions of iterates of ${{K}}_\Omega^0$ with  ${{K}}_\Omega^0\circ F$ and thus we can apply Theorem \ref{th33} given below on the composition of weakly singular operators of class 
$C^3_*(\gamma)$ with $\gamma>0$. We deduce at once that every term appearing in the summation  \eqref{w2-0}  corresponds to the evaluation at $\nabla(V_n\frac{d\phi_j}{d\n})(\cdot)$ of a weakly singular operator of class $C^3_*(\gamma)$, with $\gamma\geq 1$, and is therefore controlled as $O(\frac1{\varepsilon^2})$. The term $W_3^j$ is handled in a similar way and Proposition \ref{term-w2-4} is established.

\end{proof}

We now give the proof of Lemma \ref{compo0}.
\begin{proof}[Proof of Lemma \ref{compo0}] The argument given below is already contained in Section $7.6$ of \cite{DH}, which considers a more general situation (see, more particularly, the proof of Theorem $7.6.3$ page $147$, \cite{DH}). For sake of clarity, we reproduce the main lines. Let $M$ be a $C^1$ matrix-valued function defined on $\partial\Omega$. Then, the composition $({{K}}_\Omega^0\circ F)[M](\cdot)$ is defined, for $x\in \partial\Omega$, as the sum 
$$
( {{K}}_\Omega^0\circ F)[M](x)=R_1(x)+R_2(x),
$$
where 
\begin{equation}\label{R1}
R_1(x)=\frac{3}{4\pi}~\textrm{p.v.}\iint_{\partial\Omega\times\partial\Omega} \frac{\langle \x-\z,{\n}(\z)\rangle}{\mid \z-\x\mid^5}~( \x-\z)( \x-\z)^T
[M(y)+M^T(y)]\frac{\z-\y}{\mid \z-\y\mid^3}~d\sigma_y~d\sigma_z,
\end{equation}
and 
\begin{equation}\label{R2}
R_2(x)=\frac{3}{4\pi}~\textrm{p.v.}\iint_{\partial\Omega\times\partial\Omega} \frac{\langle \x-\z,{\n}(\z)\rangle^2}{\mid \z-\x\mid^4}~\frac{(\x-\z)}{\mid \z-\x\mid}\frac{\langle n(\z),[M(y)-M^T(y)](\z-\y)\rangle}{\mid \z-\y\mid^3}~d\sigma_y~d\sigma_z.
\end{equation}
Thanks to \eqref{geo-diff2}, the operator $R_2$ is clearly more regular than $R_1$. In the sequel, we only provide details for $R_1$ and only give the estimate for $R_2$. 

 We next develop in coordinates the above expressions and obtain that, for $i=1,2,3$,
 \begin{eqnarray}
\Big(R_1(x)\Big)_{i}&=&\frac{3}{4\pi}\sum_{k,l=1}^3\textrm{p.v.}\int_{\partial\Omega}(M(y))_{kl}~d\sigma_y\nonumber\\
&& \int_{\partial\Omega}
 \Big[\frac{ \langle \x-\z,{\n}(\z)\rangle (\x-\z)_i( \x-\z)_k}
 {\mid \z-\x\mid^5}\frac{(\z-\y)_l}{\mid \z-\y\mid^3}\nonumber\\
&& + \frac{ \langle \x-\z,{\n}(\z)\rangle (\x-\z)_i( \x-\z)_l}
 {\mid \z-\x\mid^5}\frac{(\z-\y)_k}{\mid \z-\y\mid^3}\Big]~d\sigma_z.\label{R1-1}
\end{eqnarray}
The integrand of \eqref{R1-1} shows that $R_1$ is the contraction of $M(\cdot)$ and a tensor field of order $(2,1)$ defined (in coordinates) by the interior integral in \eqref{R1-1}. In order to describe $R_1$ as a convolution, we prefer to rewrite \eqref{R1-1} in a more elementary way, as follows, 
$$
\Big(R_1(x)\Big)_{i}=\frac{3}{4\pi}\textrm{p.v.}\int_{\partial\Omega}\hbox{Tr}(M(y)c_i(x,y))~d\sigma_y,
$$
where the kernel $c_i(x,y)$ is defined for $x\neq y$ and $1\leq i\leq 3$, as 
\begin{equation}\label{ker-c_i}
c_i(x,y):=\textrm{p.v.}\int_{\partial\Omega}\frac{ \langle \x-\z,{\n}(\z)\rangle  (\x-\z)_i}{\mid \z-\x\mid^5}
 \Big[\frac{(x-z)(z-y)^T}{\mid \z-\y\mid^3}+\frac{(z-y)(x-z)^T}{\mid \z-\y\mid^3}\Big]~d\sigma_z.
\end{equation}
Let $(e_i)_{1\leq i\leq 3}$ be the canonical basis of $\R^3$. Then one has $c_i(x,y)=d_i(x,y)+d_i(x,y)^T$, where 
\begin{equation}\label{coeur}
d_i(x,y):=\textrm{p.v.} \int_{\partial\Omega} k^0(x,z)[ g^i(z,y)]~d\sigma_z,
\end{equation}
i.e., $d_i(x,y)$ is the kernel corresponding to the convolution of 
$ {{K}}_\Omega^0$ with kernel $k^0(\cdot,\cdot)$ given by
$$
k^0(x,y):=\frac1{\mid \x-\y\mid}\frac{\langle \x-\y,{\n}(\y)\rangle}{\mid \x-\y\mid^2}~\frac{( \x-\y)}{\mid \x-\y\mid}\frac{( \x-\y)^T}{\mid \x-\y\mid},
$$
and the singular operator $G^i$ with kernel $g^i(\cdot,\cdot)$ given by
$$
g^i(x,y):=\frac{ e_i(\x-\y)^T}{\mid \x-\y\mid^3}.
$$

To perform that analysis, one writes \eqref{coeur} in the chart $h_\x$ defined in \eqref{HX} and only considers the most ``singular'' term of the composition, which is given by
\begin{equation}\label{in-chart}
\textrm{p.v.} \int_{B(0,\delta)} \frac{\eta^TK_x\eta}{\mid \eta\mid^5}\eta\eta^T\frac{\tilde{e_i}(\eta-\eta_\y)^T}{\mid\eta-\eta_\y\mid^3}d\eta.
\end{equation}
Here, $\tilde{e_i}$ is the orthogonal projection of $e_i$ onto $T_x\partial\Omega$. In \eqref{in-chart}, one clearly recognizes the convolution between the kernels $\displaystyle\frac{\eta^TK_x\eta}{\mid \eta\mid^5}\eta\eta^T$ and $\displaystyle\frac{\tilde{e_i}\eta^T}{\mid\eta\mid^3}$. 
The first kernel can also be written as $\frac1{\mid \eta\mid}Q(\frac{\eta}{\mid \eta\mid})$ where the components of $Q$ are homogeneous polynomials of degree four 
defined on $S^1$. According to \cite[Th. $7.3.1$ p. $128$]{DH} (which refers to \cite{Stein} for more complete computations), the Fourier transforms of these kernels are respectively equal to 
$$
F.T.(\frac{\eta^TK_x\eta}{\mid \eta\mid^5}\eta\eta^T)(\xi)=\frac{1}{\mid \xi\mid}\tilde{Q}(\frac{\xi}{\mid \xi\mid}),
$$
and 
$$
F.T.(\frac{\eta}{\mid \eta\mid^3})(\xi)=\gamma_1\frac{\xi}{\mid \xi\mid},
$$
where $\gamma_1$ is a positive constant and the components of $\tilde{Q}$ are homogeneous polynomials of degree four. We get that the Fourier transform of 
the operator whose kernel is given by \eqref{in-chart} is equal the product of the two Fourier transforms written previously and, as a consequence, that operator is weakly singular of class $C^3_*(1)$. The same conclusion holds true as well for $R_1$. A similar line of reasoning shows that $R_2$  is weakly singular of class $C^3_*(2)$ and Lemma \ref{super0} is finally proved.

\end{proof}

%======================================== 
%Foias-Saut
% ======================================================

\section{Proof of Theorem \ref{main-theo-2}}\label{sec:proof-FS}
In this section, we establish in full generality the Foias-Saut conjecture in 3D as stated in  \cite{Foias:1987dq}. First of all, notice that there is a countable number of resonance relations as defined in Definition \ref{def:reso}. To see that, simply remark that, for every positive integer $N$, there exists a finite number of resonance relations of the type $\lambda_k=\sum_{j=1}^l m_j\lambda_j$, with $\lambda_1\leq \cdots\leq\lambda_l\leq \lambda_k$, so that $k+\sum_{j=1}^l m_j\leq N$. We use $(RR)_n$, $n\geq 1$, to denote these resonances relations. 
%
%Fix a domain $\Omega_0\in\D^3_\ell$. We define, for $n\in\N$, the sets $$A_0:=W^{4,\infty}(\Omega_0,\R^3),$$
%and, for $n\geq 1$,
%$$
%A_n:=\{u\in W^{4,\infty}(\Omega_0,\R^3),~\textrm{the } n \textrm{ first  resonance relations $(RR)_j$, $1\leq j\leq n$, are not satisfied}\}.
%$$ 
%Set $A:=\bigcap_{l\in\N}A_n$ and note that 
%$$
%A=\{u\in W^{4,\infty}(\Omega_0,\R^3),~\textrm{ $(SD_{\Omega_0+u})$ is not resonant }\}.
%$$
%

Fix a domain $\Omega_0\in\D^3_\ell$  with $\ell\geq 5$. We define, for $n\in\N$, the sets 
$$
A_0:=\D^3_\ell(\Omega_0),
$$
and, for $n\geq 1$,
\begin{eqnarray*}
&A_n:=\{\Omega_0+u,\ u\in W^{\ell,\infty}(\Omega_0,\R^3),\ \Omega_0+u\in A_0\\
~\textrm{and the}& n \textrm{ first resonance relations $(RR)_j$, $1\leq j\leq n$, are not satisfied}\}.
\end{eqnarray*}
Set $A:=\bigcap_{l\in\N}A_n$. Note that 
$$
A=\{\Omega_0+u,\  u\in W^{\ell,\infty}(\Omega_0,\R^3),\ \Omega_0+u\in A_0~\textrm{ $(SD_{\Omega_0+u})$ is not resonant}\}.
$$

For $n\geq 0$, each set $A_n$ is open and one must show that $A_{n+1}$ is dense in $A_n$.
Reasoning by contradiction, assume that there exists $n\in \N$ so that  $A_{n+1}$ is not dense in $A_n$ and fix $(RR)_{n+1}$ to be equal to $\lambda_k=\sum_{j=1}^l m_j\lambda_j$, for some integers $k,l,m_1\cdots,m_l$. With no loss of generality, we assume that there exists 
$\Omega\in\D^3_\ell$ and $\varepsilon>0$ so that, for $u\in W^{\ell,\infty}$ with $\Vert u\Vert_{\ell,\infty}<\varepsilon$, we have
\begin{itemize}
\item[(i)] the $k$ first eigenvalues $\lambda_1(u),\dots,\lambda_k(u)$ of $(SD_{\Omega+u})$ are simple;
\item[(ii)] the resonance condition holds true:
\begin{equation}\label{resonance-u}
\lambda_k(u)=\sum_{j=1}^l m_j\lambda_j(u). 
\end{equation}
\end{itemize}
By Condition (i) and Eq. \eqref{id_m}, one has, for $u\in W^{\ell,\infty}$ with $\Vert u\Vert_{\ell,\infty}<\varepsilon$ and $1\leq j\leq k$,
\begin{equation}\label{resonance-deriv}
\lambda'_j(u)=-\int_{\partial\Omega}\langle u,n\rangle\Vert \frac{\partial\phi_j}{\partial n}\Vert^2,
\end{equation}
where $\phi_j$ is the orthonormal eigenfunction associated to the eigenvalue $\lambda_j$
of $(SD_{\Omega})$. 

Taking the shape derivative of Eq. \eqref{resonance-u}, we have
\begin{equation}\label{resonance-deriv2}
\int_{\partial\Omega}\langle u,n\rangle\Vert \frac{\partial\phi_k}{\partial n}\Vert^2=\int_{\partial\Omega}\langle u,n\rangle\sum_{j=1}^l m_j\Vert \frac{\partial\phi_j}{\partial n}\Vert^2. 
\end{equation}

Since Eq. \eqref{resonance-deriv2} holds true for all $u$ small enough, we obtain
\begin{equation}\label{resonance-deriv3}
\Vert \frac{\partial\phi_k}{\partial n}\Vert^2-\sum_{j=1}^l m_j\Vert \frac{\partial\phi_j}{\partial n}\Vert^2=0 \quad \textrm{ on }~\partial\Omega.
\end{equation}

Continuing the argument by contradiction, we assume that Eq. \eqref{resonance-deriv3} holds true in a neighborhood of $\Omega$ and take again the shape derivative. By Proposition \ref{shape-derivative}, we have, on $\partial\Omega$,
\begin{eqnarray}
\Big\langle \frac{\partial \phi'_k}{\partial \nu}, \frac{\partial \phi_k}{\partial n}\Big\rangle-\sum_{j=1}^l m_j\Big\langle \frac{\partial \phi'_j}{\partial \nu}, \frac{\partial \phi_j}{\partial n}\Big\rangle
=-\langle u,n\rangle\Big[\langle\frac{\partial}{\partial n}\frac{\partial\phi_k}{\partial N}, \frac{\partial\phi_k}{\partial n}\rangle-\sum_{j=1}^l m_j\langle\frac{\partial}{\partial n}\frac{\partial\phi_j}{\partial N}, \frac{\partial\phi_j}{\partial n}\rangle\Big].
\end{eqnarray}

We choose a variation $u$ such that $\langle u,n\rangle=V_n$ with $V_n$ defined in Section \ref{choice-Vn}. Using Proposition \ref{expansion} together with Eq. \eqref{resonance-deriv3}, since $\bar{\eta}_0$ is an arbitrary unitary vector of $\R^2$, we obtain,  on $\partial\Omega$,
\begin{equation}\label{deriv4}
\frac{\partial\phi_k}{\partial n}\Big(\frac{\partial\phi_k}{\partial n}\Big)^T-\sum_{j=1}^l m_j \frac{\partial\phi_j}{\partial n}\Big(\frac{\partial\phi_j}{\partial n}\Big)^T=0.
\end{equation}

From now on, fix $x\in\partial\Omega$ such that $\displaystyle \frac{\partial\phi_k}{\partial n}(x)\neq 0$. Recall that such an $x$ exists by the result of Osses in \cite{Osses}. According to Eq. \eqref{deriv4}, there exists an open neighborhood $O_x$ of $x$ on $\partial\Omega$ such that, for $1\leq j\leq l$, there is a $C^2$ function $\mu_j$ such that 
\begin{equation}\label{eq:mu1}
\frac{\partial\phi_j}{\partial n}=\mu_j\frac{\partial\phi_k}{\partial n},\quad \textrm{ on }O_x.
\end{equation} 
In addition, one has,  
\begin{equation}\label{eq:mu2}
1-\sum_{j=1}^l m_j\mu_j^2=0,\quad \textrm{ on }O_x.
\end{equation}
It is clear that all the equations from (\ref{resonance-deriv}) to (\ref{eq:mu2}) were obtain by assuming that Eq. (\ref{resonance-u}) holds true in an open neighborhood of $u=0$. As a consequence, these equations must also hold true in an open neighborhood of $u=0$ and thus,
one can take the shape derivatives of Equations (\ref{deriv4}) at $u=0$ along any variation. 
We will perform such a shape derivation along the variations $V_n$ defined in Section \ref{choice-Vn}, with this time the real number $\delta>0$ chosen so that the support of $V_n$ is contained in $O_x$. Using Lemma \ref{shap_der}, the shape derivative of Eq. \eqref{deriv4} is equal to
\begin{eqnarray}
&&\frac{\partial\phi_k'}{\partial\nu}\Big(\frac{\partial\phi_k}{\partial n}\Big)^T+\frac{\partial\phi_k}{\partial n}\Big(\frac{\partial\phi_k'}{\partial\nu}\Big)^T-\sum_{j=1}^l m_j \Big[\frac{\partial\phi_j'}{\partial\nu}\Big(\frac{\partial\phi_j}{\partial n}\Big)^T+\frac{\partial\phi_j}{\partial n}\Big(\frac{\partial\phi_j'}{\partial\nu}\Big)^T\Big]\label{deriv5}\\
&+&\Big(p_k'+\langle\frac{\partial\phi_k}{\partial n}, n'\rangle\Big)\Big[n\Big(\frac{\partial\phi_k}{\partial n}\Big)^T+\frac{\partial\phi_k}{\partial n}n^T\Big]-\sum_{j=1}^l m_j\Big(p_j'+\langle\frac{\partial\phi_j}{\partial n}, n'\rangle\Big)\Big[n\Big(\frac{\partial\phi_j}{\partial n}\Big)^T+\frac{\partial\phi_j}{\partial n}n^T\Big]\nonumber\\
&=&-V_n\Big(\frac{\partial}{\partial n}\frac{\partial\phi_k}{\partial N}\Big(\frac{\partial\phi_k}{\partial n}\Big)^T+\frac{\partial\phi_k}{\partial n}\Big(\frac{\partial}{\partial n}\frac{\partial\phi_k}{\partial N}\Big)^T-\sum_{j=1}^l m_j\mu_j\Big[\frac{\partial}{\partial n}\frac{\partial\phi_j}{\partial N}\Big(\frac{\partial\phi_k}{\partial n}\Big)^T+\frac{\partial\phi_k}{\partial n}\Big(\frac{\partial}{\partial n}\frac{\partial\phi_j}{\partial N}\Big)^T\Big]\Big)\nonumber,
\end{eqnarray}
where the above equation holds on $\partial\Omega$. 

Moreover, on $O_x$, one deduces that 
\begin{eqnarray*}
\frac{\partial}{\partial n}\frac{\partial\phi_j}{\partial n}&=&\nabla (\mu_j(\nabla\phi_kn))n=\frac{\partial \mu_j}{\partial n}\frac{\partial\phi_k}{\partial n}+\mu_j\nabla^2\phi_k(n,n).\\
\frac{\partial}{\partial n}\nabla^T\phi_j n&=&\nabla(\mu_j(\nabla^T\phi_kn))n=\frac{\partial\mu_j}{\partial n}\nabla^T\phi_kn+\mu_j\nabla(\nabla^T\phi_kn)n=\mu_j\nabla(\nabla^T\phi_kn)n.
\end{eqnarray*}This implies that, on $O_x$,
\begin{equation}
\frac{\partial}{\partial n}\frac{\partial\phi_j}{\partial N}=\frac{\partial \mu_j}{\partial n}\frac{\partial\phi_k}{\partial n}+\mu_jv_k,
\end{equation}with $v_k:=\nabla^2\phi_k(n,n)+\nabla(\nabla^T\phi_kn)n$. Therefore, one has on $O_x$,
\begin{eqnarray}
&&\frac{\partial}{\partial n}\frac{\partial\phi_k}{\partial N}\Big(\frac{\partial\phi_k}{\partial n}\Big)^T+\frac{\partial\phi_k}{\partial n}\Big(\frac{\partial}{\partial n}\frac{\partial\phi_k}{\partial N}\Big)^T-\sum_{j=1}^l m_j\mu_j\Big[\frac{\partial}{\partial n}\frac{\partial\phi_j}{\partial N}\Big(\frac{\partial\phi_k}{\partial n}\Big)^T+\frac{\partial\phi_k}{\partial n}\Big(\frac{\partial}{\partial n}\frac{\partial\phi_j}{\partial N}\Big)^T\Big]\nonumber\\
&=&(1-\sum_{j=1}^l m_j\mu_j^2)\big[v_k\Big(\frac{\partial\phi_k}{\partial n}\Big)^T+\frac{\partial\phi_k}{\partial n}v_k^T\big]-2\sum_{j=1}^l m_j\mu_j\frac{\partial\mu_j}{\partial n}\frac{\partial\phi_k}{\partial n}\Big(\frac{\partial\phi_k}{\partial n}\Big)^T\nonumber\\
&=&-2\sum_{j=1}^l m_j\mu_j\frac{\partial\mu_j}{\partial n}\frac{\partial\phi_k}{\partial n}\Big(\frac{\partial\phi_k}{\partial n}\Big)^T.\label{vn-term}
\end{eqnarray}

Plugging Eqs. \eqref{eq:mu1}, \eqref{eq:mu2}, and \eqref{vn-term} into Eq. \eqref{deriv5}, we obtain on $O_x$ that
\begin{eqnarray}
&&\frac{\partial\phi_k'}{\partial\nu}\Big(\frac{\partial\phi_k}{\partial n}\Big)^T+\frac{\partial\phi_k}{\partial n}\Big(\frac{\partial\phi_k'}{\partial\nu}\Big)^T-\sum_{j=1}^l m_j\mu_j \Big[\frac{\partial\phi_j'}{\partial\nu}\Big(\frac{\partial\phi_k}{\partial n}\Big)^T+\frac{\partial\phi_k}{\partial n}\Big(\frac{\partial\phi_j'}{\partial\nu}\Big)^T\Big]\nonumber\\
&&+\Big(p_k'-\sum_{j=1}^l m_j\mu_j p'_j\Big)\Big[n\Big(\frac{\partial\phi_k}{\partial n}\Big)^T+\frac{\partial\phi_k}{\partial n}n^T\Big]\nonumber\\
&=&2V_n\sum_{j=1}^l m_j\mu_j\frac{\partial\mu_j}{\partial n}\frac{\partial\phi_k}{\partial n}\Big(\frac{\partial\phi_k}{\partial n}\Big)^T.\label{deriv6}
\end{eqnarray}

On $O_x$, set 
\begin{equation}\label{eq:dk0}
d_k:=\sum_{j=1}^l m_j\mu_j(\mu_j\frac{\partial\phi_k'}{\partial\nu}-\frac{\partial\phi_j'}{\partial\nu}).
\end{equation} \smallskip

The main part of the rest of the argument consists in deriving the main term of the asymptotic expansion of $P_x(d_k)$ at $x$, in terms of the powers of $\displaystyle\frac{1}{\eps}$. The first step will be to establish that $P_x(d_k(x))=O(\displaystyle\frac{1}{\eps})$ and, in the second step, we will compute precisely the coefficient $d^1$ defined as 
$$
P_x(d_k(x))=\displaystyle\frac{d^1}{\eps}+O(1),
$$where the coefficient $d^1$ will depend on the parameters involved in the special variation $V_n$. Once this is performed, we will resume the contradiction argument using the information contained in $d^1$. 

To prepare these computations, we first rewrite Eq. \eqref{deriv6} using again 
Eq. \eqref{eq:mu2} as 
\begin{eqnarray}\label{eq:dk}
&&d_k\Big(\frac{\partial\phi_k}{\partial n}\Big)^T+\frac{\partial\phi_k}{\partial n}d_k^T+\Big(p_k'-\sum_{j=1}^l m_j\mu_j p'_j\Big)\Big[n\Big(\frac{\partial\phi_k}{\partial n}\Big)^T+\frac{\partial\phi_k}{\partial n}n^T\Big]\nonumber\\
&=&2\big[\sum_{j=1}^l m_j\mu_j\frac{\partial\mu_j}{\partial n}\big]V_n\frac{\partial\phi_k}{\partial n}\Big(\frac{\partial\phi_k}{\partial n}\Big)^T\label{deriv7}.
\end{eqnarray}

Multiplying Eq. \eqref{eq:dk} from the left by $\displaystyle \big(\frac{\partial\phi_k}{\partial n}\big)^T$ and from the right by $\displaystyle\frac{\partial\phi_k}{\partial n}$, we obtain the following scalar equation which will be used to achieve a contradiction.
\begin{equation}\label{eq:dk2}
\langle\frac{\partial\phi_k}{\partial n},d_k\rangle=\big[\sum_{j=1}^l m_j\mu_j\frac{\partial\mu_j}{\partial n}\big]V_n.
\end{equation}

%\begin{proof}
%\begin{eqnarray*}
%\frac{\partial \Delta_{ij}(x-y)}{\partial N(y)}&=&-2\Big(\delta_{ij}\frac{x_l-y_l}{\vert x-y\vert}+\frac{(x_i-y_i)(x_j-y_j)(x_l-y_l)}{\vert x-y\vert^3}+\delta_{il}\frac{x_j-y_j}{\vert x-y\vert}-\delta_{jl}\frac{x_i-y_i}{\vert x-y\vert}\Big)n_l(y)\\
%&=&-2\Big(\delta_{ij}\frac{\langle x-y, n(y)\rangle}{\vert x-y\vert}+\frac{(x_i-y_i)(x_j-y_j)\langle x-y,n(y)\rangle}{\vert x-y\vert^3}\\
%&&+\frac{(x_j-y_j)n_i(y)}{\vert x-y\vert}-\frac{(x_i-y_i)n_j(y)}{\vert x-y\vert}\Big).
%\end{eqnarray*}
%
%\begin{eqnarray*}
%\frac{\partial^2\Delta_{ij}(x-y)}{\partial N(x)\partial N(y)}
%\end{eqnarray*}
%
%Taking Eq. \eqref{geo-diff2} into account, it is straightforward that
%
%\end{proof}

We now prove the following lemma.
\begin{lemma}\label{le:dk1}
With the notations above, one has  $P_x(d_k(x))=O(\displaystyle\frac{1}{\eps})$.
\end{lemma}
\begin{proof}[Proof of Lemma \ref{le:dk1}]
Set $\displaystyle\psi:=\frac{\partial \phi_k}{\partial n}$ and, for $y\in O_x$,  
\begin{equation}\label{est:beta}
 \beta(y):=\sum_{j=1}^l m_j\mu_j(x)(\mu_j(y)-\mu_j(x)).
\end{equation}
Then, one has $\beta(y)=O(\vert x-y\vert^2)$. More precisely, if we use the parameterization defined in Eq. \eqref{HX}, we obtain (in local coordinates)
\begin{equation}\label{est:beta2}
\beta(y)=\frac{1}{2}\Big(\frac{\partial \beta}{\partial n}(x)\eta^TK_x\eta+\eta^TH_x\eta\Big)+O(\vert \eta\vert^3):=\frac{1}{2}\eta^T F_x\eta+O(\vert\eta\vert^3),
\end{equation}where $H_x$ denotes the Hessian matrix of $\beta$ at $x$. Note that by taking twice tangent derivatives of Eq.\eqref{eq:mu2}, we know that $H_x$ is a negative semi-definite matrix. \smallskip

Consider now the representation formula of $d_k$ as described in Eq. \eqref{conormal_derivative0}. Note that, in $P_x(d_k)$, two contributions give rise to the term of order of $\displaystyle O(\frac{1}{\eps})$, namely these coming from $b^{(0)}$ and $e^{(\lambda)}$ respectively. 

The term corresponding to $b^{(0)}$ in that equation is equal to 
\begin{equation}\label{asym1}
P_x(E(\beta V_n\frac{\partial\phi_k}{\partial n})(x)).
\end{equation}

Thanks to the estimate of $\beta$ in \eqref{est:beta2}, it is clear, by proceeding as in Subsection \ref{remainder0}, that all the other terms of the representation formula of $P_x(d_k)$ are indeed of the type $O(\displaystyle\frac{1}{\eps})$. Therefore, 
one has only to determine the asymptotic expansion of the term given in Eq. \eqref{asym1}. According to Lemma \ref{singularity_couche}, it amounts now to estimate the five terms $P_x(A_i)$, $1\leq i\leq 5$, and after elementary or standard computations using systematically Eq. \eqref{est:beta2} , we obtain
\begin{equation}\label{eq:a11}
P_x(A_1(\alpha, \beta\psi)(x))=\frac{1}{4\pi\eps^2}\int_{\R^2}\frac{\alpha_\eps(\eta)}{\vert \eta\vert^3}\eta^TF_x\eta\Big(\langle\psi(x),\eta\rangle(\eta-\eta_0)+\langle\eta-\eta_0,\eta\rangle\psi(x)\Big)d\eta
+O(1),
\end{equation}
\begin{equation}\label{eq:a22}
P_x(A_2(\alpha,\beta\psi)(x))=-\frac{1}{4\pi}\int_{\R^2}\frac{\alpha_\eps(\eta)}{\vert \eta\vert^3}\Big(\psi(x)\eta^TF_x\eta+\langle\psi(x),\eta\rangle F_x\eta\Big)d\eta+O(1),
\end{equation}
and, for $3\leq j\leq 5$, $A_j(\alpha,\beta\psi)(x)=O(1)$. \smallskip

Let us now treat the terms given by $e^{(\lambda)}$. Note that the presence of these terms reflects the fact that $\phi_j$ and $\phi_k$ correspond to different eigenvalues of the Stokes operator. Using Lemma \ref{deltaNN}, we define the operator $\Delta^{\lambda}(\alpha,\psi)$ as follows
\begin{equation}\label{eq:e-lambda}
\Delta^{\lambda}(\alpha,\psi)(x):=-\frac{\lambda}{8\pi}\int_{\R^2}\frac{\alpha_\eps(\eta)\langle\psi(x),\eta/\vert\eta\vert\rangle}{\vert \eta\vert}\frac{\eta}{\vert\eta\vert}d\eta.
\end{equation}It is clear that $$\Delta^{\lambda}(\alpha,\psi)(x)=O(\frac{1}{\eps}).$$

Therefore, with above notations, we have
\begin{eqnarray}
P_x(d_k(x))&=&A_1(\alpha,\psi)(x)+A_2(\alpha,\psi)(x)+\sum_{j=1}^{l} m_j\mu_j(\mu_j\Delta^{\lambda_k}(\alpha,\psi)(x)-\Delta^{\lambda_j}(\alpha,\mu_j\psi)(x))+O(1)\nonumber\\
&=&O(\frac{1}{\eps}).\label{eq:e-lambda2}
\end{eqnarray}

%the following operators.except the two terms $e^{\lambda}$ given by 
%\begin{eqnarray}
%\Delta^{\lambda_k}(\psi)&:=&-\frac{\lambda_k}{8\pi}\int_{\R^2}\frac{\alpha_\eps(\eta)\langle\psi(x),\eta/\vert\eta\vert\rangle}{\vert \eta\vert}\frac{\eta}{\vert\eta\vert}d\eta,\label{eq:elk}\\
%\Delta^{\lambda_j}(\mu_j\psi)&:=&-\frac{\lambda_j}{8\pi}\mu_j(x)\int_{\R^2}\frac{\alpha_\eps(\eta)\langle\psi(x),\eta/\vert\eta\vert\rangle}{\vert \eta\vert}\frac{\eta}{\vert\eta\vert}d\eta.\label{eq:elj}
%\end{eqnarray}
%To derive the previous equations, we have applied Lemma \ref{deltaNN} to $e^{\lambda_k}$ and $e^{\lambda_j}$, for $1\leq j\leq l$. 
%
This ends the proof of Lemma \ref{le:dk1}.

\end{proof}

Let us now pursue the proof of Theorem \ref{main-theo-2}. Since the value of the right-hand side of Eq. \eqref{eq:dk2} at $x$ is given by the following expression $$\displaystyle \big[\sum_{j=1}^l m_j\mu_j(x)\frac{\partial\mu_j}{\partial n}(x)\big]\frac{e^{-\br^2}}{\eps^2}=O(\frac{1}{\eps^2}),$$ we conclude that
\begin{equation}
\sum_{j=1}^l m_j\mu_j(x)\frac{\partial\mu_j}{\partial n}(x)=0,\qquad\textrm{i.e., }\quad \frac{\partial\beta}{\partial n}(x)=0,\label{beta-n}
\end{equation}
which implies that 
\begin{equation}
\langle\frac{\partial \phi_k}{\partial n}(x),d_k(x)\rangle=0,\label{eq:dk3}
\end{equation}
%and the matrix $F_x$ defined in \eqref{est:beta2} verifies
%\begin{equation}\label{eq:fx}
%F_x =H_x\leq 0.
%\end{equation}

In order to get additional information from Eq. \eqref{eq:dk3}, we compute explicitly the numerical coefficient in front of $\displaystyle\frac{1}{\eps}$ in the asymptotic expansion of $P_x(d_k(x))$. It is enough to have a closer look at the representation formula of $P_x(d_k)$ as described in Eq. \eqref{conormal_derivative0}. From Eq. \eqref{eq:e-lambda2}, we have
\begin{equation}\label{eq:dk-expression}
P_x(d_k(x))=a_1+a_2+\rho a_3+O(1),
\end{equation}
where
\begin{eqnarray}
a_1&:=&\frac{1}{4\pi\eps^2}\int_{\R^2}\frac{\alpha_\eps(\eta)}{\vert \eta\vert^3}\eta^TF_x\eta\Big(\langle\psi(x),\eta\rangle(\eta-\eta_0)+\langle\eta-\eta_0,\eta\rangle\psi(x)\Big)d\eta,\label{eq:aa1}\\
a_2&:=&-\frac{1}{4\pi}\int_{\R^2}\frac{\alpha_\eps(\eta)}{\vert \eta\vert^3}\Big(\psi(x)\eta^TF_x\eta+\langle\psi(x),\eta\rangle F_x\eta\Big)d\eta,\label{eq:aa2}\\
a_3&:=&-\frac{1}{8\pi}\int_{\R^2}\frac{\alpha_\eps(\eta)\langle\psi(x),\eta/\vert\eta\vert\rangle}{\vert \eta\vert}\frac{\eta}{\vert\eta\vert}d\eta,\label{eq:aa3}\\
\rho&:=&\sum_{j=1}^l m_j\mu_j^2(\lambda_k-\lambda_j).
\end{eqnarray}
Notice that $\rho>0$ since $\lambda_k>\lambda_j$ for $1\leq j\leq l$ and at least one of the integers $m_j$ is positive.

We now compute the coefficients $a_i$, $1\leq i\leq 3$. For $\theta_0\in S^1$, we set 
$$R_{\theta_0}:=\begin{pmatrix}\cos\theta_0&-\sin\theta_0\\\sin\theta_0&\cos\theta_0\end{pmatrix}, \qquad F_{\theta_0}:=R_{\theta_0}^TF_xR_{\theta_0}:=(F_{\theta_0}^{ij})_{i,j=1,2}.
$$For $1\leq i\leq 10$, we define the functions $M_i$ as follows.
\begin{eqnarray}
M_6(z)&:=&\int_0^\infty e^{-r^2}r^2dr\int_0^{2\pi}e^{2rz\cos\theta}d\theta,\label{eq:m6}\\
M_7(z)&:=&\int_0^\infty e^{-r^2}r^2dr\int_0^{2\pi}e^{2rz\cos\theta}\cos^2\theta d\theta,\label{eq:m7}\\
M_8(z)&:=&\int_0^\infty e^{-r^2}r^2dr\int_0^{2\pi}e^{2rz\cos\theta}\cos^4\theta d\theta,\label{eq:m8}\\
M_9(z)&:=&\int_0^\infty e^{-r^2}rdr\int_0^{2\pi}e^{2rz\cos\theta}\cos\theta d\theta,\label{eq:m9}\\
M_{10}(z)&:=&\int_0^\infty e^{-r^2}rdr\int_0^{2\pi}e^{2rz\cos\theta}\cos^3\theta d\theta.\label{eq:m10}
\end{eqnarray}
The expressions of $a_1$, $a_2$, and $a_3$ are summarized in the following lemma whose proof is postponed in Section \ref{pf-calcul-final} of Appendix. 
\begin{lemma}\label{calcul-final}
We have
\begin{eqnarray}
a_1&=&\frac{1}{4\pi}\frac{e^{-\br^2}}{\eps}\Big(\nonumber\\
&&\big[2F^{22}_{\theta_0}M_6(\br)+(2F^{11}_{\theta_0}-3F^{22}_{\theta_0})M_7(\br)-(F^{11}_{\theta_0}-F^{22}_{\theta_0})(M_8(\br)+M_{10}(\br))-F^{22}_{\theta_0}M_9(\br)\big]\psi(x)\nonumber\\
&-&\big[\big(F^{22}_{\theta_0}M_9(\br)+(F^{11}_{\theta_0}-F^{22}_{\theta_0})(M_{10}(\br)-2M_8(\br))+(F^{11}_{\theta_0}-3F^{22}_{\theta_0})M_7(\br)\nonumber\\ 
&&+F^{22}_{\theta_0}M_6(\br)\big)\langle\psi(x),\bet\rangle
+2F^{12}_{\theta_0}(M_9(\br)-M_{10}(\br))\langle\psi(x),\bet^\perp\rangle\big]\bet\nonumber\\
&-&2 F_{\theta_0}^{12}(M_7(\br)-M_8(\br))\psi(x)^\perp+4 F_{\theta_0}^{12}(M_7(\br)-M_8(\br))\langle\psi(x),\bet\rangle\bet^\perp\Big),\\
a_2&=&-\frac{1}{4\pi}\frac{e^{-\br^2}}{\eps}\Big\{\Big[F_{\theta_0}^{22}M_5^{A_1}(\br)+(F_{\theta_0}^{11}-F_{\theta_0}^{22})M_1^{A_1}(\br)\Big]\psi(x)\nonumber\\
&&\quad+F_x\Big((M_5^{A_1}(\br)-M_1^{A_1}(\br))\psi(x)+(2M_1^{A_1}(\br)-M_5^{A_1}(\br))\langle\psi(x),\bet\rangle\bet\Big)\Big\},\\
a_3&=&-\frac{1}{8\pi}\frac{e^{-\br^2}}{\eps}\Big[ (M_5^{A_1}(\br)-M_1^{A_1}(\br))\psi(x)+(2M_1^{A_1}(\br)-M_5^{A_1}(\br))\langle\psi(x),\bet\rangle\bet\Big].
\end{eqnarray}
\end{lemma}\bigskip

Let us now finish the proof of Theorem \ref{main-theo-2}. We choose $\bet\perp\psi(x)$. Without loss of generality, we also assume that $\bet=(1,0)^T$ and $\displaystyle{\psi(x)}/{\vert\psi(x)\vert} =(0,1)^T$. Recall that we have chosen $x$ such that $\psi(x)\neq 0$. Then, we deduce from Eq. \eqref{eq:dk-expression} and Lemma \ref{calcul-final} that 
\begin{eqnarray}\label{eq:dk4}
d_k(x)=\frac{e^{-\br^2}}{4\pi\eps}\Big(\alpha_1\psi(x)+\alpha_2\psi(x)^\perp+\alpha_3F_x\psi(x)\Big),
\end{eqnarray}where,
\begin{eqnarray*}
\alpha_1&:=&2F^{22}_{x}M_6(\br)+(2F^{11}_{x}-3F^{22}_{x})M_7(\br)-(F^{11}_{x}-F^{22}_{x})(M_8(\br)+M_{10}(\br))-F^{22}_{x}M_9(\br)\\
&&-\big((F_{x}^{22}+\frac{\rho}{2})M_5^{A_1}(\br)+(F_{x}^{11}-F_{x}^{22}-\frac{\rho}{2})M_1^{A_1}(\br)\big),\\
\alpha_2&:=&2F^{12}_x(M_9(\br)+M_8(\br)-M_{10}(\br)-M_7(\br)),\\
\alpha_3&:=&-\big(M_5^{A_1}(\br)-M_1^{A_1}(\br)\big).
\end{eqnarray*}

Plugging Eq. \eqref{eq:dk4} into Eq. \eqref{eq:dk3}, we obtain
\begin{equation}\label{eq:al00}
\alpha_1+F_x^{22}\alpha_3=0.
\end{equation}

The final contradiction will be obtained by showing that the two non zero entire functions of $\br$ given by $\alpha_1$ and $\alpha_2$ cannot satisfy Eq. (\eqref{eq:al00}). For that purpose, we need to get more explicit expressions. Eq.  (\eqref{eq:al00}) writes
\begin{eqnarray}
0&=&2F^{22}_{x}M_6(\br)+(2F^{11}_{x}-3F^{22}_{x})M_7(\br)-(F^{11}_{x}-F^{22}_{x})(M_8(\br)+M_{10}(\br))-F^{22}_{x}M_9(\br)\nonumber\\
&&-\big((2F_{x}^{22}+\frac{\rho}{2})M_5^{A_1}(\br)+(F_{x}^{11}-2F_{x}^{22}-\frac{\rho}{2})M_1^{A_1}(\br)\big)\nonumber\\
&=&F_x^{11}\big(2M_7(\br)-M_8(\br)-M_{10}(\br)-M_1^{A_1}(\br)\big)\nonumber\\
&+&F_x^{22}\big(2M_6(\br)-3M_7(\br)+M_8(\br)+M_{10}(\br)-M_9(\br)-2M_5^{A_1}(\br)+2M_1^{A_1}(\br)\big)\nonumber\\
&+&\frac{\rho}{2}\big(M_1^{A_1}(\br)-M_5^{A_1}(\br)\big).
\label{eq:final1}
\end{eqnarray}

Since the right-hand side of Eq. \eqref{eq:final1} is an entire function, we deduce that all the coefficients in its power series expansion are equal to zero. We need the following lemma whose proof is deferred in Section \ref{pf-final-cal} of Appendix.

\begin{lemma}\label{final-cal}
The entire functions involved in Eq. \eqref{eq:final1} have the following power series expansions
\begin{eqnarray*}
&&2M_7(z)-M_8(z)-M_{10}(z)-M_1^{A_1}(z)\smallskip\\
&=&\sum_{p=0}^\infty \frac{2^{2p+1}}{(2p)!}\Gamma(p+\frac{1}{2})I_{2p+2}\frac{2p^2+4p-\frac{3}{2}}{2p+3}z^{2p}-\sum_{p=0}^{\infty}\frac{2^{2p+2}}{(2p+1)!}\Gamma(p+\frac{3}{2})I_{2p+4} ~z^{2p+1},\smallskip\\
&&2M_6(z)-3M_7(z)+M_8(z)+M_{10}(z)-M_9(z)-2M_5^{A_1}(z)+2M_1^{A_1}(z)\smallskip\\
&=&\sum_{p=0}^\infty \frac{2^{2p+1}}{(2p)!}\Gamma(p+\frac{1}{2})I_{2p}\frac{6p^2+16p+\frac{7}{2}}{(2p+2)(2p+4)}z^{2p}-\sum_{p=0}^\infty\frac{2^{2p+2}}{(2p+1)!}\Gamma(p+\frac{3}{2})I_{2p+2}\frac{1}{2p+4}z^{2p+1}.
%&&M_1^{A_1}(\br)-M_5^{A_1}(\br)\smallskip\\
%&=&-\sum_{p=0}^\infty\frac{2^{2p+1}}{(2p)!}\Gamma(p+\frac{1}{2})I_{2p}\frac{1}{2p+2}z^{2p}.
\end{eqnarray*}
\end{lemma}\bigskip

Using Lemma \ref{final-cal} and considering the coefficients of the odd powers of $z$ in the power series expansion of the right-hand side of Eq. \eqref{eq:final1}, we deduce that
\begin{equation}
F_x^{11}(2p+3)+F^{22}_x=0, \qquad \textrm{for all } p\in\N.
\end{equation}This implies that
\begin{equation}
F_x^{11}=F^{22}_x=0.
\end{equation}

Therefore, using Eq. \eqref{eq:final1}, we have
\begin{equation}\label{eq:final2}
\frac{\rho}{2}\big(M_1^{A_1}(\br)-M_5^{A_1}(\br)\big)=0.
\end{equation}Recall that  $\displaystyle M_5^{A_1}(z)-M_1^{A_1}(z)$ is equal to $M_2^{A_1}(z)$, then not identically equal to zero. Then Eq. \eqref{eq:final2} yields that $$\rho=0,$$ which is in contradiction with the fact that $$\rho=\sum_{j=1}^l m_j\mu_j^2(\lambda_k-\lambda_j)>0.$$
Theorem \ref{main-theo-2} is finally proved.
%============================= Appendix =========================

\appendix

\section{Layer potentials and representation formulas} 
Most of the material presented here is borrowed from \cite{ammari,Ammari:fk,lady,DH}. Let $\lambda$ be a non negative real number. Consider $\phi$ and $p$ satisfying the eigenvalue problem associated to the following Stokes system 
 $$
 \left\{
 \begin{array}{rll}
( \Delta +\lambda) \phi -\nabla p &=&h~\textrm{ in}~ \Omega\\ 
\text{div~} \phi &=&0~\textrm{ in}~ \Omega\\
\phi&=&g ~\textrm{ on } \partial\Omega\\
\displaystyle\int_\Omega  p&=& 0,  
\end{array}
\right.
 $$
 under the compatibility condition 
 \begin{equation}\label{compatibilty}
 \int_\Gamma  \phi\cdot n~ds =0,
 \end{equation}
 where $\n$ is the outward unit normal to $\partial\Omega$.  Recall that, for such a pair of fields, the  conormal derivative denoted by $
\displaystyle\frac{\partial \phi}{\partial \nu}$  was defined in \eqref{conor0}.

\subsection{Layer potentials}
We denote by $\partial_i$ the operator $\frac{\partial}{\partial x_i}$ and by $\sqrt{-\lambda}$ 
%$\kappa$ the real defined as $\kappa=\sqrt{\lambda}$. We use $\sqrt{-\lambda}$ to denote 
the complex number $i\sqrt{\lambda}$. In this section, we adopt the {Einstein summation convention} which omits the summation sign for the indices appearing twice.
\paragraph{Fundamental tensors} ~~We define the fundamental tensors ${{\Gamma}}^\lambda=(\Gamma_{ij}^{\lambda})_{i,j=1}^3$ and ${{F}}=(F_i)_{i=1}^3$ as 
\begin{equation}
\left\{
\begin{array}{lll}
\displaystyle \Gamma_{i,j}^{\lambda}&=&\displaystyle -\frac{\delta_{ij}e^{\sqrt{-\lambda}\vert x \vert}}{4\pi \vert x \vert}-\frac{1}{4\pi \lambda}\partial_i\partial_j \frac{ e^{\sqrt{-\lambda}\vert x\vert}-1 }{\vert x \vert},\\
F_i(x)&=&\displaystyle-\frac{x_i}{4\pi \vert x \vert^3}.
\end{array}
\right.
\end{equation}
In the sense of distributions,  straightforward computations of the fundamental solution of Helmholtz operator $\Delta + \lambda$ allow to get  
$$
(\Delta+\lambda)\Gamma_{ij}^{\lambda}-\partial_jF_i=\delta_{ij}\delta(x),\hbox{ and }
\partial_i\Gamma_{ij}^\lambda=0,
$$
where we use $\delta(x)$ to denote the delta distribution based at $x\in\R^3$.
The tensor  ${\Gamma}^0$, which is the fundamental tensor for the standard Stokes system, is defined as 
$$
\Gamma_{ij}^{0}(x):=-\frac{1}{8\pi}\left(\frac{\delta_{ij}}{\vert x \vert}  +\frac{x_ix_j}{\vert x \vert^3} \right),
$$
and one has, uniformly on compact subsets of $\R^3$,
\begin{equation}\label{eq:gamma-ij}
\Gamma_{ij}^{\lambda}(x)=\Gamma_{ij}^{0}(x)-\frac{\delta_{ij}\sqrt{-\lambda}}{6\pi} +O(\lambda).
\end{equation}
Denoting that
\begin{equation}\label{eq:delta}
\Delta_{ij}^{\lambda}(x):=\Gamma_{ij}^{\lambda}(x)-\Gamma_{ij}^{0}(x),
\end{equation}
we have
\begin{equation}\label{eq:delta2}
\Delta_{ij}^{\lambda}(x)=-\frac{\delta_{ij}\sqrt{-\lambda}}{6\pi}-\frac{\lambda}{32\pi}\Delta_{ij}(x)+O(\vert x\vert^2 ).
\end{equation} where $\Delta_{ij}(\cdot)$ is defined by
\begin{equation}\label{eq:delta3}
\Delta_{ij}(x):=3\delta_{ij}\vert x\vert-\frac{x_ix_j}{\vert x\vert}.
\end{equation}
After simple computations, one gets the following useful result.
\begin{lemma}\label{deltaNN}
We have
\begin{equation}
\frac{\partial^2\Delta^\lambda(x-y)}{\partial N(x)\partial N(y)}=-\frac{\lambda}{8\pi}\Big[\frac{\langle n_x,n_y\rangle}{\vert x-y\vert}\frac{(x-y)(x-y)^T}{\vert x-y\vert^2}+\frac{n_yn_x^T}{\vert x-y\vert}\Big]+T_\lambda,
\end{equation}where $T_\lambda$ is a kernel of class $C^1$. %Here, the definition of $\frac{\partial}{\partial N}$ given in \eqref{der-nu}.
\end{lemma}

\paragraph{Single and double boundary layers}~~In the sequel, we use the Einstein convention for summation signs, i.e., we omit them for indices appearing twice. Let $\phi=(\phi^1,\phi^2,\phi^3) \in L^2(\partial \Omega)^3$.  The {{\it {single-layer potential}} pair $({S}_\Omega^\lambda,\mathcal{F}_\Omega)$ with density $\phi$ is defined, for $x\in\Omega$, as 
\begin{equation}
\left\{
\begin{array}{lll}
{S}^\lambda_\Omega[\phi]_i(x)&=&\displaystyle \int_{\partial \Omega}\Gamma_{ij}^{\lambda}(x-y)\phi^j(y)~d\sigma_y,~~1\leq i\leq 3,\smallskip\\
\mathcal{F}_\Omega[\phi](x)&=&\displaystyle\int_{\partial \Omega}F_j(x-y)\phi^j(y)~ d\sigma_y,
\end{array}
\right.
\end{equation} 
while the {\it{double  hydrodynamic potential}} pair $({D}_\Omega^\lambda,\mathcal{V}_\Omega)$ with density $\phi$ is defined by  
\begin{equation}
\displaystyle \left\{
\begin{array}{lll}
{D}^\lambda_\Omega[\phi]_i(x)&=&\displaystyle \int_{\partial \Omega}\left(\frac{\partial \Gamma_{ij}^{\lambda}}{\partial N(y)}(x-y) +F_i(x-y)n_j(y)\right) \phi_j(y)~d\sigma_y,~~1\leq i\leq 3,\\
\displaystyle\mathcal{V}_\Omega[\phi](x)&=&\displaystyle -2\int_{\partial \Omega}\frac{\partial F_j}{\partial x_l}(x-y)\phi^j(y)~n_l(y)~d\sigma_y.
\end{array}
\right.
\end{equation} 
We quote from \cite{lady} that 
$$
\frac{\partial \Gamma_{ij}^{\lambda}}{\partial N(y)}(x-y)=\left(
\frac{\partial \Gamma_{ij}^{\lambda}(x-y)}{\partial y_l}+
\frac{\partial \Gamma_{il}^{\lambda}(x-y)}{\partial y_j}\right)n_l(y).
$$
%according to the definition of $\frac{\partial}{\partial N}$ given in \eqref{der-nu}.
\paragraph{Some background results  about  the layer potential representations }
\par
\noindent
From \cite{ammari}, we quote the following  integral equations satisfied by $\phi^\lambda$   and the associated pressure $p^\lambda$. First,  we have the following representation formulas,
 \begin{equation}
\left\{
\begin{array}{llll}
\phi^\lambda(x)&=&\displaystyle-{S}^\lambda_\Omega[\frac{\partial \phi^\lambda}{\partial \nu}](x)+{D}^\lambda_\Omega[\phi^\lambda](x),&~~x\in \Omega,\\
p^\lambda(x)&=&\displaystyle-\mathcal{F}_\Omega[ \frac{\partial \phi^\lambda}{\partial \nu}](x)+\displaystyle\mathcal{V}_\Omega[\phi^\lambda](x),&~~x\in \Omega.
\end{array}
\right.
\end{equation}
Applying the trace  stress operators and taking into account the single layer potential as well as the jump relations for the double layer potential across the boundary,  we get for  $\phi$ belonging to $L^2(\partial \Omega)^3$ the following relations,
\begin{equation}
\left\{
\begin{array}{llll}
\displaystyle 
\displaystyle{D}^\lambda_\Omega[\phi](x)&=&\displaystyle(\frac{1}{2}I + {{K}}_\Omega^\lambda)[\phi](x),&\hbox{ a.e. on }\partial \Omega,\\
\displaystyle\frac{\partial}{\partial \nu} {S}_\Omega^\lambda[\phi](x)&=&\displaystyle(-\frac{1}{2}I + ({{K}}_\Omega^\lambda)^*)[\phi](x),&\hbox{ a.e. on }\partial \Omega,
\end{array}
\right.
\end{equation} 
where the kernel  ${{K^\lambda_\Omega}}[\phi]$ is defined a.e. on $\partial\Omega$ by its components, 
\begin{equation}\label{k-lamb}
{K}^\lambda_\Omega[\phi]_i(x):=\textrm{p.v.} \int_{\partial \Omega}\frac{\partial {{\Gamma}}_{ij}^\lambda}{\partial { N}(y)}(x-y)\phi^j(y) ~d\sigma_y+\textrm{p.v.} \int_{\partial \Omega}F_i(x-y)\phi^j(y)n_j(y)~d\sigma_y.
\end{equation}
Here, the notation `` $\textrm{p.v.}$'' indicates the Cauchy principal value when the integrand is singular at $x$, more precisely 
$$
\displaystyle \textrm{p.v.}\int_\Gamma \ldots= \lim_{\varepsilon \rightarrow 0} \int_{\Omega \backslash B(x,\varepsilon)} \ldots
$$
where $B(x,\varepsilon)$ is the ball centered at $x$ of radius $\varepsilon$.  The adjoint operator  ${ {K}^\lambda_\Omega  } ^*$ of ${K}^\lambda_\Omega$ is defined similarly a.e. on $\partial\Omega$ by its components 
\begin{equation}
{{K^\lambda_\Omega}}^*[\phi]_i(x)=\textrm{p.v.} \int_{\partial \Omega}\frac{\partial {{\Gamma}}_{ij}^\lambda}{\partial {N}(x)}(x-y)\phi^j(y) ~d\sigma_y-\textrm{p.v.} \int_{\partial \Omega}F_i(x-y)\phi^j(y)n_j(x)~d\sigma_y,
\end{equation}
for all functions ${{\phi}}$ belonging to  $L^2(\partial \Omega)^3.$  Let us recall that in the case of the standard Stokes system ($\lambda=0$), we have 
\begin{equation}\label{K0}
{{K}}_\Omega^0[\phi](x)=-\frac{3}{4\pi}\int_{\partial \Omega} (\x-\y) \frac{\langle \x-\y,{ n}(y)\rangle~\langle \x-\y,{\bf{\phi}}(y)\rangle}{\mid \x-\y\mid^5}~d\sigma_y.
\end{equation}
An important fact is that the single and double layer potentials ${S}_\Omega^\lambda$ and ${D}_\Omega^\lambda$ are compact perturbations of the single and double layer potentials corresponding to the standard Stokes problem.  

From  the $C^\ell$ regularity of the boundary $\Gamma$ with $\ell\geq 4$, it comes that 
\begin{equation}\label{geo-diff2}
\vert\langle \x-\y, {\bf{\phi}}(y)\rangle\vert\le C \vert\x-\y\vert^2,
\end{equation}
hence, we deduce (cf. \cite{lady}) that the mapping  ${K}^\lambda_\Omega[\phi] : C^\alpha(\partial \Omega) \mapsto C^{\alpha+1}(\partial \Omega)$ is in fact continuous. That shows that  ${K}^\lambda_\Omega[\phi]$  has a weakly singular kernel and then that it is a compact  operator on $L^2({\partial\Omega})^3$.  According to \eqref{eq:gamma-ij}, the operators $S_\Omega^\lambda-S_\Omega^0$ and $D_\Omega^\lambda-D_\Omega^0$ are smoothing operators.

Thanks to the integral representations provided in the preceding paragraph, we can use the trace  and the stress operators to deduce the second boundary integral equation satisfied by the conormal derivative. Indeed,  by using the same arguments of jump relations and the integral equations satisfied by $\phi^\lambda$, we get %, for $\phi \in H^1(\partial \Omega)^3$, 
\begin{equation}\label{Neumann}
\begin{array}{lll}
\displaystyle\Big( \frac{1}{2} I+({K}_\Omega^\lambda)^* \Big)\displaystyle\Big[\frac{\partial \phi}{\partial  \nu}\Big]_i(x)&=&
\displaystyle\Big[\frac{\partial D_\Omega^\lambda[\phi]}{\partial\nu}(x)\Big]_i\\
&=&\textrm{p.v.}  \displaystyle\int_{\partial \Omega}\frac{\partial^2{\Gamma}_{ij}^\lambda(x-y)}{\partial N(x)\partial N(y)}\phi^j(y)~d\sigma_y\\.
&=&\textrm{p.v.}  \displaystyle\int_{\partial \Omega}\frac{\partial^2{\Gamma}_{ij}^0(x-y)}{\partial N(x)\partial N(y)}\phi^j(y)~d\sigma_y\\ 
&+& \displaystyle\int_{\partial \Omega}\frac{\partial^2 \Delta_{ij}^{\lambda}(x-y)}{\partial N(x)\partial N(y)}\phi^j(y)~d\sigma_y.
\end{array}
\end{equation}
We cannot deduce directly the Neumann data (conormal derivative) since the operator $\displaystyle\Big( \frac{1}{2} I+({K}_\Omega^\lambda)^* \Big)$ is not invertible. We give, in the next paragraph, the recipes to get the solution of the system by using the projector methods.

 \subsection{Weakly singular integral operators of exponent $\alpha >0$}\label{WSO}
The rest of the paragraph follows Section 7.2 of \cite{DH}. Recall that the conormal derivative is solution of  $Tx=f$ where $T=I +2(K_\Omega^\lambda)^* $ is a Fredholm operator with  a nontrivial kernel.  We use $\mathcal{R}(T)$ to denote its closed image and $\mathcal{N}(T)$ its finite dimensional null space. We can therefore  find projections $P$ and $Q$ of finite rank such that there exists a unique operator $S$  satisfying  $TS=I-Q$ and $PS=0$. Hence,    the equation  $Tx=f$ has a solution if and only if $Qf=0$. In our context,
we have $T=I-C$ with $C=-2(K_\Omega^\lambda)^*$, which is a compact operator.  %From the projector theory recalled above,  we can find  $ S$ such that the equation $T x=b$ has a solution $x=Sf$ when $Q=0$. 
To proceed, we need some regularity assumptions on the operator $T$. For that purpose, we recall the following definition \cite[Definition 7.1.1, p117]{DH}.
\begin{de}
Let $A$ be an open set in $\R^3$. A function $K(x,y)$ defined for $x\neq y$ in $A\times A$ is a kernel of class $C^r_*(\alpha)$ in $A$ ($r$ non negative integer, and $\alpha>0$) if it is $C^r$ for $x\neq y$ and for any $\delta>0$ and $\vert i\vert+\vert j\vert+\vert k\vert\leq r$, one has $$\partial ^i_x\partial^j_y(\partial_x+\partial_y)^k K(x,y)=O(1+\vert x-y\vert^{\alpha-m-\vert i\vert-\vert j\vert-\delta}),$$ uniformly for $x\neq y$ in compact subsets of $A$. If $\alpha>m+\vert i\vert+\vert j\vert$, we require $\partial ^i_x\partial^j_y(\partial_x+\partial_y)^k K(x,y)$ to extend continuously to $\{x=y\}$.
\end{de}

Assume now that $ T=I-C$ is an integral operator where $C$ has a kernel belonging to $C_*^r(\alpha)$ for some $\alpha>0$. We may choose the projections $P$ and $Q$ to be integral operators with $C^r$  kernels so that, if $S$ is the operator such that 
 $$
 \begin{array}{lll}
 TS&=&I-Q,\\
 S&=&I+ R,\\%\qquad\textrm{modulo operators of finite ranks}
 PS&=&0,
 \end{array}
 $$
then the resolvent kernel $R$ is an integral operator with $C_*^r(\alpha)$ kernel. In fact, $ R-(C+C^2+\cdots+ C^j)$ has kernel of class $C_*^r((j+1)\alpha)$ for each $j\ge 1$.  Hence, for $N$ sufficiently large, the operator  $ R-\sum_{k=1}^N C^j$ has a smooth kernel of class $C^r$ \cite[Chapter III, p 79-94]{Pogorzelski:1966uq}. In summary, one has the following result.
  \begin{theo}[Theorem 7.2.3, page 125 in \cite{DH}]\label{important}
 We suppose $\Omega$ regular of class $C^{r+1}$, for some $r>0$. If  $K$ a kernel of class $C^r_*(\alpha)$, then we may choose the kernels  $P$ and $Q$ of the projections to be of class $C^r$ and  such that   the resolvent kernel $R$ belongs to $C^r_*(\alpha)$. Furthermore, the kernel of $R-(K+K^2+\ldots K^N)$ is of class $C^r$ for $N$ large enough. 
\end{theo}
We return to the study of Eq. (\ref{Neumann}). We introduce the vectors $b^{(0)}(x)=(b^{(0)}_i(x))_i$ and $e^{(\lambda)}(x)=(e^{(\lambda)}_i(x))$ where 
$$
b^{(0)}_i(x)= \displaystyle\int_{\partial \Omega}\frac{\partial^2\Gamma^0_{ij}(x-y)}{\partial N(x)\partial N(y)}\phi_j^\lambda(y)~d\sigma_y,
$$
and where 
$$
e^{(\lambda)}_i(x)=\displaystyle\int_{\partial \Omega}\frac{\partial^2 \Delta_{ij}^{\lambda}(x-y)}{\partial N(x)\partial N(y)}\phi_j^\lambda(y)~d\sigma_y.
$$
Hence it comes that 
\begin{equation}\label{conormal_derivative0}
\Big[\frac{\partial \phi^\lambda}{\partial  \nu}\Big]=b^{(0)}+
(\sum_{k=1}^NK^k)b^{(0)}+(\sum_{k=0}^NK^k)e^{(\lambda)}+
(R- \sum_{k=1}^NK^k)(b^{(0)} +e^{(\lambda)}).
\end{equation}
%Note that $(I+R)e^{(\lambda)}$ is actually a weakly singular operator acting on the Dirichlet data $\phi^\lambda$ of class $C^3_*(1)$. 

With our specific choice of Dirichlet data, we will show that $N=1$ is sufficient in our context and that all the other terms in  the sum will be absorbed by the remainder.  
 
 \subsection{Composition of weakly singular kernels}
For applications to our result on generic perturbation of the boundary, we need  to give an explicit representation of the conormal derivative or at least, of its principal and subprincipal parts as it is treated in the case of the Laplacian (for more details in the Laplacian case, one can refer to \cite{Yang}).   Some preliminaries are required in order 
 to study the resolvent kernels and their regularity. We begin by recalling some results due to  D. Henry (cf. \cite{DH}). It concerns kernels $K(x,y)$ of the form
 \begin{equation}
K(x,y):= \mid \x -\y \mid^{\alpha-2} Q\Big(x,y,\frac{x-y}{\mid x-y \mid}\Big)
\end{equation}
where $Q(x,y,s)$ is of class $C^r$ ($r>0$)  on $\mathbb{R}^2$. We will denote by $\mathcal{K}(r,\alpha)$ the set of such kernels, which is a subclass of $C_*^r(\alpha)$. 

These kernels are in fact smoothing operators and we recall the main result of \cite{DH}.
\begin{theo}[Theorem 7.1.2 in  \cite{DH}]
Given  a kernel $K$  belonging to the class $\mathcal{K}(\alpha,r)$, $\alpha,r>0$, we denote by $\tilde K$ the corresponding integral operator 
$$
\tilde{K} u(x) =\int_{\mathbb{R}^2}K(x,y) u(y) ~d\sigma_y.
$$
Then  we have 
\begin{itemize}
\item
$\tilde{K} : ~ W^{j,p}\mapsto W^{k,p}$ is a compact operator if $j-\frac{m}{p}\alpha >k-\frac{m}{q}$;
\par
\noindent
\item
$\tilde{K} :  ~C^{j,\sigma}\mapsto C^{k,\tau}$ is a compact operator if $j+\sigma+\alpha > k+\tau,~~k<r$ and $k<j+\alpha$.
\end{itemize}
\end{theo}
As it was mentioned in \cite{DH}, the above result can be summarized by the fact the operator $\tilde K$ is smoothing of order $\alpha$. We will also need a result  on the composition of certain weakly singular operators. For that purpose, we first define the composition of corresponding kernels as follows.
\par
\noindent
\begin{de} Let $K$ and $L$ be kernels belonging to $\mathcal{K}(\alpha,r)$ and $\mathcal{K}(\beta,r)$ respectively with $\alpha,\beta,r>0$. Then $K\circ L$ is defined by 
\begin{equation}\label{composition}
(K\circ L)(x,y)=\int_{\mathbb{R}^2} K(x,z)L(x,z) ~dz
\end{equation}
\end{de}
Then, one has the following property.
\begin{theo}[Theorem 7.1.3, p. 119 in  \cite{DH}]\label{th33}
 Let $K$ and $L$ be kernels belonging to $\mathcal{K}(\alpha,r)$ and $\mathcal{K}(\beta,r)$  respectively, with $\alpha,\beta,r>0$. Then $K\circ L$ is kernel of compact support belonging to $\mathcal{K}(\alpha+\beta,r)$. Furthermore, if $\alpha+\beta >r+2$, then $K\circ L$ is of class $C^r$.
\end{theo}
\par
\noindent

To these kernels, are associated integral operators $u \mapsto \displaystyle \int_{\partial \Omega} K(x,y)~dS(y)$ where $dS$ is the surface area measure on $\partial \Omega$.  In a first step, we begin to work in $\R^2$. To transfer all the results to $\partial \Omega$ (in particular, those provided above), one has to follow the classical  steps: construct a partition of unity and then define the integral by a local change of variables as it is precisely performed in \cite[Section 7.1]{DH}.

\section{Proofs of computational lemmas}

\subsection{Proof of Lemma \ref{deriv1}}\label{pf-deriv1}

From Eq. (\ref{eq:shape-stokes-tv}), we get the following system
\begin{eqnarray}
 -(\Delta+\lambda) {\phi}_i'(u) +\nabla p_i'(u)&=&{\lambda}_i'(u) \phi_i(u)\quad\textrm{ in }\Omega,\label{eq:stokes-tv1}\smallskip\\
\div~ \phi_i'(u)&=&0\quad\textrm{ in }\Omega,\label{eq:stokes-tv2}\smallskip\\
{\phi}_i'(u)+(u\cdot n)\frac{\partial \phi_i(u)}{\partial n}&=& 0\quad\textrm{ on }\partial\Omega,\label{eq:stokes-tv3}\smallskip\\
 p_i'(u)+\div~(up_i(u))&\in& L_0^2(\Omega).\label{eq:stokes-tv4}
%\displaystyle\int_{\Omega+u}p(u)&=&0.
\end{eqnarray}
Multiplying (\ref{eq:stokes-tv1}) by $\phi_k(u)$ with $1\leq k\leq m$, integrating over $\Omega$
and using Corollary \ref{green-coro}, we have
$$
\lambda_i'(u) \delta_{ik}= -\int_\Omega\phi_k(u)[(\Delta+\lambda) {\phi}_i'(u)-\nabla p_i'(u)]=\int_{\pO}{\phi}_i'(u)\frac{\partial \phi_k(u)}{\partial \nu}.
$$
%\begin{eqnarray*}
%\lambda_i'(u) \delta_{ik}&=& -\int_\Omega\phi_k(u)[(\Delta+\lambda) \phi_i'(u)-\nabla p_i'(u)]=\int_{\pO}{\phi}_i'(u)\frac{\partial \phi_k(u)}{\partial N}\\
%&=& -\int_\Omega\phi_k(u)[(\Delta+\lambda) {\phi}_i'(u)-\nabla p_i'(u)]=\int_{\pO}{\phi}_i'(u)\frac{\partial \phi_k(u)}{\partial N},
%\end{eqnarray*}
%since $\displaystyle \frac{\partial u}{\partial n}\cdot n=0$.
Hence, it comes that 
\begin{equation}
{\lambda}_i'(v)\delta_{ik}=-\int_{\pO} (u\cdot n)\frac{\partial \phi_i(u)}{\partial n}\cdot\frac{\partial \phi_k(u)}{\partial \nu}.
\end{equation}
Moreover, by Lemma (\ref{F12}), we have 
\begin{eqnarray*}
&&\frac{\partial \phi_i(u)}{\partial n}\cdot\frac{\partial \phi_k(u)}{\partial \nu}=\frac{\partial \phi_i(u)}{\partial n}\cdot(\frac{\partial \phi_k(u)}{\partial n}+\nabla^T\phi_k(u)n-p_k(u)n)\\
&=&\frac{\partial \phi_i(u)}{\partial n}\cdot\frac{\partial \phi_k(u)}{\partial n}+\frac{\partial \phi_i(u)}{\partial n}\cdot(\frac{\partial \phi_k(u)}{\partial n}n^T)^Tn=\frac{\partial \phi_i(u)}{\partial n}\cdot\frac{\partial \phi_k(u)}{\partial n}.
\end{eqnarray*}
Therefore, we immediately get Eq. (\ref{id_m}).

\subsection{Proof of Lemma \ref{singularity_couche}}

Lemma \ref{singularity_couche} is derived from \cite[Lemma 2.2.3 Formula (2.2.34) and Lemma 2.3.1]{Hsiao} by straightforward computations. For the reader's convenience, we first summarize these results in the following lemma and then give the proof of Lemma \ref{singularity_couche}.

\begin{lemma}\label{formula-Hsiao}
Let $\partial\Omega$ be of class $C^1$ and $u=(u^\ell)_{\ell=1,2,3}$ be a H{\"o}lder continuously differentiable function. Then the operator $\E$ defined in \eqref{E} can be expressed as follows
\begin{eqnarray}
\E u(x)&=&-\frac{1}{4\pi}(\n_\x\times\nabla_\x)\cdot\int_{\partial\Omega}\frac{1}{\vert \x-\y\vert}(\n_\y\times\nabla_\y)\uu(\y)d\sigma(\y)\label{E1}\\
&&-\frac{1}{2\pi}\mathcal{M}(\partial_\x,\n_\x)\int_{\partial\Omega}\frac{(\x-\y)(\x-\y)^T}{\vert \x-\y\vert^3}\mathcal{M}(\partial_\y,\n_\y)\uu(\y)d\sigma(\y))\label{E2}\\
&&+\frac{1}{4\pi}\Big(\sum_{l,k=1}^3{m}_{lk}(\partial_\x,\n_\x)\int_{\partial\Omega}\frac{1}{\vert \x-\y\vert}(m_{kj}(\partial_\y,\n_\y)\uu^\ell)(\y)d\sigma(\y)\Big)_{j=1,2,3},\label{E3}
\end{eqnarray}where the $\ell^{\textrm{th}}$-column of the matrix $(\n_\y\times\nabla_\y)\uu(\y)$ is given by the vector $\n_\y\times\nabla_\y\uu^\ell(\y)$, and the G\"unter derivatives $\mathcal{M}$ is given by the following matrix of differential operators
$$\mathcal{M}(\partial_\x,\n_\x)=(m_{jk}(\partial_\x,\n_\x))_{j,k=1,2,3}:=(\n_{\x,k}\partial_{\x_j}-\n_{\x,j}\partial_{\x_k})_{j,k=1,2,3},$$with $\n_\x=(\n_{\x,j})_{j=1,2,3}$.
\end{lemma}

\begin{coro}\label{coro:hsiao}
Under the assumptions of Lemma \ref{formula-Hsiao}, we have
\begin{eqnarray}
4\pi\E u(x)&=& \emph{\textrm{p.v.}} \int_{\partial\Omega} \frac{\langle n_x,n_y\rangle}{\vert x-y\vert^3}\Big(\nabla u(y)+\nabla^Tu(y)\Big)(x-y)d\sigma_y\label{EE1}\\
&+&\emph{\textrm{p.v.}} \int_{\partial\Omega}\frac{\langle n_x,(\nabla u(y)-\nabla^Tu(y))(x-y)\rangle}{\vert x-y\vert^3}n_yd\sigma_y\label{EE2}\\
&-& \int_{\partial\Omega}\frac{\langle x-y, n_y\rangle}{\vert x-y\vert^3}\Big(\nabla u(y)+\nabla^Tu(y)\Big)n_xd\sigma_y\label{EE3}\\
&+&\int_{\partial\Omega}\frac{\langle x-y, n_y\rangle}{\vert x-y\vert^3}\Big(I-3\frac{(x-y)(x-y)^T}{\vert x-y\vert^2}\Big)\mathcal{M}(\partial_y,n_y)u(y)d\sigma_y.\label{EE4}
\end{eqnarray}
\end{coro}
\begin{proof}[Proof of Corollary \ref{coro:hsiao}]%
%Recall that $\uu=\alpha\bpsi$ with $\alpha:\partial\Omega\mapsto\R$ and $\bpsi:\partial\Omega\mapsto\R^3$. We treat separately each of the three integral terms considered in Lemma \ref{formula-Hsiao}. 
For \eqref{E1}, we get
\begin{eqnarray*}
&&(\n_\x\times\nabla_\x)\cdot\int_{\partial\Omega}\frac{1}{\vert x-y\vert}(\n_\y\times\nabla_yu^\ell(\y))d\sigma(\y)\\
&=&\textrm{p.v.}\int_{\partial\Omega}(n_x\times\nabla_x\frac{1}{\vert x-y\vert})\cdot(n_y\times\nabla_yu^\ell(\y))d\sigma(\y).
\end{eqnarray*}
For $x\neq y$, one has
\begin{eqnarray*}
&&(n_x\times\nabla_x\frac{1}{\vert \x-\y\vert})\cdot(n_y\times\nabla_yu^\ell(y))\\
&=&(n_x^Tn_y)(\nabla_x\frac{1}{\vert x-y\vert}\nabla_y^Tu^\ell(y))-(\nabla_yu^\ell(\y)n_x)(\nabla_x\frac{1}{\vert x-y\vert}n_y)\\
&=&-\frac{n_x^Tn_y}{\vert \x-\y\vert^3}\nabla_yu^\ell(y)(\x-\y)+\frac{(x-y)^Tn_y}{\vert \x-\y\vert^3}\nabla_yu^\ell(y)n_x.
\end{eqnarray*}
Therefore, we have
\begin{eqnarray}
&&(\n_\x\times\nabla_\x)\cdot\int_{\partial\Omega}\frac{1}{\vert \x-\y\vert}(\n_\y\times\nabla_\y)\uu(\y)d\sigma(\y)\nonumber\\
&=&-\textrm{p.v.}\int_{\partial\Omega}\frac{\langle n_x, n_y\rangle}{\vert \x-\y\vert^3}\nabla_yu(y)(\x-\y)d\sigma_y+\textrm{p.v.}\int_{\partial\Omega}\frac{\langle x-y, n_y\rangle}{\vert \x-\y\vert^3}\nabla_yu(y)n_xd\sigma_y.\label{pEE1}
\end{eqnarray}

We compute now the second piece of  \eqref{E2} and obtain for $x\neq y$
\begin{eqnarray*}
&&\mathcal{M}(\n_\x,\partial_\x)\frac{(\x-\y)(\x-\y)^T}{\vert \x-\y\vert^3}=\Big(\sum_{k=1}^3m_{ik}(\n_{\y},\partial_{\y})\frac{(\x_k-\y_k)(\x_j-\y_j)}{\vert \x-\y\vert^3}\Big)_{i,j=1,2,3}\\
&=&\Big(\sum_{k=1}^3 (\n_{\x,k}\partial_{\x_i}-\n_{\x,i}\partial_{\x_k})\frac{(\x_k-\y_k)(\x_j-\y_j)}{\vert \x-\y\vert^3}\Big)_{i,j=1,2,3}\\
&=&\Big(\sum_{k=1}^3 \n_{\x,k}[-3\frac{(\x_i-\y_i)(\x_k-\y_k)(\x_j-\y_j)}{\vert\x-\y\vert^3}+\frac{\delta_{ik}(\x_j-\y_j)}{\vert\x-\y\vert^3}+\frac{\delta_{ij}(\x_k-\y_k)}{\vert\x-\y\vert^3}]\\
&&-\n_{\x,i}[-3\frac{(\x_k-\y_k)^2(\x_j-\y_j)}{\vert\x-\y\vert^3}+\frac{(\x_j-\y_j)}{\vert\x-\y\vert^3}+\frac{\delta_{kj}(\x_k-\y_k)}{\vert\x-\y\vert^3}]\Big)_{i,j=1,2,3}\\
&=&-3\frac{\langle\n_\x, \x-\y\rangle}{\vert\x-\y \vert^5}(\x-\y)(\x-\y)^T+\frac{\n_\x(\x-\y)^T}{\vert \x-\y\vert^3}+\frac{\langle\n_\x, \x-\y\rangle}{\vert\x-\y \vert^3}I_3\\
&&+3\frac{\n_\x(\x-\y)^T}{\vert\x-\y \vert^3}-3\frac{\n_\x(\x-\y)^T}{\vert\x-\y \vert^3}-\frac{\n_\x(\x-\y)^T}{\vert \x-\y\vert^3}\\
&=&\frac{\langle\n_\x, \x-\y\rangle}{\vert\x-\y \vert^3}\Big(I_3-3\frac{(\x-\y)(\x-\y)^T}{\vert\x-\y\vert^2}\Big).
\end{eqnarray*}
%As a first consequence and by using \eqref{geo-diff2}, one has the above expression is equal to $k_2(x,y)\cdot \nabla(\alpha\psi)(y)$ where $k_2(x,y)$ is a $C^3_*(1)$ kernel (of appropriate matrix size) and ``$\cdot$'' means that the action of $k_2$ on  $\nabla(\alpha\psi)$ is linear.

%We also derive for \eqref{E2} that 
Therefore, we have
\begin{eqnarray}
&&\mathcal{M}(\partial_\x,\n_\x)\int_{\partial\Omega}\frac{(\x-\y)(\x-\y)^T}{\vert \x-\y\vert^3}\mathcal{M}(\partial_\y,\n_\y)u(\y)d\sigma(\y)\nonumber\\
&=&\textrm{p.v.}\int_{\partial\Omega}\frac{\langle\n_\x, \x-\y\rangle}{\vert\x-\y \vert^3}\Big(I_3-3\frac{(\x-\y)(\x-\y)^T}{\vert\x-\y\vert^2}\Big)\mathcal{M}(\partial_y,n_y)u(y)d\sigma_y,\label{pEE2}%[\alpha(\y)\mathcal{M}(\partial_\y,\n_\y)\psi(\y)-\langle\nabla\alpha(\y),\psi(\y)\rangle\n_\y]d\sigma(\y),
\end{eqnarray}
keeping in mind that there is no principal value if one uses \eqref{geo-diff2}.\medskip

We finally turn to \eqref{E3}. One has, for $x\neq y$,
\begin{eqnarray*}
&&\Big(\sum_{\ell,k=1}^3({m}_{lk}(\partial_\x,\n_\x)\frac{1}{\vert \x-\y\vert})(m_{kj}(\partial_\y,\n_\y)\uu^\ell(\y))\Big)_{j=1,2,3}\\
&=&\Big(\sum_{\ell,k=1}^3\Big(-\n_{\x,k}\frac{\x_\ell-\y_\ell}{\vert \x-\y\vert^3}+\n_{\x,l}\frac{\x_k-\y_k}{\vert \x-\y\vert^3} \Big)\Big(\n_{\y,j}\partial_{\y_k}u^\ell(\y)-\n_{\y,k}\partial_{\y_j}u^\ell(\y)\Big)\Big)_{j=1,2,3}\\
&=&\frac{\langle n_x,(\nabla u(y)-\nabla^Tu(y))(x-y)\rangle}{\vert x-y\vert^3}n_y-\frac{\langle x-y,n_y\rangle}{\vert x-y\vert^3}\nabla^T u(y)n_x\\
&&+~Ê \frac{\langle n_x,n_y\rangle}{\vert x-y\vert^3}\nabla^Tu(y)(x-y).
\end{eqnarray*}

Therefore, we have
\begin{eqnarray}
&&\Big(\sum_{\ell,k=1}^3{m}_{lk}(\partial_\x,\n_\x)\int_{\partial\Omega}\frac{1}{\vert \x-\y\vert}(m_{kj}(\partial_\y,\n_\y)\uu^\ell)(\y)d\sigma(\y)\Big)_{j=1,2,3}\nonumber\\
&=&\textrm{p.v.}\int_{\partial\Omega}\frac{\langle n_x,(\nabla u(y)-\nabla^Tu(y))(x-y)\rangle}{\vert x-y\vert^3}n_yd\sigma_y-\textrm{p.v.}\int_{\partial\Omega}\frac{\langle x-y,n_y\rangle}{\vert x-y\vert^3}\nabla^T u(y)n_xd\sigma_y\nonumber\\
&&+~ \textrm{p.v.}\int_{\partial\Omega}\frac{\langle n_x,n_y\rangle}{\vert x-y\vert^3}\nabla^Tu(y)(x-y)d\sigma_y.\label{pEE3}
\end{eqnarray}

Gathering (\ref{pEE1}), (\ref{pEE2}), and (\ref{pEE3}), Corollary \ref{coro:hsiao} is proved.

\end{proof}

\begin{proof}[Proof of Lemma \ref{singularity_couche}]
Recall that $u=\alpha\psi$ with $\alpha:\partial\Omega\mapsto\R$ and $\psi:\partial\Omega\mapsto\R^3$. We note that
\begin{eqnarray}
\nabla(\alpha\psi)&=&\alpha\nabla\psi+\psi\nabla\alpha,\label{EEE1}.
%\nabla^T(\alpha\psi)&=&\alpha\nabla^T\psi+\nabla^T\alpha~\psi^T,\label{EEE2}
\end{eqnarray} and
\begin{eqnarray}
&&\mathcal{M}(\partial_\y,\n_\y)(\alpha\psi)(\y)=\Big(\sum_{k=1}^3m_{ik}(\n_{\y},\partial_{\y})(\alpha(\y)\psi_k(\y))\Big)_{i=1,2,3}\nonumber\\
&=&\alpha(\y)\mathcal{M}(\partial_\y,\n_\y)\psi(\y)+\Big(\sum_{k=1}^3(\n_{\y,k}\partial_{\y_i}\alpha(\y)-\n_{\y,i}\partial_{\y_k}\alpha(\y))\psi_k(\y)\Big)_{i=1,2,3}\nonumber\\
&=&\alpha(\y)\mathcal{M}(\partial_\y,\n_\y)\psi(\y)+\langle\n_y,\psi(\y)\rangle\nabla^T\alpha(\y)-(\nabla\alpha(\y)\psi(\y))\n_\y\nonumber\smallskip\\
&=&\alpha(\y)\mathcal{M}(\partial_\y,\n_\y)\psi(\y)-(\nabla\alpha(\y)\psi(\y))\n_\y.\label{EEE2}
\end{eqnarray}
Then, the expressions of $A_i$, $1\leq i\leq 4$, simply result from developping $\nabla u$ in (\ref{EE1}) and (\ref{EE2}) of Corollary \ref{coro:hsiao} and $A_5$ collects (\ref{EE3}) and (\ref{EE4}) 
as a weakly singular operator of class $C^3_*(1)$. Hence Lemma \ref{singularity_couche} follows.

\end{proof}

\subsection{Proof of Lemma \ref{miracle1}}\label{pf-miracle1}
Using polar coordinates, we have
\begin{eqnarray*}
&&\int_{B(0,\delta)}\frac{\alpha_{\varepsilon,\eta_0}(\eta)}{\mid \eta\mid^{1-m}}d\eta=\frac{1}{\varepsilon^2}\int_{B(0,\delta)}\frac{e^{-\frac{\vert \eta-\eta_0\vert^2}{\varepsilon^2}}}{\vert\eta\vert^{1-m}}d\eta=\frac{e^{-\br^2}}{\varepsilon^2}\int_0^\delta\int_0^{2\pi}\exp{(-\frac{r^2}{\varepsilon^2}+2\frac{r}{\varepsilon}\br\cos\theta)}r^mdrd\theta\\
&=&\frac{e^{-\br^2}}{\varepsilon^{1-m}}\int_0^{\delta/\varepsilon}\int_0^{2\pi}\exp{(-r^2+2r\br\cos\theta)}r^mdrd\theta\\
&\leq&\frac{e^{-\br^2}}{\varepsilon^{1-m}}\int_0^{\infty}\int_0^{2\pi}\exp{(-r^2+2r\br\cos\theta)}r^mdrd\theta\leq \frac{2\pi}{\varepsilon^{1-m}}\int_{-\br}^{\infty}(r+\br)^m\exp{(-r^2)}dr
\end{eqnarray*}
As $\br\leq 1$, there exists a constant $C(m)>0$ depending only on $m$ such that (\ref{est1}) holds true. 
 
\subsection{Proof of Lemma \ref{miracle2}}\label{pf-miracle2}
We use polar coordinates and get 
\begin{eqnarray*}
~\textrm{p.v.}\int_{\R^2}\frac{\alpha_{\varepsilon,\eta_0}(\eta)\eta}{\mid \eta\mid^3}d\eta&=&\frac{e^{-\br^2}}{\varepsilon^2}~\textrm{p.v.}\int_{0}^{\infty}\frac{e^{-r^2}}{r}dr\int_{0}^{2\pi}
\exp(2r\br\cos(\theta-\theta_0))\begin{pmatrix}\cos\theta\\ \sin\theta\end{pmatrix}d\theta\\
&=&\frac{e^{-\br^2}}{\varepsilon^2}~\textrm{p.v.}\int_{0}^{\infty}\frac{e^{-r^2}}{r}dr\int_{0}^{2\pi}
\cos\theta\exp(2r\br\cos\theta)d\theta\begin{pmatrix}\cos\theta_0\\ \sin\theta_0\end{pmatrix}\\
&=&\frac{e^{-\br^2}}{\varepsilon^2}M_3^{A_1^1}(\br)\bar\eta_0,
\end{eqnarray*}
where we recall that $\bar\eta_0=\br\begin{pmatrix}\cos\theta_0\\ \sin\theta_0\end{pmatrix}$
and where we have set
\begin{equation}\label{def-M3}
M_3^{A_1}(z):=\frac1{z}~\textrm{p.v.}\int_{0}^{\infty}\frac{e^{-r^2}}{r}dr\int_{0}^{2\pi}\cos\theta
\exp(2rz\cos\theta)d\theta.
\end{equation}
Standard computations yield that
\begin{eqnarray*}
M_3^{A_1}(z)
&=&\frac1{z}~\textrm{p.v.}\int_{0}^{\infty}\frac{e^{-r^2}}{r}dr\sum_{k=0}^{\infty}\frac{(2r)^kz^k}{k!}\int_{0}^{2\pi}\cos^{k+1}\theta d\theta\\&=&\frac4{z}~\textrm{p.v.}\int_{0}^{\infty}\frac{e^{-r^2}}{r}dr\sum_{p=0}^{\infty}\frac{(2r)^{2p+1}z^{2p+1}}{(2p+1)!}I_{2p+2}\\
%4\sum_{p=0}^\infty\frac{2^{2p+1}}{(2p+1)!}I_{2p+2}z^{2p+1}\int_{0}^{c}e^{- r^2}r^{2p}dr=
&=&2\sum_{p=0}^\infty\frac{2^{2p+1}}{(2p+1)!}I_{2p+2}\Gamma(p+\frac{1}{2})z^{2p},
%&=&\pi\sum_{p=0}^{\infty}\frac{(p+1)\Gamma(p+\frac{1}{2})}{((p+1)!)^2}z^{2p},
\end{eqnarray*} 
where $\displaystyle I_k:=\int_{0}^{\pi/2}\cos^k\theta d\theta$ is the Wallis integral and $\displaystyle\Gamma(s):=\int_0^\infty t^{s-1}e^{-t}dt$ is the Gamma function.
Using the fact that $\displaystyle I_{2p}=\frac{(2p)!}{2^{2p}(p!)^2}\frac{\pi}{2}$, we have
\begin{equation}\label{cv-M3}
M_3^{A_1}(z)=\pi\sum_{p=0}^{\infty}\frac{\Gamma(p+\frac{1}{2})}{p!(p+1)!}z^{2p}.
\end{equation}
The radius of convergence of $M_3^{A_1}$ is clearly infinite, since
$$
\lim_{p\rightarrow\infty}\frac{\Gamma(p+\frac{1}{2})(p+1)!(p+2)!}{\Gamma(p+\frac{3}{2})p!(p+1)!}=\frac{(p+1)(p+2)}{p+\frac{1}{2}}=\infty,
$$
where we have used the standard fact that $\Gamma(z+1)=z\Gamma(z)$ for $\Re(z)>0$.
Lemma~\ref{miracle2} is thus established.

\subsection{Proof of Lemma \ref{de:M-A1}}\label{pf-de:M-A1}
One has
\begin{eqnarray*}
&&M_1^{A_1}(z)\\
&=&\int_{0}^{\infty}e^{-r^2}dr\int_0^{2\pi}\cos^2\theta\exp(2rz\cos\theta)d\theta=\int_{0}^{\infty}e^{-r^2}dr\int_0^{2\pi}\cos^2\theta\sum_{k=0}^{\infty}\frac{(2r)^kz^k}{k!}\cos^k\theta d\theta\\
&=&\int_{0}^{\infty}e^{-r^2}dr\sum_{k=0}^{\infty}\frac{(2r)^kz^k}{k!}\int_{0}^{2\pi}\cos^{k+2}\theta d\theta=\sum_{p=0}^\infty\frac{2^{2p}z^{2p}}{(2p)!}I_{2p+2}\int_{0}^{\infty}e^{- r^2}r^{2p}dr\\
&=&\sum_{p=0}^\infty\frac{2^{2p+1}}{(2p)!}I_{2p+2}\Gamma(p+\frac{1}{2})z^{2p}.
\end{eqnarray*}
%where $\displaystyle I_k:=\int_{0}^{\pi/2}\cos^k\theta d\theta$ is the Wallis integral and $%\Gamma(s):=\int_0^\infty t^{s-1}e^{-t}dt$ is the gamma function.
Using the fact that $\displaystyle I_{2p}=\frac{(2p)!}{2^{2p}(p!)^2}\frac{\pi}{2}$, we have
\begin{equation}\label{cv-M1}
M_1^{A_1}(z)=\frac{\pi}{4}\sum_{p=0}^{\infty}\frac{(2p+2)(2p+1)}{((p+1)!)^2}\Gamma(p+\frac{1}{2})z^{2p}.
\end{equation}

The radius of convergence of $M_1^{A_1}$ is infinite since 
\begin{equation*}
\lim_{p\rightarrow+\infty}\frac{(2p+2)(2p+1)}{(2p+4)(2p+3)}\frac{((p+2)!)^2}{((p+1)!)^2}\frac{\Gamma(p+\frac{1}{2})}{\Gamma(p+\frac{1}{2}+1)}=+\infty.
\end{equation*}

Let $M_5^{A_1}(z)$ be defined by
\begin{equation}\label{eq:m5}
M_5^{A_1}(z):=\int_{0}^{\infty}\exp(-r^2)dr\int_0^{2\pi}\exp(2rz\cos\theta)d\theta.
\end{equation}We have
\begin{eqnarray*}
M_5^{A_1}(z)&=&\int_{0}^{\infty}\exp(-r^2)dr\int_0^{2\pi}\exp(2rz\cos\theta)d\theta=\int_{0}^{\infty}e^{-r^2}dr\sum_{k=0}^{\infty}\frac{(2r)^kz^k}{k!}\int_{0}^{2\pi}\cos^k\theta d\theta\\
&=&\sum_{p=0}^\infty\frac{2^{2p+1}}{(2p)!}I_{2p}\Gamma(p+\frac{1}{2})z^{2p}=\pi\sum_{p=0}^\infty\frac{\Gamma(p+\frac{1}{2})}{(p!)^2}z^{2p}.
\end{eqnarray*} It is clear that the radius of convergence of $M_5^{A_1}(\cdot)$ is infinite. Since $\displaystyle M_2^{A_1}(z)=M_5^{A_1}(z)-M_1^{A_1}(z)$, the radius of convergence of $M_2^{A_1}(z)$ is also infinite.

We now prove that $z\mapsto M_4^{A_1}(z)$ is well-defined and not identically equal to zero. Indeed, 
\begin{eqnarray*}
%M_4^{A_1^1}(z)&=&
&&M_1^{A_1}(z)-z^2M_3^{A_1^1}(z)-M_2^{A_1^1}(z)\\
&=&2M_1^{A_1}(z)-\pi\sum_{p=0}^\infty\frac{\Gamma(p+\frac{1}{2})}{(p!)^2}z^{2p}-z^2M_3^{A_1^1}(z)\\
&=&\pi\sum_{p=0}^{\infty}[\frac{(p+1)(2p+1)}{((p+1)!)^2}-\frac{1}{(p!)^2}]\Gamma(p+\frac{1}{2})z^{2p}-\pi\sum_{p=1}^{\infty}\frac{p}{(p!)^2}\Gamma(p-\frac{1}{2})z^{2p}\\
&=&-\frac{3\pi}2\sum_{p=1}^{\infty}\frac{p\Gamma(p-\frac{1}{2})}{(p+1)(p!)^2}z^{2p},
%\frac{\Gamma(p+\frac{1}{2})}{(p+1)(2p-1)}%\Gamma(p+\frac{1}{2})-\Gamma(p-\frac{1}{2})]z^{2p}.
\end{eqnarray*}
%where we have used the standard identity $\Gamma(p+\frac{1}{2})=(p-\frac{1}{2})\Gamma(p-\frac{1}{2})$. 
Then, the function $z\mapsto M_4^{A_1}(z)$ is defined by
\begin{equation}\label{def-M4}
M_4^{A_1}(z)=-\frac{3\pi}2\sum_{p=0}^{\infty}\frac{(p+1)\Gamma(p+\frac{1}{2})}{(p+2)((p+1)!)^2}z^{2p},
\end{equation}
which is clearly a non zero entire function.

\subsection{Proof of Lemma \ref{calcul-final}}\label{pf-calcul-final}
We give in this section explicit expressions of $a_1$, $a_2$, and $a_3$ defined respectively in Eqs. \eqref{eq:aa1}, \eqref{eq:aa2}, and \eqref{eq:aa3}. The computations are lengthy but straightforward. \smallskip

We start by computing $a_1$.
\begin{eqnarray*}
&&\int_{\R^2}\frac{\alpha_\eps(\eta)}{\vert \eta\vert^3}\eta^TF_x\eta\langle\psi(x),\eta\rangle\eta d\eta\\
&=&\eps e^{-\br^2}R_{\theta_0}\int_0^\infty e^{-r^2}r^2dr\int_0^{2\pi}e^{2r\br\cos\theta}(F_{\theta_0}^{11}\cos^2\theta+2F_{\theta_0}^{12}\cos\theta\sin\theta+F_{\theta_0}^{22}(1-\cos^2\theta))\\
&&~\begin{pmatrix}\cos^2\theta&\sin\theta\cos\theta\\\sin\theta\cos\theta&1-\cos^2\theta\end{pmatrix}d\theta~R^T_{\theta_0}\psi(x)\\
&=&\eps e^{-\br^2}R_{\theta_0}\int_0^\infty e^{-r^2}r^2dr\int_0^{2\pi}e^{2r\br\cos\theta}\\
&&~\begin{pmatrix}F^{22}_{\theta_0}\cos^2\theta+(F^{11}_{\theta_0}-F^{22}_{\theta_0})\cos^4\theta,&2 F_{\theta_0}^{12}\cos^2\theta(1-\cos^2\theta)\\2 F_{\theta_0}^{12}\cos^2\theta(1-\cos^2\theta),&F^{22}_{\theta_0}+(F^{11}_{\theta_0}-2F^{22}_{\theta_0})\cos^2\theta-(F^{11}_{\theta_0}-F^{22}_{\theta_0})\cos^4\theta \end{pmatrix}d\theta\\
&&R^T_{\theta_0}\psi(x).
\end{eqnarray*}\smallskip

The functions $M_6(\cdot)$, $M_7(\cdot)$, and $M_8(\cdot)$ were defined in Eqs. (\ref{eq:m6}), (\ref{eq:m7}) and (\ref{eq:m8}) respectively.
%\begin{eqnarray*}
%M_6(z)&=&\int_0^\infty e^{-r^2}r^2dr\int_0^{2\pi}e^{2rz\cos\theta}d\theta,\\
%M_7(z)&=&\int_0^\infty e^{-r^2}r^2dr\int_0^{2\pi}e^{2rz\cos\theta}\cos^2\theta d\theta,\\
%M_8(z)&=&\int_0^\infty e^{-r^2}r^2dr\int_0^{2\pi}e^{2rz\cos\theta}\cos^4\theta d\theta.
%\end{eqnarray*}
Then, we have
\begin{eqnarray}
\int_{\R^2}\frac{\alpha_\eps(\eta)}{\vert \eta\vert^3}\eta^TF_x\eta\langle\psi(x),\eta\rangle\eta d\eta=\eps e^{-\br^2}R_{\theta_0}\mathcal{M}(\br)R^T_{\theta_0}\psi(x),
\end{eqnarray}with
\begin{eqnarray*}
&&\mathcal{M}(\br)\\
&:=&\begin{pmatrix}F^{22}_{\theta_0}M_7(\br)+(F^{11}_{\theta_0}-F^{22}_{\theta_0})M_8(\br),&2 F_{\theta_0}^{12}(M_7(\br)-M_8(\br))\\2 F_{\theta_0}^{12}(M_7(\br)-M_8(\br)),&F^{22}_{\theta_0}M_6(\br)+(F^{11}_{\theta_0}-2F^{22}_{\theta_0})M_7(\br)-(F^{11}_{\theta_0}-F^{22}_{\theta_0})M_8(\br) \end{pmatrix}.\nonumber
\end{eqnarray*}
Then, 
\begin{eqnarray*}
&&R_{\theta_0}\mathcal{M}(\br)R_{\theta_0}^T\\
%&=&\begin{pmatrix}(\mathcal{M}_{11}+\mathcal{M}_{22})+(\mathcal{M}_{11}-\mathcal{M}_{22})\cos2\theta_0-2\mathcal{M}_{12}\sin2\theta_0,&(\mathcal{M}_{11}-\mathcal{M}_{22})\sin2\theta_0+2\mathcal{M}_{12}\cos2\theta_0\\
%(\mathcal{M}_{11}-\mathcal{M}_{22})\sin2\theta_0+2\mathcal{M}_{12}\cos2\theta_0,&(\mathcal{M}_{11}+\mathcal{M}_{22})-(\mathcal{M}_{11}-\mathcal{M}_{22})\cos2\theta_0+2\mathcal{M}_{12}\sin2\theta_0
%\end{pmatrix}\\
&=&\frac{\mathcal{M}_{11}+\mathcal{M}_{22}}{2}I_2+\frac{\mathcal{M}_{11}-\mathcal{M}_{22}}{2}\begin{pmatrix}\cos2\theta_0&\sin2\theta_0\\\sin2\theta_0&-\cos2\theta_0\end{pmatrix}+\mathcal{M}_{12}\begin{pmatrix}-\sin2\theta_0&\cos2\theta_0\\\cos2\theta_0&\sin2\theta_0\end{pmatrix},
\end{eqnarray*}with
\begin{eqnarray*}
\frac{\mathcal{M}_{11}+\mathcal{M}_{22}}{2}&=&\frac{1}{2}\Big(F^{22}_{\theta_0}M_6(\br)+(F^{11}_{\theta_0}-F^{22}_{\theta_0})M_7({\br})\Big),\\
\frac{\mathcal{M}_{11}-\mathcal{M}_{22}}{2}&=&\frac{1}{2}\Big(2(F^{11}_{\theta_0}-F^{22}_{\theta_0})M_8({\br})-(F^{11}_{\theta_0}-3F^{22}_{\theta_0})M_7(\br)-F^{22}_{\theta_0}M_6(\br)\Big),\\
\mathcal{M}_{12}&=&2 F_{\theta_0}^{12}(M_7(\br)-M_8(\br)).
\end{eqnarray*}

We also note that
\begin{eqnarray*}
\bet\bet^T&=&\frac{1}{2}I_2+\frac{1}{2}\begin{pmatrix}\cos2\theta_0&\sin2\theta_0\\ \sin2\theta_0&-\cos2\theta_0\end{pmatrix},\\
\bet^\perp\bet^T&=&\frac{1}{2}\begin{pmatrix}0&-1\\ 1&0\end{pmatrix}+\frac{1}{2}\begin{pmatrix}-\sin2\theta_0&\cos2\theta_0\\ \cos2\theta_0&\sin2\theta_0\end{pmatrix}.
\end{eqnarray*}
We get
$R_{\theta_0}\mathcal{M}(\br)R_{\theta_0}^T=\mathcal{M}_{22}I_2+(\mathcal{M}_{11}-\mathcal{M}_{22})\bet\bet^T-\mathcal{M}_{12}\begin{pmatrix}0&-1\\ 1&0\end{pmatrix}+2\mathcal{M}_{12}\bet^\perp\bet^T$, which implies that
\begin{eqnarray}
&&\int_{\R^2}\frac{\alpha_\eps(\eta)}{\vert \eta\vert^3}\eta^TF_x\eta\langle\psi(x),\eta\rangle\eta d\eta\\
&=&\eps e^{-\br^2}\Big(\mathcal{M}_{22}\psi(x)+(\mathcal{M}_{11}-\mathcal{M}_{22})\langle\psi(x),\bet\rangle\bet-\mathcal{M}_{12}\psi(x)^\perp+2\mathcal{M}_{12}\langle\psi(x),\bet\rangle\bet^\perp\Big).\nonumber
\end{eqnarray}
On the other hand, one has
\begin{eqnarray*}
&&\int_{\R^2}\frac{\alpha_\eps(\eta)}{\vert \eta\vert^3}\eta^TF_x\eta\langle\psi(x),\eta\rangle d\eta\\
&=&e^{-\br^2}\psi^T(x)R_{\theta_0}\int_0^\infty e^{-r^2}rdr\int_0^{2\pi}e^{2r\br\cos\theta}\begin{pmatrix}F^{22}_{\theta_0}\cos\theta+(F^{11}_{\theta_0}-F^{22}_{\theta_0})\cos^3\theta\\2F^{12}_{\theta_0}\cos\theta(1-\cos^2\theta) \end{pmatrix}d\theta.
\end{eqnarray*}\smallskip

The functions $M_9(\cdot)$ and $M_{10}(\cdot)$ were defined in Eqs. (\ref{eq:m9}) and (\ref{eq:m10}) respectively.
%\begin{eqnarray}
%M_9(z)&:=&\int_0^\infty e^{-r^2}rdr\int_0^{2\pi}e^{2rz\cos\theta}\cos\theta d\theta,\\
%M_{10}(z)&:=&\int_0^\infty e^{-r^2}rdr\int_0^{2\pi}e^{2rz\cos\theta}\cos^3\theta d\theta.
%\end{eqnarray}
Then, we have
$$
\int_{\R^2}\frac{\alpha_\eps(\eta)}{\vert \eta\vert^3}\eta^TF_x\eta\langle\psi(x),\eta\rangle d\eta\\
=e^{-\br^2}\psi^T(x)R_{\theta_0}\begin{pmatrix}F^{22}_{\theta_0}M_9(\br)+(F^{11}_{\theta_0}-F^{22}_{\theta_0})M_{10}(\br)\\2F^{12}_{\theta_0}(M_9(\br)-M_{10}(\br)) \end{pmatrix}.
$$
Since $\psi^T(x)R_{\theta_0}=(\langle\psi(x),\bet\rangle, ~\langle\psi(x),\bet^\perp\rangle)$, we obtain
\begin{eqnarray}
\int_{\R^2}\frac{\alpha_\eps(\eta)}{\vert \eta\vert^3}\eta^TF_x\eta\langle\psi(x),\eta\rangle d\eta
&=&e^{-\br^2}\Big(\big[F^{22}_{\theta_0}M_9(\br)+(F^{11}_{\theta_0}-F^{22}_{\theta_0})M_{10}(\br)\big]\langle\psi(x),\bet\rangle\\
&+&2F^{12}_{\theta_0}(M_9(\br)-M_{10}(\br))\langle\psi(x),\bet^\perp\rangle\Big) .\nonumber
\end{eqnarray}
\smallskip
One also gets
\begin{eqnarray*}
\int_{\R^2}\frac{\alpha_\eps(\eta)}{\vert \eta\vert}\eta^TF_x\eta d\eta
&=&\eps e^{-\br^2}\int_0^\infty e^{-r^2}r^2dr\int_0^{2\pi}e^{2r\br\cos\theta}(F_{\theta_0}^{11}\cos^2\theta+F_{\theta_0}^{22}\sin^2\theta)d\theta\\
&=&\eps e^{-\br^2}\int_0^\infty e^{-r^2}r^2dr\int_0^{2\pi}e^{2r\br\cos\theta}(F_{\theta}^{22}+(F_{\theta_0}^{11}-F_{\theta}^{22})\cos^2\theta)d\theta\\
&=&\eps e^{-\br^2}(F_{\theta_0}^{22}M_6(\br)+(F_{\theta_0}^{11}-F_{\theta_0}^{22})M_7(\br)).
\end{eqnarray*}\smallskip
Finally, one derives
\begin{eqnarray*}
&&\int_{\R^2}\frac{\alpha_\eps(\eta)}{\vert\eta\vert^3}(\eta^TF_x\eta) \eta d\eta\\
&=&e^{-\br^2}R_{\theta_0}\int_0^\infty e^{-r^2}rdr\int_0^{2\pi}e^{2r\br\cos\theta}\begin{pmatrix}F_{\theta_0}^{22}\cos\theta+(F_{\theta_0}^{11}-F_{\theta_0}^{22})\cos^3\theta \\2F^{12}_{\theta_0}(\cos\theta-\cos^3\theta)\end{pmatrix}d\theta\\
&=&e^{-\br^2}R_{\theta_0}\begin{pmatrix}F_{\theta_0}^{22}M_9(\br)+(F_{\theta_0}^{11}-F_{\theta_0}^{22})M_{10}(\br) \\2F^{12}_{\theta_0}(M_9(\br)-M_{10}(\br))\end{pmatrix}.
\end{eqnarray*}
Since $\bar{\eta}_0^TR_{\theta_0}=(1,0)$, we have
\begin{eqnarray}
\int_{\R^2}\frac{\alpha_\eps(\eta)}{\vert\eta\vert^3}(\eta^TF_x\eta) \langle\eta_0,\eta\rangle d\eta=\eps e^{-\br^2}\big(F_{\theta_0}^{22}M_9(\br)+(F_{\theta_0}^{11}-F_{\theta_0}^{22})M_{10}(\br)\big).
\end{eqnarray}
\smallskip

In summary, we get
\begin{eqnarray*}
a_1&=&\frac{1}{4\pi}\frac{e^{-\br^2}}{\eps}\Big(\mathcal{M}_{22}\psi(x)+(\mathcal{M}_{11}-\mathcal{M}_{22})\langle\psi(x),\bet\rangle\bet-\mathcal{M}_{12}\psi(x)^\perp+2\mathcal{M}_{12}\langle\psi(x),\bet\rangle\bet^\perp\nonumber\\
&-&\big[\big(F^{22}_{\theta_0}M_9(\br)+(F^{11}_{\theta_0}-F^{22}_{\theta_0})M_{10}(\br)\big)\langle\psi(x),\bet\rangle+2F^{12}_{\theta_0}(M_9(\br)-M_{10}(\br))\langle\psi(x),\bet^\perp\rangle\big]\bar{\eta}_0\nonumber\\
&+&(F_{\theta_0}^{22}(M_6(\br)-M_9(\br))+(F_{\theta_0}^{11}-F_{\theta_0}^{22})(M_7(\br)-M_{10}(\br))\psi(x)\Big)\nonumber\\
&=&\frac{1}{4\pi}\frac{e^{-\br^2}}{\eps}\Big(\nonumber\\
&&\big[2F^{22}_{\theta_0}M_6(\br)+(2F^{11}_{\theta_0}-3F^{22}_{\theta_0})M_7(\br)-(F^{11}_{\theta_0}-F^{22}_{\theta_0})(M_8(\br)+M_{10}(\br))-F^{22}_{\theta_0}M_9(\br)\big]\psi(x)\nonumber\\
&-&\big[\big(F^{22}_{\theta_0}M_9(\br)+(F^{11}_{\theta_0}-F^{22}_{\theta_0})(M_{10}(\br)-2M_8(\br))+(F^{11}_{\theta_0}-3F^{22}_{\theta_0})M_7(\br)\nonumber\\ 
&&+F^{22}_{\theta_0}M_6(\br)\big)\langle\psi(x),\bet\rangle
+2F^{12}_{\theta_0}(M_9(\br)-M_{10}(\br))\langle\psi(x),\bet^\perp\rangle\big]\bet\nonumber\\
&-&2 F_{\theta_0}^{12}(M_7(\br)-M_8(\br))\psi(x)^\perp+4 F_{\theta_0}^{12}(M_7(\br)-M_8(\br))\langle\psi(x),\bet\rangle\bet^\perp\Big).
\end{eqnarray*}\smallskip

Let us now compute $a_2$. Using the computations performed for the term $a_1$, one has
\begin{eqnarray*}
\int_{\R^2}\frac{\alpha_\eps(\eta)}{\vert \eta\vert^3}\eta^TF_x\eta d\eta
%&=&\frac{e^{-\br^2}}{\eps}\int_{0}^\infty e^{-r^2}dr\int_{0}^{2\pi}e^{2r\br\cos\theta}(\cos\theta,\sin\theta)F_{\theta_0}\begin{pmatrix}\cos\theta\\\sin\theta\end{pmatrix}d\theta\\
&=&\frac{e^{-\br^2}}{\eps}\int_{0}^\infty e^{-r^2}dr\int_{0}^{2\pi}e^{2r\br\cos\theta}(F_{\theta_0}^{11}\cos^2\theta+F_{\theta_0}^{22}(1-\cos^2\theta))d\theta\\
&=&\frac{e^{-\br^2}}{\eps}\Big[ F_{\theta_0}^{22}M_5^{A_1}(\br)+(F_{\theta_0}^{11}-F_{\theta_0}^{22})M_1^{A_1}(\br)\Big].
\end{eqnarray*}
The other contribution in $a_2$ is given by the following expression.
\begin{eqnarray*}
&&\int_{\R^2}\frac{\alpha_\eps(\eta)}{\vert \eta\vert^3}\langle\psi(x),\eta\rangle F_x\eta d\eta\\
&=&F_xR_{\theta_0}\frac{e^{-\br^2}}{\eps}\int_{0}^\infty e^{-r^2}dr\int_0^{2\pi}e^{2r\br\cos\theta}\begin{pmatrix}\cos^2\theta&\sin\theta\cos\theta\\\sin\theta\cos\theta&1-\cos^2\theta\end{pmatrix}d\theta~ R^T_{\theta_0}\psi(x)\\
&=&\frac{e^{-\br^2}}{\eps}F_xR_{\theta_0}\begin{pmatrix}M_1^{A_1}(\br)&0\\ 0&M_5^{A_1}(\br)-M_1^{A_1}(\br)\end{pmatrix}~ R^T_{\theta_0}\psi(x)\\
&=&\frac{e^{-\br^2}}{\eps}F_x\Big(\frac{1}{2}M_5^{A_1}(\br)I_2+(M_1^{A_1}(\br)-\frac{1}{2}M_5^{A_1}(\br))\begin{pmatrix}\cos2\theta_0&\sin 2\theta_0\\ \sin 2\theta_0&-\cos2\theta_0\end{pmatrix}\Big)~\psi(x)\\
&=&\frac{e^{-\br^2}}{\eps}F_x\Big((M_5^{A_1}(\br)-M_1^{A_1}(\br))I_2+(2M_1^{A_1}(\br)-M_5^{A_1}(\br))\bet\bet^T\Big)\psi(x).
\end{eqnarray*}Therefore, we have
\begin{eqnarray*}
a_2&=&-\frac{1}{4\pi}\frac{e^{-\br^2}}{\eps}\Big\{\Big[F_{\theta_0}^{22}M_5^{A_1}(\br)+(F_{\theta_0}^{11}-F_{\theta_0}^{22})M_1^{A_1}(\br)\Big]\psi(x)\\
&&\quad+F_x\Big((M_5^{A_1}(\br)-M_1^{A_1}(\br))\psi(x)+(2M_1^{A_1}(\br)-M_5^{A_1}(\br))\langle\psi(x),\bet\rangle\bet\Big)\Big\}.
\end{eqnarray*}

Finally, $a_3$ is computed as follows.
\begin{eqnarray*}
a_3&=&-\frac{1}{8\pi}\int_{\R^2}\frac{\alpha_\eps(\eta)\langle\psi(x),\eta/\vert\eta\vert\rangle}{\vert \eta\vert}\frac{\eta}{\vert\eta\vert}d\eta\\
&=&-\frac{1}{8\pi}\frac{e^{-\br^2}}{\eps}R_{\theta_0}\int_{0}^\infty e^{-r^2}dr\int_{0}^{2\pi}e^{2r\br\cos\theta}\begin{pmatrix}\cos^2\theta&\sin\theta\cos\theta\\\sin\theta\cos\theta&(1-\cos^2\theta)\end{pmatrix}d\theta R^T_{\theta_0}\psi(x)\\
&=&-\frac{1}{8\pi}\frac{e^{-\br^2}}{\eps}R_{\theta_0}\begin{pmatrix} M_1^{A_1}(\br) &0\\ 0&M_5^{A_1}(\br)-M_1^{A_1}(\br)\end{pmatrix}R^T_{\theta_0}\psi(x)\\
&=&-\frac{1}{8\pi}\frac{e^{-\br^2}}{\eps}\Big[ (M_5^{A_1}(\br)-M_1^{A_1}(\br))\psi(x)+(2M_1^{A_1}(\br)-M_5^{A_1}(\br))\langle\psi(x),\bet\rangle\bet\Big].
\end{eqnarray*}
This ends the proof of Lemma \ref{calcul-final}.

\subsection{Proof of Lemma \ref{final-cal}}\label{pf-final-cal}
Recall that
$$
\int_0^{2\pi}\cos^{2p}\theta d\theta=4 I_{2p},\ \int_0^\infty e^{-r^2}r^{2p}dr=\frac{1}{2}\Gamma (p+\frac{1}{2}).
$$
Then, one gets
\begin{eqnarray*}
M_6(z)&=&\int_0^\infty e^{-r^2}r^2dr\int_0^{2\pi}e^{2rz\cos\theta}d\theta=\sum_{k=0}^\infty\frac{2^k}{k!}\big[\int_0^\infty e^{-r^2}r^{k+2}dr\int_0^{2\pi}\cos^k\theta d\theta \big] z^k\\
&=&\sum_{p=0}^{\infty}\frac{2^{2p+1}}{(2p)!}\Gamma(p+\frac{3}{2})I_{2p} ~z^{2p},
\end{eqnarray*}
\begin{eqnarray*}
M_7(z)&=&\int_0^\infty e^{-r^2}r^2dr\int_0^{2\pi}e^{2rz\cos\theta}\cos^2\theta d\theta=\sum_{k=0}^\infty\frac{2^k}{k!}\big[\int_0^\infty e^{-r^2}r^{k+2}dr\int_0^{2\pi}\cos^{k+2}\theta d\theta \big] z^k\\
&=&\sum_{p=0}^{\infty}\frac{2^{2p+1}}{(2p)!}\Gamma(p+\frac{3}{2})I_{2p+2} ~z^{2p},
\end{eqnarray*}

\begin{eqnarray*}
M_8(z)&=&\int_0^\infty e^{-r^2}r^2dr\int_0^{2\pi}e^{2rz\cos\theta}\cos^4\theta d\theta=\sum_{k=0}^\infty\frac{2^k}{k!}\big[\int_0^\infty e^{-r^2}r^{k+2}dr\int_0^{2\pi}\cos^{k+4}\theta d\theta \big] z^k\\
&=&\sum_{p=0}^{\infty}\frac{2^{2p+1}}{(2p)!}\Gamma(p+\frac{3}{2})I_{2p+4} ~z^{2p},
\end{eqnarray*}

\begin{eqnarray*}
M_9(z)&=&\int_0^\infty e^{-r^2}rdr\int_0^{2\pi}e^{2rz\cos\theta}\cos\theta d\theta=\sum_{k=0}^\infty\frac{2^k}{k!}\big[\int_0^\infty e^{-r^2}r^{k+1}dr\int_0^{2\pi}\cos^{k+1}\theta d\theta \big] z^k\\
&=&\sum_{p=0}^{\infty}\frac{2^{2p+2}}{(2p+1)!}\Gamma(p+\frac{3}{2})I_{2p+2} ~z^{2p+1},
\end{eqnarray*}
and finally
\begin{eqnarray*}
M_{10}(z)&:=&\int_0^\infty e^{-r^2}rdr\int_0^{2\pi}e^{2rz\cos\theta}\cos^3\theta d\theta=\sum_{k=0}^\infty\frac{2^k}{k!}\big[\int_0^\infty e^{-r^2}r^{k+1}dr\int_0^{2\pi}\cos^{k+3}\theta d\theta \big] z^k\\
&=&\sum_{p=0}^{\infty}\frac{2^{2p+2}}{(2p+1)!}\Gamma(p+\frac{3}{2})I_{2p+4} ~z^{2p+1}.
\end{eqnarray*}

Therefore, we obtain
\begin{eqnarray*}
&&2M_7(z)-M_8(z)-M_{10}(z)-M_1^{A_1}(z)\\
&=&\sum_{p=0}^\infty \frac{2^{2p+1}}{(2p)!}\Gamma(p+\frac{1}{2})I_{2p+2}\frac{2p^2+4p-\frac{3}{2}}{2p+3}z^{2p}-\sum_{p=0}^{\infty}\frac{2^{2p+2}}{(2p+1)!}\Gamma(p+\frac{3}{2})I_{2p+4} ~z^{2p+1},
\end{eqnarray*}

\begin{eqnarray*}
&&2M_6(z)-3M_7(z)+M_8(z)+M_{10}(z)-M_9(z)-2M_5^{A_1}(z)+2M_1^{A_1}(z)\\
&=&\sum_{p=0}^\infty \frac{2^{2p+1}}{(2p)!}\Gamma(p+\frac{1}{2})I_{2p}\frac{6p^2+16p+\frac{7}{2}}{(2p+2)(2p+4)}z^{2p}-\sum_{p=0}^\infty\frac{2^{2p+2}}{(2p+1)!}\Gamma(p+\frac{3}{2})I_{2p+2}\frac{1}{2p+4}z^{2p+1}.
\end{eqnarray*}
%and we conclude that 
%\begin{eqnarray*}
%M_1^{A_1}(\br)-M_5^{A_1}(\br)&=&-\sum_{p=0}^\infty\frac{2^{2p+1}}{(2p)!}\Gamma(p+\frac{1}{2})I_{2p}\frac{1}{2p+2}z^{2p}.
%\end{eqnarray*}

\bibliographystyle{plain}
\bibliography{Foias-Saut-Final29012013}
\end{document}